\documentclass[11pt]{amsart}

\usepackage{amsmath, amsthm, amssymb}
\usepackage{amsfonts}
\usepackage{esint}
\usepackage[ansinew]{inputenc}
\usepackage[dvips]{epsfig}
\usepackage{graphicx}
\usepackage[english]{babel}
\usepackage{pdfsync}
\usepackage{hyperref}
\usepackage{color}
\usepackage{array}
\usepackage{appendix}

\usepackage[
  hmarginratio={1:1},     
  vmarginratio={1:1},     
  textwidth=17cm,        
  textheight=21cm,
  heightrounded,          
]{geometry}

\newtheorem{thm}{Theorem}[section]
\newtheorem{cor}[thm]{Corollary}
\newtheorem{lem}[thm]{Lemma}
\newtheorem{prop}[thm]{Proposition}
\newtheorem{conj}[thm]{Conjecture}
\newtheorem{defn}[thm]{Definition}
\newtheorem{rem}[thm]{Remark}

\newtheorem*{thm*}{Theorem}
\newtheorem*{defn*}{Definition}
\numberwithin{equation}{section}

\newcommand{\dx}{\,{\rm d}x}
\newcommand{\dsi}{\,{\rm d}\sigma}

\newcommand{\rd}{{\rm d}}

\newcommand{\suubset}{\subset\joinrel\subset}

\newcommand{\NN}{\mathbb{N}}

\newcommand{\RR}{\mathbb{R}}
\newcommand{\Sf}{\mathbb{S}}

\newcommand\EE{{\mathcal{E}}}

\newcommand\WW{{\mathcal{W}}}

\newcommand{\vep}{\varepsilon}
\newcommand{\ep}{\epsilon}
\newcommand{\vth}{\vartheta}
\newcommand{\dd}{\delta}

\newcommand{\lm}{\lambda}

\def\dist{\mathrm{dist}} 

\DeclareMathOperator*{\osc}{osc}

\begin{document}

\title[]{Interface regularity for semilinear one-phase problems}

\author[A. Audrito]{Alessandro Audrito}\thanks{}
\address{Alessandro Audrito \newline \indent
Department of Mathematics, ETH Zurich, \newline \indent
R\"amistrasse 101, 8092 Zurich, Switzerland
 }
\email{alessandro.audrito@math.ethz.ch}

\author[J. Serra]{Joaquim Serra}\thanks{}
\address{Joaquim Serra \newline \indent
Department of Mathematics, ETH Zurich, \newline \indent
R\"amistrasse 101, 8092 Zurich, Switzerland}
\email{joaquim.serra@math.ethz.ch}

%
%
%
%
%
%
%
%
%
%
%
\date{\today} 


\thanks{This project has received funding from the European Union's Horizon 2020 research and innovation programme under the Marie Sk{\l}odowska--Curie grant agreement 892017 (LNLFB-Problems). JS has been supported by the European Research Council (ERC) under the grant agreement 948029.}

\begin{abstract}
We study critical points of  a one-parameter family of functionals arising  in combustion models. The problems we consider converge, for infinitesimal values of the parameter,  to Bernoulli's free boundary problem, also known as one-phase problem.
We prove a $C^{1,\alpha}$ estimates for the ``interfaces'' (level sets separating the {\em burnt} and {\em unburnt} regions).
As a byproduct, we obtain the one-dimensional symmetry of minimizers in the whole $\RR^N$, for $N \le 4$, answering positively a conjecture of Fern\'andez-Real and Ros-Oton.

Our results are to Bernoulli's free boundary problem what Savin's results for the Allen-Cahn equation are to minimal surfaces.
\end{abstract}

\maketitle



%
%
%
%
\section{Introduction}

In this paper we study critical points of the following (non-convex) energy functional
\begin{equation}\label{eq:Energy}
\EE_\vep(u,\Omega) := \int_\Omega |\nabla u|^2 + \Phi_\vep(u) \dx,
\end{equation}
where $\vep \in (0,1]$ is a parameter, $\Omega\subset \mathbb R^N$ some open domain,  and 
\begin{equation}\label{PHI000}
\Phi_\vep(t) := \Phi(t/\vep)
\end{equation}
\begin{equation}\label{PHI111}
 \Phi(t) := \begin{cases}  \int_0^t \beta(\tau) \rd\tau \quad &  \mbox{for }t\ge0  \\
  0 &  \mbox{for }t<0,
  \end{cases}
\end{equation}
for some given function $\beta\in C^\infty_c \big( [0,+\infty) \big)$ satifying
\begin{equation}\label{PHI222}
\beta\ge 0,  \quad \beta(0)=0, \quad \beta'(0)>0, \quad{\textstyle \int_0^\infty \beta =1}.
\end{equation}
When $\vep=1$, $\EE_1(u,\Omega)$ will be sometimes denoted by $\EE(u,\Omega)$. The assumption  $\beta'(0)>0$ is made for simplicity, but in all our main results it could be actually replaced by $\liminf_{\tau \downarrow 0} \beta(\tau) \tau^{-p} >0$  for some $p\in (1,\infty)$ ---see Remark \ref{rem:otherp}.

These type of functional arises in combustion models (e.g.  flame propagation)  \cite{CV95, BL82, DanielliEtAl2003:art, We03, PY07}, and were studied in detail in the book of Caffarelli and Salsa \cite{CafSal2005:book}.

\subsection*{Connection to the one-phase problem} Due to the assumptions on $\Phi$, as $\vep \downarrow 0$, the energy $\mathcal E_\vep$ formally converges towards
\begin{equation}\label{eq:Energy1Phase}
\mathcal{E}_0(u,\Omega) := \int_{\Omega} |\nabla u|^2 + \chi_{\{u>0\}} \dx.
\end{equation}
Critical points of $\mathcal E_0$ are solutions to Bernoulli's  (or one-phase) free boundary problem:
\begin{equation}\label{eq:Energy1Phase00}
u\ge 0,\quad \Delta u = 0 \quad \mbox{in } \{u>0\}, \quad \partial_n u=1 \quad \mbox{on } \partial\{u>0\},
\end{equation}
where $n$ is the inwards unit normal to $\partial\{u>0\}$. The regularity of solutions and free boundaries for minimizers of $\mathcal E_0$ has been extensively studied in \cite{Caffa1987:art, Caffa198:art, Caffa1988:art,CafJerKen2004:art,DeSilva2011:art,DeSilvaJerison2009:art, JerisonSavin2009:art} (see also the treatment given in \cite{Velichkov2019:art}).
The convergence of $\mathcal E_\vep$ towards  $\mathcal E_0$ as  $\vep\downarrow0$ is not merely formal: as proven in \cite[Theorem 1.15]{CafSal2005:book},  sequences of minimizers $u_ {\ep_k}$ of $\mathcal E_{\vep_k}$ converge as $\vep_k\downarrow 0$ (and up to subsequences) towards minimizers of the functional $\mathcal E_0$.

\subsection*{A conjecture ``alla De Giorgi''}  By the results in \cite{CafJerKen2004:art,JerisonSavin2009:art}, it is know that every minimizer $u_0: \mathbb R^N\to [0,\infty)$ of $\mathcal E_0$ in $\RR^N$  must have one-dimensional symmetry in dimensions $N\le 4$,  while this fails for $N \geq 7$ (see \cite{DeSilvaJerison2009:art}).  On the other hand if $u: \mathbb R^N\to \mathbb R_+$ is a minimizer of $\mathcal E=\mathcal E_1$  in $\RR^N$, then the {\em blow-down sequence} $u_{\vep_k}(x): =  \vep_k u(x/\vep_k)$, $\vep_k \downarrow 0$, are minimizers of  $\mathcal E_{\vep_k}$  in $\RR^N$. Since by \cite[Theorem 1.15]{CafSal2005:book} $u_{\vep_k}$ converges (up to subsequence and uniformly in every compact subset of $\RR^N$) to some entire minimizer of $\mathcal E_0$,  every blow-down of $u$ must  one-dimensional if $N\le 4$.  By analogy with De Giorgi's conjecture   for the Allen-Cahn equation (see for instance \cite{DeGiorgi1978:art,Savin2009:art}), Fern\'andez-Real and Ros-Oton raised the following
\begin{conj}[\cite{FerRealRosOton2019:art}]\label{conj:XX}
Let $N \leq 4$ and $u: \RR^N\to \mathbb \RR_+$ be a minimizer of  $\mathcal E$ in $\RR^N$  (see Definition \ref{def:LocMinimizer} below). Then, $u$ must be of the form
\begin{equation}\label{wtioewhot}
u(x) =  v(\nu\cdot x - l) \quad \mbox{where } v(t) = \psi^{-1}(t) \quad \mbox{for} \quad \psi(z): =   \int_{1} ^{z} \frac{\rd \zeta}{\sqrt{\Phi(\zeta)}},
\end{equation}
for some $\nu\in \mathbb S^{N-1}$ and $l \in \RR$.
\end{conj}
The results in this paper answer positively this conjecture.

\subsection*{Minimizers and critical points}  We define next minimizer and critical point of $\mathcal E_\vep$.
\begin{defn}\label{def:LocMinimizer}
Let $\Omega\subseteq \RR^N$ be some open domain ad let $\vep >0$. We say that $u_\vep \in H^1_{loc}(\Omega)$ is a {\em minimizer} of \eqref{eq:Energy} in $\Omega$ if for every $V \suubset \Omega$ and for every $\xi \in H^1_0(V)$  we have
\[
\EE_\vep(u_\vep,V) \leq \EE_\vep(u_\vep+\xi,V).
\]

\end{defn}
\begin{defn}\label{def:criticalpoint}
Let $\Omega$, $N$ and $\vep >0$, as in Definition \ref{def:LocMinimizer}. We say that $u_\vep \in H^1_{loc}(\Omega)$ is a {\em critical point} of \eqref{eq:Energy} in $\Omega$ if for every $V \suubset \Omega$ and for every $\xi \in H^1_0(V)$  we have
\[
\frac{d}{dt} \bigg|_{t=0} \EE_\vep(u_\vep + t\xi,V) =0 \qquad \Leftrightarrow \qquad \int _V 2\nabla u_\vep \cdot \nabla \xi  + \Phi_\vep'(u_\vep)\xi \, dx =0.
\]
\end{defn}
Notice that (after integration by parts) any critical point $u_\vep$ of satisfies
\begin{equation}\label{eq:EulerEq}
\Delta u_\vep = \tfrac{1}{2} \Phi_\vep'(u_\vep),
\end{equation}
in the weak sense. Since $\Phi'$ is smooth and bounded,  by elliptic regularity and the standard  ``bootstrap argument'' for semilinear equations,   any critical point is locally smooth (with estimates which degenerate in principle as $\vep\downarrow 0$) and  hence satisfies \eqref{eq:EulerEq} in the classical sense.

\subsection*{New results} We describe next the main results of the paper. Our main contribution is the following rigidity results for  critical points of $\mathcal E$ in $\RR^N$ which are ``asymptotic'' to $(\nu\cdot x)_+$ at very large scales. In its statement (and in the rest of the paper) we use the following convenient notation for inclusion of sets: we write ``$X\subset Y$ in $Z$'' when $X\cap Z\subset Y\cap Z$.

\begin{thm}\label{thm:main}
Let $\Phi$ be as in \eqref{PHI111}-\eqref{PHI222}.  There exist constants $\vartheta_1$ and $\vartheta_2$ depending only on $\Phi$  such that the following holds. Let  $u: \RR^N \to \RR_+$ be a critical point of $\mathcal E$ in  $\RR^N$.

Assume there exist $\nu \in \Sf^{N-1}$ and sequences $R_k \uparrow \infty$ and $\delta_k\downarrow 0$ such that
\begin{equation}\label{eq:1DSymAss10}
|u - (\nu \cdot x )_+| \leq \delta_{k} R_k \quad \text{ in } B_{R_k},
\end{equation}
and
\begin{equation}\label{eq:1DSymAss20}
\{ \nu \cdot x \leq -\delta_k R_k \} \subset \{ u \leq \vth_1 \} \subset \{ u \leq \vth_2 \} \subset  \{ \nu \cdot x \leq  \delta_k R_k \} \quad \text{ in } B_{R_k}.
\end{equation}

Then $u$ is of the form \eqref{wtioewhot}.
\end{thm}

On the other hand, building on the  results of \cite[Chapter 1]{CafSal2005:book} (and introducing new ideas) we establish the following
\begin{prop}\label{prop:main}
Let $\Phi$  be as in \eqref{PHI111}-\eqref{PHI222} and let $\vartheta_1$ and $\vartheta_2$ be the constants from Theorem \ref{thm:main}. Let  $u: \RR^N \to \RR_+$ be a minimizer of $\mathcal E$ in $\RR^N$ which is not identically $0$.
Then, for every sequence  $R_k \uparrow \infty$ there exists a subsequence $R_{k_\ell}$, a 1-homogeneous minimizer $u_0$ of $\mathcal E_0$ in $\RR^N$ --- also not identically zero--- and a sequence $\delta_\ell\downarrow 0$ such that
\begin{equation}\label{eq:1DSymAss1XX}
|u - u_0| \leq \delta_{\ell} R_{k_\ell} \quad \text{ in } B_{R_{k_\ell}},
\end{equation}
and
\begin{equation}\label{eq:1DSymAss2XX}
\{ x: {\rm dist} (x,  \{u_0>0\}) \ge \delta_\ell R_{k_\ell} \} \subset \{ u \leq \vth_1 \} \subset \{ u \leq \vth_2 \} \subset  \{x: {\rm dist} (x,  \{u_0=0\}) \le \delta_\ell R_{k_\ell} \} \quad \text{ in } B_{R_k}.
\end{equation}
\end{prop}

Combining Theorem \ref{thm:main},  Proposition \ref{prop:main}, and using the classification results for 1-homogeneous minimizers of $\mathcal E_0$ of \cite{CafJerKen2004:art,JerisonSavin2009:art} we obtain
\begin{cor}\label{cor:main}
Conjecture \ref{conj:XX} holds true.
\end{cor}
%
%
%
%

%
%
%
%
\section{Overview of the proofs and organization of the paper}
The proof of Theorem \ref{thm:main} is split in several intermediate steps, some of them having independent interest. The main step (and our main contribution) is establishing an ``improvement of flatness'' result for critical points of $\mathcal E$ that we state below. Before that,  we need to introduce two positive constants $\vth_1$ and $\vth_2$, with $\vth_1<\vth_2$ and depending only on $\Phi$,  that will appear throughout the paper. 
Under our assumptions on $\Phi$ ---see \eqref{PHI000}-\eqref{PHI222}--- we can choose positive constants $\vth_1$, $\vth_2$, and $c_1$, such that the following holds:
\begin{equation}\label{eq:AssumptionsPhi}
\begin{cases}
\Phi = 0 \quad \text{in } (-\infty,0], \quad \Phi = 1 \quad \text{in } [\vth_2,\infty), \\
\tfrac{1}{c_1} u \leq \tfrac12\Phi'(u) \leq c_1 u, \quad  \forall u \in [0,\vartheta_1].
\end{cases}
\end{equation}

We can now give the statement of our ``improvement of flatness'' result.
\begin{thm}\label{lem:ImpFlat3}
Let $\Phi$  be as in \eqref{PHI111}-\eqref{PHI222} and let $\vartheta_1$ and $\vartheta_2$ as in \eqref{eq:AssumptionsPhi}. Fix $\gamma\in( 0,1)$. There exist constants $\delta_0>0$ and $\varrho_0\in(0,1/4)$ depending only on $N$ and $\Phi$, such that the following holds.
For every $R > 0$, every $\delta \in (0,\delta_0]$, every $\varepsilon/R \in (0, \delta^2)$, and every critical point $u_\varepsilon$ of \eqref{eq:Energy} in $B_R\subset \RR^N$ satisfying
\begin{equation}\label{eq:0FBpoint}
u_\vep(0) \in [\vartheta_1\vep, \vartheta_2\vep]
\end{equation}
and
\begin{equation}\label{eq:FlatnessScaleR}
\begin{array}{rlll}
{} &\hspace{-5pt} u_\vep(x) - x_N \leq \dd R \quad &\mbox{in }& B_R \cap \{u_\vep \geq \vartheta_1\vep\}\\
-\dd R \leq&\hspace{-5pt}  u_\vep(x) - x_N \quad   &\mbox{in }&B_R, 
\end{array}
\end{equation}
there exists $\nu \in \Sf^{N-1}$ such that
\begin{equation}\label{eq:FlatnessScaleRho0R}
\begin{array}{rlll}
{} &\hspace{-5pt} u_\vep(x) - \nu\cdot x \leq \dd \varrho_0^{1+\gamma}R \quad &\mbox{in }& B_{\varrho_0 R} \cap \{u_\vep \geq \vartheta_1\vep\}\\
-\dd  \varrho_0^{1+\gamma}R \leq&\hspace{-5pt}  u_\vep(x) -  \nu\cdot x \quad   &\mbox{in }&B_{\varrho_0 R} 
\end{array}
\end{equation}
with
\begin{equation}\label{eq:ImpFlat3ImpFlat2}
|\nu - e_N| \leq  \sqrt 2 N \delta.
\end{equation}
\end{thm}

Let us discuss some key aspects in the statement of Theorem \ref{lem:ImpFlat3}:

\vspace{3pt}

Assumption \eqref{eq:0FBpoint} must be though as the analogue of  asking $0$ to be a free boundary point in the one-phase setting ($\vep=0^+$). 
Indeed,  on the one hand it follows from the definition of $\vth_2$ that $u_\vep$ is harmonic in $\{u_\vep >\vth_2\vep\}$. 
On the other hand, using the definition of $\vth_1$ we will show (cf. Lemma \ref{lem:ExpDecayu}) that  $u_\vep$ has  ``exponentially small size in $\vep$''  inside $\{u_\vep < \vth_1\vep\}$. Consequently, the ``fat hypersurface'' $\{\vth_1\vep < u_\vep < \vth_2\vep \}$ is really analogous the free boundary in the one-phase setting.

\vspace{3pt}

Assumption \eqref{eq:FlatnessScaleR} and conclusion \eqref{eq:FlatnessScaleRho0R} must be thought, respectively, as a  $\dd$-flatness property of $u_\vep$ at scale $R > 0$ and a $(\varrho_0^{\gamma}\dd)$-flatness property at scale $\varrho_0 R$. 
In our framework this turns out to the appropriate notion of $\delta$-{\em flatness}.
As it is customary, the flatness is a dimensionless parameter: Roughly speaking, it measures the ratio between  $\min_{e\in \mathbb S^{N-1}} {\rm dist} \big(\{\vth_1\vep < u_\vep < \vth_2\vep \}\cap B_R , \{e\cdot x=0\} \cap B_R\big)$ and  $R$. 
With respect to \cite{DeSilva2011:art},  we remark that in \eqref{eq:FlatnessScaleR}-\eqref{eq:FlatnessScaleRho0R}  inequality from above is not required to hold in $\{u_\vep>0\}$, but only in  $\{u_\vep \geq \vartheta_1\vep\}$ (otherwise the result would be empty since non-zero solutions to our semilinear PDE are everywhere positive!).

\vspace{3pt}

The conclusion of the theorem can be phrased as an ``improvement of flatness'':   if $u_\vep$ is $\dd$-flat at scale $R$ (for small values of $\vep$ and $\dd$), it is $(\varrho_0^{\gamma}\dd)$-flat at scale $\varrho_0R$.

\bigskip

We now say a few words about the proof of Theorem \ref{lem:ImpFlat3}. 
In some sense, this proof is  an ``interpolation'' of  the proofs of De Silva  in \cite{DeSilva2011:art} and Savin in \cite{Savin2009:art} (although an additional ``sliding method'' step in the spirit of Berestycki, Caffarelli, Nirenberg \cite{BerCaffaNiren1997:art} is also needed, by similar reasons as in \cite{DipierroEtAl2020:art}). 
Indeed, our goal is to generalize the proof of De Silva \cite{DeSilva2011:art} for the one-phase free boundary problem to the setting of critical points of  $\mathcal{E}_\vep(\cdot,\RR^N)$). But since we need to go from a scaling invariant problem to a non-scaling invariant semilinear problem, there is an obvious  analogy with what Savin did in his celebrated paper \cite{Savin2009:art}. In this work Savin proved a version of the De Giorgi's improvement of flatness for area-minimizing hypersurfaces (a scaling invariant problem), in the framework of energy minimizers of the Allen-Cahn equation (a semilinear PDE).

Both Savin's and De Silva's proofs follow a ``small perturbations'' approach (linearization around flat solutions). 
In both cases --- although for different reasons--- the {\em deviation} between an almost-flat solution and the flat one which best approximates it, is found to be an ``almost-harmonic'' function. Further, in both proofs, the quadratic decay of  harmonic function towards their linear Taylor expansion is somehow transferred to the almost-flat solutions in order to obtain the improvement of flatness property. 
To accomplish this, both proofs use a delicate compactness argument, where deviations converge in $C^0$ towards some limit function which is proved to be harmonic in the viscosity sense. 
This type of argument requires some $C^\alpha$ estimate, or  {\em  improvement of oscillation estimate}, which guarantees the compactness  in $C^0$ (via Arzel\`a-Ascoli) of the sequences the deviations.

In our proof we also need such improvement of oscillation estimate, and finding an appropriate statement we could use in our setting turned out to be not easy at all!
Indeed, in a first ``naive approximation'', one could try  to  extend De Silva's improvement of oscillation (\cite[Theorem 3.1]{DeSilva2011:art})  to the semilinear setting as follows:  
\begin{lem}\label{lem:ImpFlat11}
Let  $v_\vep$ be the solution of \eqref{eq:EulerEq} in $\RR$ satisfying $v_\vep(0) =\vth_1\vep$  (see Lemma \ref{lem:1DSolution}, part (i)). 
There exist $\delta_0,c_0 \in (0,1)$ and $\theta_0 \in (0,1)$ depending only on $N$, $\Phi$ such that the following holds.  For every $R > 0$, every $\delta \in (0,\delta_0)$, every $a \in \mathbb{R}$ and $b \leq 0$ such that $a + |b| = \delta R$, every $\varepsilon/R \in (0,c_0 \delta)$ and every critical point $u_\varepsilon$ of \eqref{eq:Energy} in $B_R$ satisfying
\begin{equation}\label{eq:FlatCondSolFunIntro}
v_\varepsilon(x_N - a) \leq u_\varepsilon(x) \leq v_\varepsilon(x_N - b) \quad \text{ in } B_R, \\
\end{equation}
there exist $a' \in \mathbb{R}$, $b'\leq 0$ such that
\[
\begin{aligned}
&v_\varepsilon(x_N - a') \leq u_\varepsilon(x) \leq v_\varepsilon(x_N - b') \quad \text{ in } B_{R/4}, \\
&b \leq b' \leq a' \leq a,  \\
&a' + |b'| \leq \theta_0(a + |b|).
\end{aligned}
\]
\end{lem}

Lemma \ref{lem:ImpFlat11} is true\footnote{By a small modification of the proof of Lemma \ref{lem:ImpFlat1}.}. Unfortunately, it seems useless: the reason is that we cannot exclude  the existence of minimizers  $\EE_\vep$ in $B_{2R}$ which are  $\frac{\delta}{100}$-close to $(x_N)_+$ ---with $\vep>0$ and $\vep/\delta$ arbitrarily small--- but failing to  satisfy \eqref{eq:FlatCondSolFunIntro}.

Lemma \ref{lem:ImpFlat1}, where the $\dd$-shifts of $v_\varepsilon$ are replaced by  $\dd$-shifts of two suitable $1D$ super and subsolutions, is the right replacement to the previous naive statement. We construct these useful super and subsolutions in Lemma \ref{lem:1DSolution} part (ii) and (iii). Since they play a very important role in the paper, we devote the entire Section \ref{Sec:ODEandBarriers} to the classification of $1D$ (super- and sub-) solutions and the study of their properties. We do not give yet the statement of Lemma \ref{lem:ImpFlat1} because such preliminaries are needed.

Let us remark that this notion of $\dd$-flatness consisting in ``being trapped'' between  $\dd$-shifts of $1D$ super and subsolutions is essentially equivalent to the notion
\eqref{eq:FlatnessScaleR} when $\vep\in (0, \delta^2)$ ---this is actually the reason behind this nonlinear relation between $\vep$ and $\delta$ in the statement of Theorem \ref{lem:ImpFlat3}.  Definition \ref{def:Fla1Flat2} and Lemma \ref{lem:ImpFlat2tris} establish this essential equivalence, when $\vep\in (0, \delta^2)$,  of the these two notions of flatness which are used throughout the paper.

Last, but not least, in order to prove Theorem \ref{thm:main} we need to be able to apply our new improvement of flatness result (Theorem \eqref{lem:ImpFlat3}) to $u_\vep : = \vep u(\, \cdot\, /\vep)$ where $u$ is a minimizer of $\EE_1$ in $\RR^N$, $N\le 4$. 
To do so, first we need to show that the assumption \eqref{eq:FlatnessScaleR} will be satisfied ---for  some $\delta= \delta_0$  and $R=1$--- when $\vep$ is taken sufficiently small. 
This part essentially combines previous results in \cite{CafJerKen2004:art,JerisonSavin2009:art} and \cite{CafSal2005:book} (altough some improvements are needed)
and it is contained in Section \ref{Sec:NonDeg}.
However there is an important difference with respect to \cite{Savin2009:art}  that is related to our assumption $\vep/R< \delta^2$ in Theorem \ref{lem:ImpFlat3}.  
Indeed, in contrast with the Allen-Cahn setting (where $\vep$ and $\dd$ are comparable and the analogue of Theorem \ref{thm:main} is a corollary of the improvement of flatness), in our setting  Theorem \ref{thm:main} does not follow as a direct consequence of Theorem \ref{lem:ImpFlat3}.
The reason is the following: suppose you want to apply Theorem \ref{lem:ImpFlat3} iteratively (in balls of radius $R\varrho_0^{-i}$) to an entire minimizer  $u$ of $\EE_1$, starting from a huge ball $B_R$ (for which  $u$ is $\delta_0$-flat).  Then, at a mesoscale $1\ll R'\ll R$ the flatness will have improved to $\delta =(R'/R)^\gamma \delta_0$. 
So,  if we want to continue applying  Theorem \ref{lem:ImpFlat3} to $u$  in $B_{R'}$,  we must check that $1/R'<(R'/R)^{2\gamma} \delta_0^2$ (since $\vep =1$) and hence, we will always reach a critical mesoscale $R' = CR^{\frac{2\gamma}{1+2\gamma}}$ for which we cannot continue iterating.
To solve this,  we need an additional ``sliding method'' step  in the spirit of  Berestycki, Caffarelli, Nirenberg \cite{BerCaffaNiren1997:art}. This last step follows the ideas of  \cite{DipierroEtAl2020:art} and is done in Section \ref{Sec:AsymFlatImplies1D}.

\begin{rem}\label{rem:otherp}
We assume $\beta'(0) >0$  for simplicity, although this assumption  is not really necessary.
Indeed, our same proofs gives almost identical results  if the assumption $\beta'(0) >0$ is relaxed to $\liminf_{t\downarrow 0}\beta(t) t^{-p} >0$, ,  for some $p > 1$. 

More precisely, Theorem \ref{lem:ImpFlat3} can be proved under this more general condition, up to assuming $\vep/R < \dd^q$ (instead of $\vep/R < \dd^2$ ), for some suitable $q = q(p)>2$.
The reason for this change  is  the following:  while $\beta'(0)>0$  implies the exponential decay increasing $1D$ solutions at $-\infty$,  $\beta(t) \geq t^p$ gives a slower power-like decay. Accordingly, the properties of 1D solutions like \eqref{eq:1DSignChangeSolProp} and \eqref{eq:1DEvenSolProp2} change to similar ones where powers replace logarithms.
Up to this changes, all of our statements and proofs are still valid ---with minor modifications--- in this more general framework. 
The most important modifications are localized in Section \ref{Sec:ODEandBarriers} and only propagate to rest of the paper thought  Lemma \ref{lem:ImpFlat2tris},  where the size of the error is not $\sqrt{\vep/R}$ but  $(\vep/R)^{1/q}$ (for some $q>2$). This is the reason why we need  to assume $\vep/R < \dd^q$ instead of $\vep/R <\dd^2$ in Theorem \ref{lem:ImpFlat3}.
By the rest, all the proofs remain essentially the same.
\end{rem}

%

%
%
%
%
\section{ODEs analysis and barriers}\label{Sec:ODEandBarriers}
In this section we consider the family of second order ODEs
\begin{equation}\label{eq:1DEqn}
\ddot u_\varepsilon = \tfrac{1}{2}\Phi_\varepsilon'(u_\varepsilon) \quad \text{ in } \RR,
\end{equation}
and we provide a classification of its solutions, for every $\varepsilon \in (0,1)$ fixed. With respect to \cite[Section 2.3]{FerRealRosOton2019:art}, our ODEs analysis shows finer properties of global solutions such as \eqref{eq:1DSignChangeSolProp}, \eqref{eq:1DEvenSolProp} and \eqref{eq:1DEvenSolProp2}, which will be needed later in the proofs our main theorems.
\begin{lem}\label{lem:1DSolution} (1D global solutions)
Fix $\varepsilon \in (0,1)$ and let $\Phi$ be as in \eqref{eq:AssumptionsPhi}. Then:

\smallskip

\noindent (i) Equation \eqref{eq:1DEqn} has a unique solution $v_\varepsilon$ with
\[
v_\varepsilon(0) = \vartheta_1\varepsilon, \qquad \lim_{x\to+\infty} \dot v_\varepsilon(x) = 1,
\]
which is implicitly given by 
\[
\int_{\vartheta_1\vep}^{v_\vep(x)} \frac{\rd w} {\sqrt {\Phi_\vep(w)}} = x.
\]
This solution $v_\vep$  is smooth, positive, increasing, convex,  and satisfies  $v_\varepsilon(x) \to 0$ as $x \to -\infty$.

\smallskip

\noindent (ii) For every $t > 0$, equation \eqref{eq:1DEqn} has a unique solution $v_\varepsilon^t$ with
\[
v_\varepsilon^t(0) = \vartheta_1\varepsilon, \qquad \lim_{x\to+\infty} \dot v_\varepsilon^t(x) = 1+t. 
\]
Moreover, $v_\varepsilon^t$ is of class $C^{2}$, increasing, convex, and satisfies $v_\varepsilon^t(x) \to -\infty$, $\dot v_\varepsilon^t(x) \to \sqrt{2t + t^2}$ as $x \to -\infty$. Also, if $x_\varepsilon^t$ is denotes the unique  root of $v_\ep^t$ ---i.e. the point where $v_\varepsilon^t(x_\varepsilon^t) = 0$---, then
\begin{equation}\label{eq:1DSignChangeSolProp}
x_\varepsilon^t \geq -  \vep{\sqrt{2c_1}} \log\bigg(1+\frac{\vartheta_1}{t}\bigg),
\end{equation}
where $c_1 > 0$ is the constant in \eqref{eq:AssumptionsPhi}.
\smallskip

\noindent (iii) For any $\tau \in (-1,0)$, equation \eqref{eq:1DEqn} has a unique solution $v_\varepsilon^\tau$ with
\[
v_\varepsilon^\tau(0) = \vartheta_1\varepsilon, \qquad \lim_{x\to+\infty} \dot v_\varepsilon^\tau(x) = 1-|\tau|.
\]
Moreover, $v_\varepsilon^\tau$ is smooth, positive, and satisfies $v_\varepsilon^\tau(x) \to +\infty$, $\dot v_\varepsilon^\tau(x) \to -1 + |\tau|$ as $x \to -\infty$. Also, $v_\varepsilon^\tau$ has a unique point of minimum $y_\varepsilon^\tau$  satisfying
\begin{equation}\label{eq:1DEvenSolProp}
\sqrt{\tfrac{|\tau|}{c_1}} \, \varepsilon \leq v_\varepsilon^\tau(y_\varepsilon^\tau) \leq \sqrt{2c_1|\tau|} \, \varepsilon,
\end{equation}
and
\begin{equation}\label{eq:1DEvenSolProp2}
 y_\varepsilon^\tau \geq - \vep \sqrt{2c_1}\bigg(2 + \log \frac{\vartheta_1}{\sqrt{2|\tau| /c_1}} \bigg),
\end{equation}
where $c_1 > 0$ is the constant in \eqref{eq:AssumptionsPhi}.
\end{lem}
\begin{proof}
After scaling, let us assume $\varepsilon = 1$ and set $u = u_\varepsilon$, $v = v_\varepsilon$, $v^t= v_\varepsilon^t$ and $v^\tau = v_\varepsilon^\tau$.

Since $\Phi'$ is bounded, nonnegative and continuous, a local $C^2$ solution $u = u(x)$ to \eqref{eq:1DEqn} with $(u(0),\dot{u}(0)) = (\vartheta_1,\dot{u}_0)$ exists and it is convex on its maximal interval of definition $I$. Using the assumptions on $\Phi'$, it is not difficult to see that $I = \mathbb{R}$. Further, since \eqref{eq:1DEqn} is invariant under even reflections ($x \to -x$), we assume $\dot{u}_0 > 0$.

\smallskip

\emph{Step 1.}  Since $\dot u$ is nondecreasing the limits $\lim_{x\to \pm\infty} \dot u$ exist. Since $\dot{u}_0 > 0$ we see that $u(x) \to +\infty$ as $x \to +\infty$. Let us define 
\[
 \lim_{x \to +\infty} \dot{u}(x) =: A\in (0,+\infty).
\]
Hence, using that the Hamiltonian $x \to \dot u(x)^2 - \Phi(u(x))$ must be constant  (and $\Phi(u) =1$ for $u>0$ large enough) we obtain
\begin{equation}\label{eq:ODEEnergyCons}
\dot u(x)^2 - \Phi(u(x)) \equiv  A^2 - 1, \quad x \in \mathbb{R}.
\end{equation}

\smallskip

\emph{Step 2.}  Let us classify first monotone solutions:  assume $\lim_{x\to -\infty}  \dot{u} \geq 0$ and hence $\dot u>0$ in $\RR$. 
In this case (since $\Phi =\Phi'(u) =0$ for $u<0$) we  obtain that either 
\[
\lim_{x \to -\infty} u(x) =0 \quad \mbox{and} \quad \lim_{x \to -\infty} \dot u(x) =0
\]
or 
\[
\lim_{x \to -\infty} u(x) =-\infty  \quad \mbox{and} \quad \lim_{x \to -\infty} \dot u(x) =: B  \in (0,A).
\]
From  \eqref{eq:ODEEnergyCons}, we obtain that in the first case $A=1$, while in the second one we have
\[
A^2-B^2 = 1,
\]
and hence $A>1$.

Now in the first case integrating \eqref{eq:ODEEnergyCons}  ---with $A=1$--- we get
\begin{equation}\label{eq:1DImplicitv}
\int_{v(y)}^{v(x)} \frac{\rd w}{\sqrt{\Phi(w)}} = x - y,
\end{equation}
for every $y \leq x$ and so (i) follows. The solution in (ii), is obtained in the case $A = 1+t$, so  $B^2 = A^2 - 1 = 2t + t^2$. 
To complete (ii) we are left to show \eqref{eq:1DSignChangeSolProp}. 
Integrating \eqref{eq:ODEEnergyCons} between $x^t\le 0$ (the root of $v^t$) and $0$  (recall $v^t(0) = \vartheta_1$) and using \eqref{eq:AssumptionsPhi} we obtain
\[
\begin{aligned}
0 - x^t &= \int_{0}^{\vartheta_1} \frac{\rd w}{\sqrt{\Phi(w) + 2t + t^2}} \le  \int_{0}^{\vartheta_1} \frac{\rd w}{\sqrt{\frac{1}{2c_1} w^2 + 2t +t^2}} \\
& \leq \sqrt{2c_1} \int_{0}^{\vartheta_1} \frac{\rd w}{w+t}   = \sqrt{2c_1} \log\bigg(1+\frac{\vartheta_1}{t}\bigg).
\end{aligned}
\]

\smallskip

\emph{Step 3.} Let us consider now the case where $\dot u$ changes sign. If so, there is $x_0 \in \mathbb{R}$ such that $\dot{u}(x) \leq 0$ for $x \leq x_0$ and $\dot{u}(x) \geq 0$ for $x \geq x_0$ (by convexity of $u$). Since the equation is invariant under the reflection $x \mapsto2x_0 -x$, it follows that $u(x) = u(2x_0 -x)$ and thus $\lim_{x\to -\infty} \dot u = -A$.  Note that the solutions  $u = v^\tau$ described in  (iii) corresponds to the setting  $A = 1-|\tau|$, with $\tau \in (-1,0)$.

To show \eqref{eq:1DEvenSolProp}, we notice that if $y^\tau$ is the minimum point of $v^\tau$, then $\dot v^\tau(y^\tau) = 0$. Thus,  by \eqref{eq:ODEEnergyCons}, it follows
\begin{equation}\label{eq:1DEvenSolEnergyMin}
\Phi(v^\tau(y^\tau)) = 2|\tau| - \tau^2.
\end{equation}
Using again \eqref{eq:AssumptionsPhi} ---note that $v^\tau(y^\tau)< v^\tau(0) = \vartheta_1$---  we obtain 
\[
 \frac{|\tau|}{c_1} \le  \frac{2|\tau| - \tau^2}{c_1} \le \frac{1}{2} (v^\tau(y^\tau))^2 \le c_1(2|\tau| - \tau^2) \le  2c_1 |\tau| 
\]
and \eqref{eq:1DEvenSolProp} follows.

We are left to prove \eqref{eq:1DEvenSolProp2}. We use now \eqref{eq:AssumptionsPhi} to obtain that, for all $w\in (v^\tau(y^\tau), \vartheta_1)$,
\begin{equation}\label{whtiowhtoiw}
\begin{split}
\Phi(w) - 2|\tau| +\tau^2 &= \Phi(w) - \Phi(v^\tau(y^\tau)) = \int_{v^\tau(y^\tau)}^{w}  \Phi'(t) \,\rd t \ge \frac 1 {c_1}  \big[t^2 \big]_{v^\tau(y^\tau)}^w  = \frac 1{c_1} \big(w^2 - (v^\tau(y^\tau))^2\big) 
\\
&\ge \frac w{2c_1} \big(w - v^\tau(y^\tau)\big).
\end{split}
\end{equation}
Hence, integrating \eqref{eq:ODEEnergyCons} between $y^\tau$ and $0$  (recall $v^\tau(0) = \vartheta_1$) we obtain 
\[
\begin{aligned}
0 - y^\tau &= \int_{v^\tau(y^\tau)}^{\vartheta_1} \frac{\rd w}{\sqrt{\Phi(w) -2|\tau| + \tau^2}} \le \sqrt{2c_1}  \int_{v^\tau(y^\tau)}^{\vartheta_1} \frac{\rd w}{\sqrt{w} \sqrt{w - v^\tau(y^\tau)}}\\
&= \sqrt{2c_1}  \int_{1}^{\vartheta_1/v^\tau(y^\tau)} \frac{\rd \omega }{\sqrt{\omega} \sqrt{\omega-1}}
\\& \le \sqrt{2c_1}  \bigg( \int_1^2 \frac{\rd \omega }{\sqrt{\omega} \sqrt{\omega-1}}  +  \int_2^{\vartheta_1/v^\tau(y^\tau)}  \frac{\rd \omega }{\omega-1}\bigg)
\\& = \sqrt{2c_1} \bigg( \log(3+ 2\sqrt 2) + \log\bigg( \frac{\vartheta_1}{v^\tau(y^\tau)}\bigg)\bigg) \le \sqrt{2c_1}\bigg(2 + \log \frac{\vartheta_1}{\sqrt{2|\tau| /c_1}} \bigg).
\end{aligned}
\]
\end{proof}
In the following remark we introduce important one-dimensional super- and sub- solutions which will be used in the sequel.
\begin{rem}\label{rem:Choicettau} Lemma \ref{lem:1DSolution} gives a classification of solutions to \eqref{eq:1DEqn} in one dimension. The properties of such solutions are determined by their slopes  at infinity, 1, $1+t$ , or $1-|\tau|$,  where $t > 0$ and $\tau \in (-1,0)$ are parameters. As done in Lemma  \ref{lem:1DSolution} it is convenient to``center" these solutions so that their value at $x=0$ is $\vartheta_1\vep$.

In what follows, we will always take
\[
t = \varepsilon, \qquad \tau = -\varepsilon.
\]

Within this setting, we define
\[
w_\varepsilon^\varepsilon(x) :=
\begin{cases}
0 \quad &\text{ if } x \leq x_\vep^\vep \\
v_\varepsilon^\varepsilon(x)       \quad &\text{ if } x > x_\vep^\vep,
\end{cases}
\qquad
w_\varepsilon^{-\varepsilon}(x) :=
\begin{cases}
v_\varepsilon^{-\varepsilon}(x_\vep^{-\vep}) \quad &\text{ if } x \leq y_\vep^{-\vep} \\
v_\varepsilon^{-\varepsilon}(x)       \quad &\text{ if } x > y_\vep^{-\vep},
\end{cases}
\]
where $x_\vep^\vep$ and $y_\vep^{-\vep}$ are, respectively, the (unique) root of $v_\varepsilon^\varepsilon$ and the point of minimum of  $v_\varepsilon^{-\varepsilon}$.

It is immediate to see that $w_\vep^\vep$ and $w_\vep^{-\vep}$ are, respectively, a sub- and a super- solution of \eqref{eq:1DEqn}, both in the viscosity sense or in the weak sense.
\end{rem}
The next two lemmata are auxiliary results, which will be crucial in the proofs of our main theorems (see Section \ref{Sec:ImpFlat}).
In the statement of the next lemma we use the following standard notation ${\rm diam (X)} := \sup X-\inf X$ for subsets $X\subset \RR$.
\begin{lem}\label{lem:1DSolutionBoundFB} There exists $c > 1$ depending only on $\vartheta_1$, $\vartheta_2$ and $c_1 > 0$ as in \eqref{eq:AssumptionsPhi} such that 
\[
{\rm diam} \big(\{\vartheta_1\varepsilon \leq w_\varepsilon^\vep \leq \vartheta_2\varepsilon \big\}\big) \le c\vep, \qquad \forall \vep>0,
\] 
and 
\[
{\rm diam} \big(\{\vartheta_1\varepsilon \leq w_\varepsilon^{-\vep} \leq \vartheta_2\varepsilon \big\}\big) \le c\vep,\qquad \forall \vep \in {\textstyle  \Big(0,\frac{\vartheta_1^2}{8c_1}\Big)}.
\] 
\end{lem}
\begin{proof} By scaling,  we need to prove that  $w^\vep := w_1^\vep$ and $w^{-\vep} := w_1^{-\vep}$ satisfy 
\begin{itemize}
\item[(i)]  ${\rm diam} \big(\{\vartheta_1 \leq w^\vep \leq \vartheta_2 \big) \le c$;
\item[(ii)] ${\rm diam} \big(\{\vartheta_1 \leq w^{-\vep} \leq \vartheta_2 \big) \le c$.
\end{itemize}

\smallskip

To prove (i) wee notice that \eqref{eq:ODEEnergyCons} reads as $(\dot w^\vep)^2 = \Phi(w^\vep) + 2\vep + \vep^2$ in $\{w^\vep > 0\}$ and so, by \eqref{eq:AssumptionsPhi}, we find
\[
\tfrac{\vartheta_1^2}{c_1} \leq \Phi(\vartheta_1) \leq (\dot w^\vep)^2 \quad \text{ in } \{\vartheta_1 \leq w^\vep \leq \vartheta_2 \}.
\]
Integrating between $y$ and $x$, it follows
\[
w^\vep(x) - w^\vep(y) \geq \tfrac{\vartheta_1}{\sqrt{c_1}} (x-y).
\]
So, choosing $x$ such that $w^\vep(x) = \vartheta_2$, $y=0$ and recalling that $w^\vep(0) = \vartheta_1$, we find $\tfrac{\vartheta_1}{\sqrt{c_1}} x \leq \vartheta_2 - w^\vep(0) = \vartheta_2 -\vartheta_1$, and (i) is proved.

\smallskip

To prove (ii) we use again \eqref{eq:ODEEnergyCons}:  $(\dot w^{-\vep})^2 - \Phi(w^{-\vep}) =  - 2\vep + \vep^2$. Hence, for $\vep \in \Big(0,\frac{\vartheta_1^2}{8c_1}\Big)$, we find
\[
(\dot w^{-\vep})^2 \geq \Phi(w^{-\vep}) - 2\vep \geq \tfrac{\vartheta_1^2}{2c_1} - 2\vep \geq \tfrac{\vartheta_1^2}{4c_1}> 0, \quad \text{ in } \{\vartheta_1 \leq w^{-\vep} \leq \vartheta_2 \},
\]
which allows us to conclude similarly as for (i).
\end{proof}
\begin{lem}\label{lem:1DSolutionCompxwvep}
For every $\sigma \in (0,1)$, there exists $\vep_0 \in (0,1)$ depending only on $\vartheta_1$, $\vartheta_2$, $c_1 > 0$ in \eqref{eq:AssumptionsPhi} and $\sigma$, such that for every $\delta \in [0,1)$ and every $\varepsilon \in (0,\vep_0)$, if $w_\varepsilon^\varepsilon$ and
$w_\varepsilon^{-\varepsilon}$ are as in Remark \ref{rem:Choicettau}, then:

\smallskip

(i) If $x_\varepsilon^\vep$ is such that $v_\varepsilon^\varepsilon(x_\varepsilon^\vep) = 0$, then
\begin{equation}\label{eq:1DSolCompxwvep}
\begin{aligned}
w_\varepsilon^\varepsilon(x - \delta - \vep^\sigma) + \delta + \tfrac{1}{2} \vep^\sigma &\leq x, \quad \;  x  \in (x_\varepsilon^\vep + \delta + \vep^\sigma, 1) \\
w_\varepsilon^\varepsilon(x + \delta + \vep^\sigma) - \delta - \tfrac{1}{2}\vep^\sigma &\geq x, \quad \;  x  \in (-1,1). 
\end{aligned}
\end{equation}

(ii) If $y_\varepsilon^{-\vep}$ is the minimum point of $v_\varepsilon^{-\varepsilon}$, then
\begin{equation}\label{eq:1DSolCompxwvep1}
\begin{aligned}
w_\varepsilon^{-\varepsilon}(x - \dd - \vep^\sigma) + \dd + \tfrac{1}{2}\vep^\sigma &\leq x, \quad \;  x  \in (y_\varepsilon^{-\vep} + \dd + \vep^\sigma,1)  \\
w_\varepsilon^{-\varepsilon}(x + \delta + \vep^\sigma) - \delta - \tfrac{1}{2} \vep^\sigma &\geq x, \quad \;  x  \in (-1,1).
\end{aligned}
\end{equation}

\end{lem}
\begin{proof} Let us prove part (i). To simplify the notations, we set $w^\varepsilon := w_\varepsilon^\varepsilon$ and $x_\vep = x_\vep^\vep$. Let $\tilde{x}_\varepsilon > 0 > x_\vep$ such that $w^\varepsilon (\tilde{x}_\varepsilon) = \vartheta_2\varepsilon$, hence $w^\vep$ is linear for $x\ge \tilde x_\vep $.  Then if $x \in (\tilde{x}_\varepsilon + \delta + \vep^\sigma,1)$, we have
\[
\begin{aligned}
w^\varepsilon(x - \delta - \vep^\sigma) - x &= \vartheta_2\varepsilon + (1+\varepsilon)(x - \delta - \vep^\sigma - \tilde{x}_\varepsilon) - x \\
&= (\vartheta_2 + x)\varepsilon - (1+\varepsilon)(\delta + \vep^\sigma) - (1+\vep)\tilde{x}_\varepsilon \\
& \leq (\vartheta_2 + 1)\vep - \vep^\sigma - \delta \leq -\delta - \tfrac{1}{2} \vep^\sigma,
\end{aligned}
\]
for every $\vep \leq \vep_0 \leq [2(\vth_2+1)]^{\frac{1}{\sigma-1}}$, while if $x  \in (x_\varepsilon + \delta + \vep^\sigma, \tilde{x}_\varepsilon + \delta + \vep^\sigma)$, we obtain by \eqref{eq:1DSignChangeSolProp} (with $t=\vep$) 
\[
\begin{aligned}
w^\varepsilon(x - \dd - \vep^\sigma) - x &\leq \vartheta_2\varepsilon - (x_\varepsilon + \dd + \vep^\sigma) \leq \vartheta_2 \vep  + C\vep |\log \vep| - \dd - \vep^\sigma \\
&\leq   - \dd -  \tfrac{1}{2} \vep^\sigma,
\end{aligned}
\]
taking eventually $\vep_0$ smaller. Notice that the constant $C > 0$ depends only on $\vth_1$, and $c_1$ (cf. \eqref{eq:1DSignChangeSolProp}).

To show the second inequality in \eqref{eq:1DSolCompxwvep}, we assume first $x + \dd + \vep^\sigma \geq \tilde{x}_\varepsilon$ and we notice that, since $\tilde{x}_\varepsilon \in (0,c\vep)$ (where $c > 0$ is as in Lemma \ref{lem:1DSolutionBoundFB}), we have
\[
\begin{aligned}
w^{\varepsilon}(x + \dd + \vep^\sigma) - x &= \vartheta_2\varepsilon + (1+\varepsilon)(x + \dd + \vep^\sigma - \tilde{x}_\varepsilon) - x \\
&\geq \dd + \vep^\sigma - \tilde{x}_\vep + \vep(x + \dd + \varepsilon^\sigma - \tilde{x}_\vep) \geq \dd + \vep^\sigma - c\vep \geq \dd + \tfrac{1}{2}\vep^\sigma,
\end{aligned}
\]
provided that $\vep_0$ is small enough. Further, since $\tilde{x}_\varepsilon \leq c\varepsilon$, when $x \leq \tilde{x}_\varepsilon - \delta - \vep^\sigma$ we have $x \leq 0$, and the second inequality in \eqref{eq:1DSolCompxwvep} follows.

To show (ii), we set $w^{-\vep} = w_\vep^{-\vep}$, $y_\varepsilon = y_\varepsilon^{-\vep}$, and we take $\tilde{y}_\vep$ such that $w^{-\varepsilon} (\tilde{y}_\varepsilon) = \vartheta_2\varepsilon$. The proof of the first inequality works exactly as before, using \eqref{eq:1DEvenSolProp2} instead of \eqref{eq:1DSignChangeSolProp}. To show the second, we assume first $x \in (\tilde{y}_\varepsilon - \delta - \vep^\sigma,1)$ and, recalling that $\tilde{y}_\varepsilon \leq c\varepsilon$, we write
\[
\begin{aligned}
w^{-\varepsilon}(x + \delta + \vep^\sigma) - x &= \vartheta_2\varepsilon + (1-\varepsilon)(x + \delta + \vep^\sigma - \tilde{y}_\varepsilon) - x \\
&= (\vartheta_2 - x)\varepsilon + (1 -\varepsilon)(\delta + \vep^\sigma) - (1-\vep)\tilde{y}_\varepsilon \\
& \geq  \delta + (1-\vep) \vep^\sigma - (1+c)\vep - \vep\dd  \geq \delta + \tfrac{1}{2} \vep^\sigma,
\end{aligned}
\]
taking eventually $\vep_0$ smaller. As above, if $x \leq \tilde{y}_\varepsilon - \delta - \vep^\sigma$, then $x$ is negative and the inequality is automatically satisfied.
\end{proof}
We end this section by proving that solutions $u_\vep$ to \eqref{eq:EulerEq} decay exponentially fast inside $\{u_\vep \leq \vth_1\vep\}$ as $\vep \to 0$. This is a main fact we will use later in Section \ref{Sec:ImpFlat} (see for instance Lemma \ref{lem:ImpFlat2tris}). This decay is obtained in Lemma \ref{lem:ExpDecayu} using a sliding type argument based on the continuous family of  super-solutions constructed in the following lemma.
\begin{lem}\label{lem:SuperSol}
Fix $c_1 > 0$ as in \eqref{eq:AssumptionsPhi} and $c^2 := \tfrac{1}{c_1}$. For every $\varepsilon \in (0,1)$, $\varrho > 0$ and $R \geq \varrho$, let
\begin{equation}\label{eq:Supersol0}
\varphi(r) = \varphi_{\varepsilon,\varrho,R}(r) := \, e^{-\frac{\mu_+}{\varepsilon}(R-r)} \, \frac{1 - \tfrac{\mu_+}{\mu_-} e^{-\frac{\mu_+ - \mu_-}{\varepsilon}(r-\varrho)}}{1 - \tfrac{\mu_+}{\mu_-} e^{-\frac{\mu_+ - \mu_-}{\varepsilon}(R-\varrho)}}, \qquad r \in [\varrho,R],
\end{equation}
where $\mu_\pm$ are defined by
\begin{equation}\label{eq:SuperSolLambdas}
\mu_{\pm} = - \tfrac{N-1}{2\varrho}\varepsilon \pm \sqrt{\big(\tfrac{N-1}{2\varrho}\big)^2\varepsilon^2 + c^2}.
\end{equation}
Then, for every $x_0 \in \mathbb{R}^N$ and $\varrho > 0$, the function
\begin{equation}\label{eq:SuperSol}
\psi(x) := \psi_{\varepsilon,\varrho,R,x_0}(x) :=
\begin{cases}
\varphi(\varrho)     \quad &\text{in } B_\varrho(x_0) \\
\varphi(|x - x_0|)   \quad &\text{in } B_R(x_0) \setminus B_\varrho(x_0)
\end{cases}
\end{equation}
satisfies
\begin{equation}\label{eq:SuperSolProb}
\begin{cases}
-\Delta \psi + \tfrac{1}{c_1\varepsilon^2} \psi \geq 0 \quad &\text{in } B_R(x_0) \\
\psi = 1                              \quad &\text{in } \partial B_R(x_0) \\
\partial_r \psi \geq 0                                      \quad &\text{in } B_R(x_0),
\end{cases}
\end{equation}
in the weak sense.
\end{lem}
\begin{proof} Up to translations and scaling, we may assume $x_0 = 0$, $\varepsilon = 1$ and set $\varphi = \varphi_1$, $\psi = \psi_1$. Notice that if $\varrho = R$, we have $\psi \equiv 1$ in $B_R$ (i.e. $\varphi = 1$ in $(0,R)$) and \eqref{eq:SuperSolProb} is trivial.

If $0 < \varrho < R$, since $\varphi(\varrho) > 0$ and $\varphi(R) = 1$, it suffices to verify that the differential inequality in \eqref{eq:SuperSolProb} is satisfied in $B_R \setminus B_\varrho$ with $\varphi'(\varrho) = 0$ and $\varphi' \geq 0$ in $(\varrho,R)$.

To see this, we notice that if $r \in (\varrho,R)$ and $\varphi' \geq 0$, then
\[
-\Delta \varphi + c^2\varphi = -\varphi'' - \tfrac{N-1}{r}\varphi' + c^2\varphi \geq  -\varphi'' - \tfrac{N-1}{\varrho}\varphi' + c^2\varphi,
\]
and so, it is enough to check that
\[
\begin{cases}
-\varphi'' - \frac{N-1}{\varrho}\varphi' + c^2\varphi = 0 \quad &\text{ in } (\varrho,R) \\
\varphi' \geq 0 \quad &\text{ in } (\varrho,R) \\
\varphi'(\varrho) = 0.
\end{cases}
\]
Integrating the equation above, we easily see that
\[
\varphi(r) = A e^{\mu_+ r} + Be^{\mu_- r} \quad r \in (R/2,R),
\]
for some suitable constants $A,B \in \mathbb{R}$, and $\mu_{\pm}$ as in \eqref{eq:SuperSolLambdas}. Imposing that $\varphi'(\varrho) = 0$ and $\varphi(R) = 1$, we deduce
\[
A = \frac{1}{e^{\mu_+R}( 1 - \tfrac{\mu_+}{\mu_-} e^{-(\mu_+ - \mu_-)(R-\varrho)})}, \qquad B = - \tfrac{\mu_+}{\mu_-} e^{(\mu_+ - \mu_-)\varrho} A
\]
and, substituting into the expression of $\varphi$, \eqref{eq:Supersol0} follows. Checking that $\varphi' \geq 0$ in $(\varrho,R)$ is a straightforward computation.
\end{proof}
\begin{lem}\label{lem:ExpDecayu}
There exists $\varepsilon_0 \in (0,1)$ depending only on $N$ and $c_1$ such that for every $\varepsilon \in (0,\varepsilon_0)$, every solution $u_\varepsilon$ to \eqref{eq:EulerEq}, every $x_0 \in \{u_\varepsilon \leq \vartheta_1\varepsilon\}$ and every ball $B_{\varepsilon^{3/4}}(x_0) \subset \{u \leq \vartheta_1\varepsilon\}$, then
\begin{equation}\label{eq:ExpDecayu}
u_\varepsilon \leq 3\vartheta_1 \varepsilon \, e^{-\frac{\vep^{-1/4}}{4 c_1^{1/2}} } \quad \text{ in } B_{\frac{\varepsilon^{3/4}}2}(x_0).
\end{equation}
\end{lem}
\begin{proof} Fix $R > 0$ and $x_0 \in \{u \leq \vartheta_1\varepsilon\}$ such that $B_R(x_0) \subset \{u \leq \vartheta_1\varepsilon\}$. Let $\psi_\varrho :=   \psi_{\varepsilon,\varrho,R,x_0}$ be defined as in \eqref{eq:SuperSol}, satisfying \eqref{eq:SuperSolProb}, and let $\tilde \psi_\varrho : =  \vartheta_1\vep \psi_\varrho$.

If $\varrho = R$, then $\tilde\psi_R = \vth_1\vep$ satisfies \eqref{eq:SuperSolProb}, with $\tilde\psi_R \geq u_\vep$ in $B_R(x_0)$. Setting $v := \tilde\psi_R - u_\vep$ and recalling that $B_R(x_0) \subset \{u \leq \vartheta_1\varepsilon\}$, we obtain
\[
-\Delta v + \tfrac{1}{c_1\varepsilon^2}v = -\Delta \tilde \psi_R + \tfrac{1}{c_1\varepsilon^2} \tilde\psi_R + \Delta u - \tfrac{1}{c_1\varepsilon^2} u \geq \Delta u - \tfrac{1}{2}\Phi_\varepsilon'(u) = 0,
\]
and thus
\[
\begin{cases}
-\Delta v + \tfrac{1}{c_1\varepsilon^2} v \geq 0 \quad &\text{in } B_R(x_0) \\
v \geq 0  \quad &\text{in } B_R(x_0).
\end{cases}
\]
By the strong maximum principle, $v > 0$ in $B_R(x_0)$ (it cannot be $v=0$ since $\psi_R$ is a strict super-solution), that is $\psi_R > u_\vep$ in $B_R(x_0)$. Now, let
\[
\varrho_\ast := \inf \{\varrho \in (0,R]: \psi_\varrho > u_\vep \text{ in } B_R(x_0) \}.
\]
We have $\varrho_\ast = 0$. If by contradiction, $\varrho_\ast > 0$, we may repeat the above argument setting $v:= \tilde\psi_{\varrho_\ast} - u_\vep$ and noticing that $v \geq 0$ in $B_R(x_0)$ with $v(x_\ast) = 0$, for some $x_\ast \in \overline{B}_R(x_0)$. Since  by construction $B_R(x_0) \subset \subset\{u \leq \vartheta_1\varepsilon\}$, $\tilde\psi_{\varrho_\ast} = \vartheta_1\vep$  on $\partial B_R(x_0)$, and  $\psi_{\varrho_\ast}$ is radially increasing near the boundary of the ball, it must be $x_\ast \in B_R(x_0)$. Thus using the linear equation for $v$ and the strong maximum principle either $v \equiv 0$ or $v > 0$ in $B_R(x_0)$. Since both scenarios are impossible, our contradiction follows.

In particular, we have $\varrho_\ast < \tfrac{R}{2}$ and so, $u_\vep \leq \psi_{R/2}$ in $B_R(x_0)$. Now, choosing $R = \vep^{3/4}$, taking $\varrho = \tfrac{R}{2}$ in \eqref{eq:Supersol0} and using \eqref{eq:SuperSol}, we obtain
\[
u_\vep \leq \vep \vartheta_1\varphi_{\vep^{3/4}/2}(\vep^{3/4}/2) \leq \vep\vartheta_1 \big( 1 - \tfrac{\mu_+}{\mu_-} \big) \, e^{-\frac{\mu_+}{2 \sqrt[4]{\varepsilon}}} \quad \text{ in } B_{\vep^{3/4}/2}(x_0),
\]
where $\mu_{\pm}$ are defined in \eqref{eq:SuperSolLambdas} (with $R = \varepsilon^{3/4}$). Since $\mu_{\pm} \to \pm \tfrac{1}{\sqrt{c_1}}$ as $\varepsilon \to 0^+$, there is $\varepsilon_0 \in (0,1)$ (depending only on $N$ and $c_1$) such that $\mu_+ \geq 1/(2\sqrt{c_1})$ and $-\mu_+/\mu_- \leq 2$ for every $\varepsilon \in (0,\varepsilon_0)$ and thus \eqref{eq:ExpDecayu} follows.
\end{proof}

%

%
%
%
%
\section{Lipschitz and non-degeneracy estimates}\label{Sec:NonDeg}

We recall now a useful  Lipchitz estimate from  \cite{CafSal2005:book}.
\begin{prop} [{Uniform Lipschitz estimate; see \cite[Theorem 1.2]{CafSal2005:book}}]\label{cor:LipEst}
For any $V \suubset B_1$, there exists $\overline{C} > 0$ depending only on $N$, $L$, $\vth_2$ and $V$ such that for every $\varepsilon \in (0,1)$ and for every critical point $u_\vep$ of \eqref{eq:Energy} in $B_1$ with  $u_\vep (0)\le \vartheta_2\vep$ we have
\begin{equation}\label{eq:LipEst2}
\sup_V |\nabla u_\vep| \leq \overline{C}.
\end{equation}
\end{prop}
We also  need a non-degeneracy estimate related to \cite[Theorem 1.8]{CafSal2005:book}.  Our estimate is stronger since  balls $B_r(z)$ do not need to be centered at some point in $\{u\ge C\vep\}$, with $C$ large, and can be centered at any point in $ \{u_\varepsilon \geq \vartheta_1\varepsilon\}$
\begin{lem} [Uniform non-degeneracy] \label{lem:NonDegEst}
There exists $\varepsilon_0 \in (0,1)$ depending only on $\vartheta_1$ and $c_1$ such that for every $\kappa > 0$, there exists $c_\kappa > 0$ depending only on $N$, $L$, $\vth_2$ and $\kappa$ such that for every $\varepsilon \in (0,\varepsilon_0)$, every local minimizer $u_\vep$ of \eqref{eq:Energy} in $B_1$, every $z \in \{u_\varepsilon \geq \vartheta_1\varepsilon\}$ and every $r \geq \kappa\vep$ such that $B_r(z) \suubset B_1$, then
\begin{equation}\label{eq:NonDegEst1}
\sup_{B_r(z)} u_\vep \geq c_\kappa \, r.
\end{equation}
\end{lem}
\begin{proof}  Let us fix $\kappa > 0$, and assume that $\varepsilon \in (0,\varepsilon_0)$, $u = u_\varepsilon$, $z \in \{u_\varepsilon \geq \vartheta_1\varepsilon\}$ and $r \geq \kappa \varepsilon$. 
Define
\begin{equation}\label{eq:NonDegDelta}
\omega(r): = \tfrac{1}{r} \sup_{B_r(z)} u.
\end{equation}
Our goal is to prove a lower bound for $\omega$, which holds if $\vep_0$ is small enough.
Up to translate and scaling, we may assume $r = 1$ and $z = 0$. Let $\sigma := \tfrac{1}{3}$ where $c > 0$ is the constant appearing in \eqref{eq:NonDegDecayIneq} depending only on $N$ and $c_1$.

\smallskip

\emph{Step 1: Estimates.} Let $\varphi \in C_0^\infty(B_1)$, $0 \leq \varphi \leq 1$, with $\varphi = 1$ in $B_{7/8}$. Assume also
\begin{equation}\label{eq:NonDegPropTest}
|\nabla \varphi| \leq c_N, \qquad  |\Delta \varphi| \leq c_N,
\end{equation}
for some $c_N > 0$. Testing the equation of $u$ with $\eta = u \varphi^2$, it is not difficult to find
\[
\int_{B_1} \left[ |\nabla u|^2 + \tfrac{1}{2} \Phi_\vep'(u)u \right] \varphi^2 \dx = \tfrac{1}{2} \int_{B_1} u^2 \Delta (\varphi^2) \dx,
\]
which, since $\Phi_\varepsilon'(u)u \geq 0$ implies
\begin{equation}\label{eq:NonDegCacc}
\int_{B_{7/8}}  |\nabla u|^2 \dx \leq c_N \int_{B_1} u^2 \dx,
\end{equation}
for some new $c_N > 0$.

Now, let $\phi \in C^\infty(\RR^N)$, $0 \leq \phi \leq 1$ with $\phi = 0$ in $B_{3/4}$ and $\phi = 1$ in $\RR^N \setminus B_{7/8}$, satisfying \eqref{eq:NonDegPropTest}. Taking $v = \phi u$ as a competitor for $u$, we deduce
\[
\begin{aligned}
\int_{B_1} \Phi_\vep(u) - \Phi_\vep(\phi u) \dx &\leq \int_{B_1}  |\nabla (u\phi)|^2 - |\nabla u|^2 \dx
\\
 &\leq \int_{B_1} \left( \phi^2 -1 \right) |\nabla u|^2\dx + 2 \int_{B_{7/8}} u^2|\nabla \phi|^2\dx +  \int_{B_{7/8}} |\nabla u|^2 \phi^2\dx \\
&\leq c_N \int_{B_{7/8}} |\nabla u|^2  + u^2 \dx,
\end{aligned}
\]
for some new $c_N > 0$ and so, recalling that $\phi \leq 1$, $\Phi_\vep' \geq 0$ and using \eqref{eq:NonDegCacc}, it follows
\[
\int_{B_{3/4}} \Phi_\vep(u) \dx \leq \int_{B_1} \Phi_\vep(u) - \Phi_\vep(\phi u) \dx \leq c_N \int_{B_1} u^2 \dx.
\]
In particular, by the definition of $\omega$, we conclude
\begin{equation}\label{eq:NonDegFunEst}
\int_{B_{3/4}} \Phi_\vep(u) \dx \leq c_N \omega(1)^2,
\end{equation}
for some new $c_N > 0$.

\smallskip

\emph{Step 2: Decay of $\omega$.} Note that for all  $y \in B_{1/2}$, since $u$ is subharmonic, we have
\[
u(y) \leq \fint_{B_{1/4}(y)} u \dx \leq c_N \int_{B_{3/4}} u \dx = c_N \left( \int_{B_{3/4}\cap\{u \geq t\}} u \dx +  \int_{B_{3/4}\cap\{u \leq t\}} u \dx \right),
\]
for every $t > 0$. Recalling that $\Phi$ is nondecreasing, there holds $\{u \geq t\} \subseteq \{\Phi_\varepsilon(u) \geq \Phi_\varepsilon(t)\}$ and, using that $\Phi_\vep(t) \geq \tfrac{1}{2c_1}(t/\vep)^2$ for $t \in (0,\vartheta_1\vep]$ combined with \eqref{eq:NonDegDelta}, it follows
\[
\begin{aligned}
\int_{B_{3/4}\cap\{u \geq t\}} u \dx &\leq \omega(1)
  \int_{B_{3/4}\cap\{u \geq t\}} \dx \leq \omega(1) \int_{B_{3/4}\cap\{\Phi_\varepsilon(u) \geq \Phi_\varepsilon(t)\}} \dx \\
&\leq c_1 \omega(1)
 \left( \frac{\varepsilon}{t} \right)^2 \int_{B_{3/4}} \Phi_\varepsilon(u) \dx \leq c_1 c_N \left( \frac{\varepsilon}{t} \right)^2 \omega^3(1),
\end{aligned}
\]
where the last inequality is a direct application of \eqref{eq:NonDegFunEst}. Substituting into the inequality above, we deduce
\[
u(y) \leq c_N \left[ c_1c_N \left( \frac{\varepsilon}{t} \right)^2 \omega^3(1) +  t \right] \leq c \left[ \left( \frac{\varepsilon}{t} \right)^2 \omega^3(1) +  t \right],
\]
for some $c > 0$ depending only on $N$ and $c_1$, and so, by the arbitrariness of $y \in B_{1/2}$,
\begin{equation}\label{eq:NonDegDecayIneq}
\omega \big(\tfrac{1}{2} \big) \leq c \Big[ \left( \frac{\varepsilon}{t} \right)^2 \omega^3(1) +  t \Big].
\end{equation}
Setting $t := \min\{ \max\{ \vep, \omega(1) \}^{1+2\sigma} , \vartheta_1\vep\} $, we have that $t \leq \vartheta_1\vep$ thanks to the definition of $\varepsilon_0$. So, using that $\sigma = \tfrac{1}{3}$, we may re-write \eqref{eq:NonDegDecayIneq} as
\begin{equation}\label{wntowetow}
\omega \big(\tfrac{1}{2}\big) \leq c \max\{\varepsilon,\omega(1)\}^{1+2\sigma}.
\end{equation}

Let us now assume by contradiction that we have  $\vep\le \omega_0$ and $\omega(1)\le \omega_0$, for $\omega_0 \in (0,1/4)$ sufficiently small so that
\eqref{wntowetow} implies
\[
\omega \big(\tfrac{1}{2}\big) \leq \max\{\varepsilon,\omega(1)\}^{1+\sigma}.
\]
After scaling (applying the above inequality to $u_\vep(rx)/r$), we obtain provided $\vep/r\in (0,\omega_0)$,
\[
\omega \big(\tfrac{r}{2}\big) \leq \max\{\varepsilon/r,\omega(r)\}^{1+\sigma}.
\]

Iterating the above inequality, we obtain that whenever $2^k\vep\le \omega_0$, we have either 
\[
{\rm (i) } \quad \omega(2^{-k}) \leq (2^k\varepsilon)^{1+\sigma} \quad \mbox{or} \quad 
{\rm (ii)} \quad  \omega(2^{-k}) \leq \omega(1)^{(1+\sigma)^k},
\]
for all $k \in \mathbb{N}$.
Finally, choosing
\[
k := \lceil \log_2(\varepsilon^{-1/2}) \rceil,
\]
we have $2^{-k}\le  \vep^{1/2}\le 2^{-k+1}$ and hence $2^{k} \vep \le 2 \vep^{1/2} \in (0,\omega_0)$, provided $\vep\in (0, \vep_0)$ with $\vep_0>0$ sufficiently small.

Hence, recalling  $\omega(1)  \le \omega_0 \le \tfrac{1}{4}$ and that by assumptionn $0 \in \{u_\varepsilon \geq \vartheta_1\varepsilon \}$, we have
\begin{equation}\label{eq:NonDegLastBound1}
\max\{(2^k \vep)^{1+\sigma}, (1/4)^{(1+\sigma)^k}\} \ge \omega(2^{-k}) := 2^k \sup_{B_{2^{-k}}} u \geq \vartheta_1 \varepsilon^{1/2},
\end{equation}
which clearly gives a contradiction if $\vep\in(0,\vep_0)$ with $\vep_0$ chosen sufficiently small (since $(2^k \vep)^{1+\sigma} \le (2\vep^{1/2})^{1+\sigma} \ll \vep^{1/2}$ 
and $(1/4)^{(1+\sigma)^k} \ll 4^{-k} \ll \vep^{1/2}$ as $\vep \downarrow 0$).
\end{proof}
%
%
%
%
%

%
%
%
%

%
\section{Proof of Proposition \ref{prop:main}}\label{Sec:AsymptBlowDown}
This is section is devoted to the proof of Proposition \ref{prop:main}. It will be obtained as a corollary of the following result, which is its equivalent version in terms of blow-down families.
\begin{prop}\label{prop:LimitBlowDown} Let $\Phi$  be as in \eqref{PHI111}-\eqref{PHI222} and let $\vartheta_1$ and $\vartheta_2$ as in \eqref{eq:AssumptionsPhi}. Let $u:\RR^N \to \RR_+$ be a minimizer of $\EE$ in $\RR^N$ not identically 0, with $0 \in \{ \vth_1 \leq u \leq \vth_2\}$. Let $\{\vep_j\}_{j\in\NN}$ be a sequence satisfying $\vep_j \to 0$ as $j \to +\infty$ and let $u_{\vep_j}$ be the corresponding blow-down family.

Then for every $\alpha \in (0,1)$, there exist sequences $\varepsilon_{j_\ell},\dd_\ell \to 0$ and a 1-homogeneous entire local minimizer of \eqref{eq:Energy1Phase} $u_0 \in W_{loc}^{1,\infty}(\RR^N)$ --- also not identically 0 --- such that
\begin{equation}\label{eq:ConvToHalhPlaneSol}
|u_{\varepsilon_{j_\ell}} - u_0| \leq \dd_\ell \quad \text{ in }  B_1,
\end{equation}
and
\begin{equation}\label{eq:1DSymAss2XXbis}
\{ x: {\rm dist} (x,  \{u_0>0\}) \ge \delta_\ell \} \subset \{ u_{\vep_{j_\ell}} \leq \vth_1\vep_{j_\ell} \} \subset \{ u_{\vep_{j_\ell}} \leq \vth_2\vep_{j_\ell} \} \subset  \{x: {\rm dist} (x,  \{u_0=0\}) \le \delta_\ell  \} \quad \text{ in } B_1,
\end{equation}
for every $\ell \in \NN$.
\end{prop}
The above statement will follow as a byproduct of several auxiliary results, having independent interest: in Lemma \ref{lem:ConvFunctionals} we prove that families of minimizers of \eqref{eq:Energy} converge (in a suitable sense, up to subsequences) to a minimizer of \eqref{eq:Energy1Phase}, while in Lemma \ref{lem:HaussConvGammaWeak} and Corollary \ref{lem:HaussConvGammaWeak1} we deal with the convergence of the level sets of $u_\vep$. Proposition \ref{prop:LimitBlowDown} is a consequence of these facts and a Weiss type monotonicity formula (Lemma \ref{lem:WeissFormula}).
\begin{lem}\label{lem:ConvFunctionals}

Let $R > 0$ and $\{u_{\varepsilon_j}\}$, $\vep_j\downarrow 0$, be a sequence of minimizers of \eqref{eq:Energy} in $B_R$, with $\vep = \vep_j$.
Assume $u_{\vep_j}(0) \le \vartheta_2\vep_j$. 
Then, up to subsequence, we have
\begin{equation}\label{eq:ConvH1Caa}
u_{\varepsilon_j} \to u_0 \quad \text{ in } H_{loc}^1(B_R)\cap C_{loc}^\alpha(B_R), \quad \mbox{for all $\alpha \in (0,1)$,}
\end{equation}
as $j \to +\infty$, where $u_0 \in W_{loc}^{1,\infty}(B_R)$ is a minimizer of \eqref{eq:Energy1Phase} in $B_R$.
\end{lem}
\begin{proof}
By scaling we may assume $R=1$. By Proposition \ref{cor:LipEst}, the family $\{u_\varepsilon\}_{\varepsilon\in(0,1)}$ is uniformly bounded in $W_{loc}^{1,\infty}(B_1)$. So, by the Ascoli-Arzel\`a theorem, for every $\alpha \in (0,1)$, there exists $u_0 \in W_{loc}^{1,\infty}(B_1)$ and $\varepsilon_j \to 0$ as $j \to +\infty$ such that $u_{\vep_j} \to u_0$ in $C_{loc}^\alpha(B_1)$. Furthermore, since in addition each $u_\varepsilon$ is subharmonic and $\{u_\varepsilon\}_{\varepsilon \in (0,1)}$ is uniformly bounded in $L_{loc}^2(B_1)$, we deduce $u_{\vep_j} \to u_0$ in $W^{1,1}_{loc}(B_1)$, up to subsequence (see for instance \cite[Lemma A.1]{CabreEtAl2020:art}). Consequently, since $u_{\vep_j}, u_0 \in W_{loc}^{1,\infty}(B_1)$ we deduce $u_{\vep_j} \to u_0$ in $H_{loc}^1(B_1)$ by interpolation and \eqref{eq:ConvH1Caa} is proved.

Now, let us set for simplicity $u := u_0$ and $u_j := u_{\varepsilon_j}$. Let us fix $V \suubset B_1$ and show that
\begin{equation}\label{eq:LocSemContFun0}
\mathcal{E}_0(u,V) \leq \liminf_{j\to+\infty} \mathcal{E}_{\varepsilon_j}(u_j,V).
\end{equation}
Indeed, by $H^1_{loc}$ convergence, it is enough to check that
\begin{equation}\label{eq:FunCovPhi}
\int_V \chi_{\{u > 0\}} \dx \leq \liminf_{j\to+\infty} \int_V \Phi_{\varepsilon_j} (u_j) \dx.
\end{equation}
To show \eqref{eq:FunCovPhi}, we first notice that $\Phi_{\varepsilon_j}(u_j) \to 1$  in $\{u > 0\}$. Indeed, if $x \in \{u > 0\}$, that is $u(x) \geq \epsilon_x$ for some $\ep_x > 0$, then $u_j(x) \geq \epsilon_x/2 > 0$ for all $j$ large enough. Now, by monotonicity, $\Phi_{\varepsilon_j}(\epsilon_x/2) \leq \Phi_{\varepsilon_j}(u_j(x))$ for $j$ large enough and thus, by definition of $\Phi_\varepsilon$,
\[
1 = \lim_{j\to+\infty} \Phi_{\varepsilon_j}(\epsilon_x/2) \leq \limsup_{j\to+\infty} \Phi_{\varepsilon_j}(u_j(x)) \leq 1.
\]
Consequently, by Fatou's lemma
\[
\int_V \chi_{\{u > 0\}} \dx = \int_{ V \cap \{ u > 0 \} }  \dx \leq \liminf_{j\to+\infty} \int_{V \cap \{ u > 0 \}} \Phi_{\varepsilon_j} (u_j) \dx \leq \liminf_{j\to+\infty} \int_V \Phi_{\varepsilon_j} (u_j) \dx,
\]
and \eqref{eq:FunCovPhi} follows.

Once \eqref{eq:LocSemContFun0} is established, let us fix $V := B_{r}$, $r<1$, $\xi \in C_0^\infty(V)$, and $\varphi  \in C^\infty(\overline B_r)$ vanishing on $\partial B_r$ with $\varphi > 0$ in $B_r$. Since $u_j$ is a local minimizer, we have
\begin{equation}\label{eq:EEjLocTest}
\EE_{\vep_j}(u_j,V) \leq \EE_{\vep_j}(u_j + \xi - \delta\varphi,V),
\end{equation}
for all $j \in \NN$ and $\delta > 0$. Since $u_j \to u$ in $H^1(V)$, we immediately see that
\begin{equation}\label{eq:H1locConvGrad}
\int_V |\nabla (u_j + \xi) - \delta \nabla \varphi|^2 \dx \to \int_V |\nabla (u + \xi) - \delta \nabla \varphi|^2 \dx
\end{equation}
as $j\to +\infty$. Now, if $x \in \{u + \xi - \delta \varphi > 0 \} \cap V$, there is $\epsilon_x >0$ such that $u(x) + \xi(x) - \delta \varphi(x) \geq \epsilon_x$ and, since $u_j \to u$ locally uniformly, it must be $u_j(x) + \xi(x) - \delta \varphi(x) \geq \epsilon_x/2$ for every $j$ large enough. Consequently, by monotonicity,
\[
1 = \lim_{j\to+\infty} \Phi_{\vep_j}(\epsilon_x/2) \leq \limsup_{j\to+\infty} \Phi_{\vep_j}(u_j(x) + \xi(x) - \delta \varphi(x)) \leq 1.
\]
Similar, whenever $x \in \{u + \xi - \delta \varphi < 0 \} \cap V$, then $u(x) + \xi(x) - \delta \varphi(x) \leq -\epsilon_x$ for some $\epsilon_x >0$ and so $u_j(x) + \xi(x) - \delta \varphi(x) \leq -\epsilon_x/2$ for every $j$ large enough, which implies
\[
0 \leq \Phi_{\vep_j}(u_j(x) + \xi(x) - \delta \varphi(x)) \leq \Phi_{\vep_j}(-\ep_x/2) = 0,
\]
when $j$ is large enough. On the other hand, for  and every $m\in \mathbb N$, we have\footnote{To see this, it is enough to apply the Coarea formula to the function $\tfrac{u+\xi}{\varphi}$, which is Lipschitz in $B_{r-1/m}$.}
\begin{equation}\label{eq:PropDelta}
|\{u + \xi - \delta \varphi = 0 \}\cap B_{r-1/m}| = 0 \quad \mbox{for all  $\delta \in E_m\subset (0,1)$, where $|(0,1)\setminus E_m| =0$}  .
\end{equation}
Consequently, since $|\cup_m ((0,1)\setminus E_m)| =0$,
\begin{equation}\label{eq:PropDelta2}
|\{u + \xi - \delta \varphi = 0 \}\cap B_{r}| = 0 \quad \mbox{for a.e. $\delta \in (0,1)$},
\end{equation}
and we deduce that for a.e. $\delta > 0$, $\Phi_{\vep_j}(u_j + \xi - \delta \varphi) \to \chi_{\{ u + \xi - \delta \varphi > 0 \}}$ a.e. in $B_r$, as $j \to +\infty$.

So, putting together \eqref{eq:LocSemContFun0}, \eqref{eq:EEjLocTest}, \eqref{eq:H1locConvGrad}, noticing that $\{u + \xi - \delta \varphi > 0\} \subseteq \{u + \xi > 0\}$ and passing to the limit as $j \to +\infty$ by means of the dominated convergence theorem, we find
\[
\begin{aligned}
\EE_0(u,V) &\leq \int_V |\nabla (u + \xi) - \delta \nabla \varphi|^2 + \chi_{\{ u + \xi - \delta \varphi > 0 \}}\dx \\
&\leq \EE_0(u+\xi,V) + 2\delta \|\nabla (u+\xi)\|_{L^2(V)}\|\nabla \varphi \|_{L^2(V)} + \delta^2\|\nabla \varphi \|_{L^2(V)}^2,
\end{aligned}
\]
for a.e. $\delta > 0$. Finally, passing to the limit along a sequence $\delta = \delta_k \to 0$ for which \eqref{eq:PropDelta2} is satisfied for every $k\in\NN$, we find $\EE_0(u,V) \leq \EE_0(u+\xi,V)$ and the thesis follows by the arbitrariness of $B_r\suubset B_1$ and $\xi \in C_0^\infty(B_r)$.
\end{proof}
\begin{lem}\label{lem:HaussConvGammaWeak}
Let $R > 0$, $\{u_\varepsilon\}_{\varepsilon \in (0,1)}$ and $u_0$ as in Lemma \ref{lem:ConvFunctionals}. Then, for every $\vth \geq \vth_1$, there exists a sequence $\vep_j \to 0$ such that
\begin{equation}\label{eq:HausConvSupp}
\{ u_{\varepsilon_j} \geq \vartheta \varepsilon_j \} \to \overline{\{ u_0 > 0 \}} \quad \text{ locally Hausdorff in } B_R,
\end{equation}
as $j \to +\infty$.
\end{lem}
\begin{proof} By scaling, we may assume $R=1$. Fix $\varrho \in (0,1)$ and $\vth \geq \vth_1$. Set $u = u_0$, $u_j = u_{\varepsilon_j}$, $U_j := \{ u_{\varepsilon_j} > \vartheta \varepsilon_j \} \cap B_\varrho$, $\Omega := \{ u > 0 \} \cap B_\varrho$, and notice that by assumption $0\in\Omega^c$. We first show that for every $z \in \overline{\Omega}$ and every $r > 0$ such that $B_r(z) \suubset B_1$, then
\begin{equation}\label{eq:NonDegBlowDown}
\sup_{B_r(z)} u \geq \tfrac {c} 2  r,
\end{equation}
where $c > 0$ is the constant appearing in Lemma \ref{lem:NonDegEst} for $\kappa=1/2$. Given such $z \in \overline{\Omega}$ and $r > 0$, we take $y \in B_{r/2}(z)$ such that $u(y) > 0$. So, by uniform convergence, $y \in U$ for $j$ large enough (and thus $u_j(y) > \vartheta_1\vep_j$). So by \eqref{eq:NonDegEst1} (with $\kappa = 1$), there is $x_j \in B_{r/2}(y)$ such that $u_j(x_j) \geq \tfrac c 2  r$. Now, up to passing to a subsequence, $x_j \to x \in \overline{B}_{r/2}(y)$ as $j \to +\infty$ and thus, by $C^\alpha_{loc}$ convergence, $u(x) \geq \tfrac c 2  r$ and \eqref{eq:NonDegBlowDown} follows since $x \in B_r(z)$.

Now fix $\sigma >0$, and define
\[
\Omega_\sigma := \{x: \text{dist}(x,\Omega) \leq \sigma \}, \qquad U_{j,\sigma} := \{x: \text{dist}(x,U_j) \leq \sigma \}.
\]
Let us show that $U_j \subset \Omega_\sigma$ for every $j \geq j_\sigma$, for some $j_\sigma$ large enough. Indeed, assume by contradiction there is a sequence $z_j$ such that $u_j(z_j) \geq \vartheta \varepsilon_j \geq \vartheta_1 \varepsilon_j$, but $z_j \not \in \Omega_\sigma$. Then, by \eqref{eq:NonDegEst1}, there is $j_\sigma$ such that
\[
u_j(x_j) := \sup_{B_{\sigma/2}(z_j)} u_j \geq \tfrac{c}{2} \, \sigma,
\]
for every $j \geq j_\sigma$ and some $x_j \in \overline{B}_{\sigma/2}(z_j)$. In addition, up to passing to a subsequence, $z_j \to z$, $x_j \to x \in \overline{B}_{\sigma/2}(z) \suubset \Omega^c$, and $u_j(x_j) \to u(x)$ as $j \to +\infty$, by $C^\alpha_{loc}$ convergence. Since $u(x) = 0$ by construction, we obtain a contradiction.

We also have $\overline{\Omega} \subset U_{j,\sigma}$ for every $j \geq j_\sigma$. Assume by contradiction there is $z_j \in \overline{\Omega}$ such that $z_j \not\in U_{j,\sigma}$. Then, by \eqref{eq:NonDegBlowDown}, there is $x_j \in \overline{B}_{\sigma/2}(z_j)$ such that $u(x_j) \geq \tfrac{c}{4} \sigma$ while, by construction, $u_j < \vartheta \varepsilon_j$ in $\overline{B}_{\sigma/2}(z_j)$. So, since $z_j \to z$, $x_j \to x \in \overline{B}_{\sigma/2}(z)$ (up to a subsequence), we have $\tfrac{c}{4} \sigma \leq u(x) \leq 0$, a contradiction. The limit \eqref{eq:HausConvSupp} follows from the arbitrariness of $\sigma > 0$.
\end{proof}
\begin{cor}\label{lem:HaussConvGammaWeak1}
Let $R > 0$, $\{u_\varepsilon\}_{\varepsilon \in (0,1)}$ and $u_0$ as in Lemma \ref{lem:ConvFunctionals}. Then, for every $\vth \geq \vth_1$, there exists a sequence $\vep_j \to 0$ such that
\begin{equation}\label{eq:HausConvCompSupp}
\{ u_{\varepsilon_j} \leq \vartheta \varepsilon_j \} \to \{ u_0 = 0 \} \quad \text{ locally Hausdorff in } B_R,
\end{equation}
as $j \to +\infty$.
\end{cor}
\begin{proof} It is enough to apply Lemma \ref{lem:HaussConvGammaWeak} and noticing that $\{ u_{\varepsilon_j} \leq \vartheta \varepsilon_j \} = \{ u_{\varepsilon_j} \geq \vartheta \varepsilon_j \}^c$ and $\{ u_0 = 0 \} = \{ u_0 > 0 \}^c$.
\end{proof}
\begin{lem}\label{lem:WeissFormula}
Let $u$ be a nonnegative entire local minimizer of \eqref{eq:Energy} with $\vep = 1$.

Then, for every $x_0 \in \RR^N$, the function
\begin{equation}\label{eq:WeissFunction}
r \to \WW(u,x_0,r) := r^{-N} \int_{B_r(x_0)} |\nabla u|^2 + \Phi(u) \dx - r^{-1-N} \int_{\partial B_r(x_0)} u^2 \dsi
\end{equation}
is well-defined in $(0,\infty)$ and satisfies
\begin{equation}\label{eq:WeissDerivative}
\frac{d}{dr}\WW(u,x_0,r) = 2r^{-N} \int_{\partial B_r(x_0)} \left(\partial_n u - \frac{u}{r} \right)^2 \rd\sigma + r^{-1-N} \int_{B_r(x_0)} u \Phi'(u) \,\rd\sigma,
\end{equation}
where $\partial_n u := \nabla u \cdot n$ and $n$ is the outward unit normal to $\partial B_r(x_0)$. In particular, the function $r \to \WW(u,x_0,r)$ is non-decreasing.
\end{lem}
%
%
\begin{proof} We follow \cite[Theorem 2]{Weiss1999:art2}. Note first that under our assumptions $u$ is a critical point of $\int |\nabla u| + \Phi(u)$ with $\Phi$  of class $C^{1,1}$.  Hence $u$ satisfies  a semilinear equation of the type $\Delta u = f(u)$ with $f$ Lipschitz. Hence, by standard elliptic regularity and ``semilinear bootstrap'' we have $u\in C^{2,\alpha}_{\rm loc}(\RR^n)$. This qualitative regularity is enough in order to justify the computations below.

Fix $x_0 \in \RR^N$ and let $u_r(x) := \frac{u(x_0+rx)}{r}$. Then
\[
\WW(u,x_0,r) = \int_{B_1} |\nabla u_r|^2 \dx  + \int_{B_1} \Phi(r u_r) \dx - \int_{\partial B_1} u_r^2 \dsi.
\]
Noticing that $r\frac{d}{dr} u_r = \nabla u_r \cdot x -  u_r$ and using the equation of $u_r$, we obtain
\[
\begin{aligned}
\frac{d}{dr} \int_{B_1} |\nabla u_r|^2 \dx &= \tfrac{2}{r} \int_{B_1} \nabla u_r \cdot \nabla (\nabla u_r \cdot x - u_r) \dx \\
&= -\tfrac{2}{r} \int_{B_1} \Delta u_r  (\nabla u_r \cdot x - u_r) \dx + \tfrac{2}{r} \int_{\partial B_1} (\nabla u_r \cdot x) (\nabla u_r \cdot x -  u_r) \dsi \\
&= - \int_{B_1} \Phi' (r u_r)  (\nabla u_r \cdot x - u_r) \dx + \tfrac{2}{r} \int_{\partial B_1} (\nabla u_r \cdot x) (\nabla u_r \cdot x - u_r) \dsi.
\end{aligned}
\]
Similar,
\[
\begin{aligned}
\frac{d}{dr} \left( \int_{B_1} \Phi(r u_r) \dx \right) &=  \int_{B_1} \Phi'(r u_r) (\nabla u_r \cdot x) \dx, \\
-\frac{d}{dr} \int_{\partial B_1} u_r^2 \dsi &= -\tfrac{2}{r} \int_{\partial B_1} u_r (\nabla u_r \cdot x - u_r) \dsi.
\end{aligned}
\]
Summing and rearranging terms, we find
\[
\frac{d}{dr}\WW(u,x_0,r) = \tfrac{2}{r} \int_{\partial B_1} (\nabla u_r \cdot x - u_r)^2 \dsi  + \tfrac{1}{r} \int_{B_1} u_r \Phi_{1/r}' (u_r).
\]
Changing variables $x \to \frac{x-x_0}{r}$, \eqref{eq:WeissDerivative} follows.
\end{proof}
\begin{proof}[Proof of Proposition \ref{prop:LimitBlowDown}] By scaling, $\{u_{\vep_j}\}_{j \in \NN}$ is a family of minimizers of \eqref{eq:Energy} in $\RR^N$ and thus, by Lemma \ref{lem:ConvFunctionals}, Lemma \ref{lem:HaussConvGammaWeak}, Corollary \ref{lem:HaussConvGammaWeak1} and using a standard diagonal argument, we deduce the existence of sequences $\vep_\ell = \vep_{j_\ell},\dd_l \to 0$ and a minimizer $u_0$ of \eqref{eq:Energy1Phase} in $\RR^N$ with $0 \in \partial\{u_0 > 0\}$ such that \eqref{eq:ConvToHalhPlaneSol} and \eqref{eq:1DSymAss2XXbis} are satisfied. The fact that $u_0$ is nontrivial follows by uniform non-degeneracy (Lemma \ref{lem:NonDegEst}).

We are left to show that $u_0$ is $1$-homogeneous. To see this, we use Weiss' monotonicity formula. For every $\vep \in (0,1)$, we consider the function
\[
r \to \WW_\vep(u_\vep,r) := r^{-N} \int_{B_r} |\nabla u_\vep|^2 + \Phi_\vep(u_\vep) \dx - r^{-1-N} \int_{\partial B_r} u_\vep^2 \dsi.
\]
Noticing that $\WW_\vep(u_\vep,r) = \WW(u,r/\vep)$, we easily compute
\[
\frac{d}{dr} \WW_\vep(u_\vep,r) = \frac{1}{\vep} \frac{d}{dr} \WW(u,r/\vep) = 2r^{-N} \int_{\partial B_r} \left( \partial_n u_\vep - \frac{u_\vep}{r} \right)^2 \rd\sigma + r^{-1-N} \int_{B_r} u_\vep \Phi_\vep'(u_\vep) \,\rd\sigma,
\]
and thus, integrating and neglecting the second term in the r.h.s., we deduce
\begin{equation}\label{eq:PropWeissIneqFinal}
\WW(u,R/\vep) - \WW(u,\varrho/\vep) \geq 2 \int_\varrho^R r^{-N} \int_{\partial B_r} \left( \partial_n u_\vep - \frac{u_\vep}{r} \right)^2 \rd\sigma \rd r,
\end{equation}
for every $0 < \varrho < R$ fixed. On the other hand, since $u$ is globally Lipschitz and $\Phi \leq 1$, we have
\[
\WW(u,r) \leq r^{-N} \int_{B_r} |\nabla u_\vep|^2 + \Phi_\vep(u_\vep) \dx \leq  c_N (1 + \|\nabla u\|_{L^\infty(\RR^N)}) < +\infty, \quad \forall r>0.
\]
This,  together with the monotonicity $r \to \WW(u,r)$, yields $\WW(u,r) \to l$ as $r \to +\infty$, for some $l < +\infty$ (depending on $u$). Consequently, taking $\vep = \vep_\ell$ and passing to the limit as $\ell \to +\infty$ in \eqref{eq:PropWeissIneqFinal}, we obtain by $H_{loc}^1$ and $C_{loc}^\alpha$ convergence
\[
\int_\varrho^R r^{-N} \int_{\partial B_r} \left( \partial_n u_0 - \frac{u_0}{r} \right)^2 \rd\sigma \rd r = 0.
\]
By the arbitrariness of $\varrho$ and $R$, it follows $\partial_n u_0 = \frac{u_0}{r}$ in $\partial B_r$, for every $r > 0$, that is, $u_0$ is $1$-homogeneous.
\end{proof}
\begin{proof}[Proof of Proposition \ref{prop:main}] Let $\{R_j\}_{j\in\NN}$ be any sequence satisfying $R_j \to +\infty$ as $j \to +\infty$, and let $\vep_j := \tfrac{1}{R_j}$. Let $\vep_{j_\ell}$, $\dd_\ell$, $u_{\vep_{j_\ell}}$ and $u_0$ as in Proposition \ref{prop:LimitBlowDown} and $R_{j_\ell} := \tfrac{1}{\vep_l}$. Then, since $u_0$ is 1 homogeneous, \eqref{eq:1DSymAss1XX} and \eqref{eq:1DSymAss2XX} follow by scaling back to $u$ into \eqref{eq:ConvToHalhPlaneSol} and \eqref{eq:1DSymAss2XXbis}.
\end{proof}
%
%
%
%
%


%
%
%
\section{Improvement of flatness}\label{Sec:ImpFlat}
This section is devoted to the proof of Theorem \ref{lem:ImpFlat3}. As mentioned in the introduction, its proof can be regarded as a suitable ``interpolation'' of  the methods by De Silva \cite{DeSilva2011:art} and Savin \cite{Savin2009:art}, and requires some auxiliary results: a uniform H\"older type estimate given in Lemma \ref{lem:ImpFlat1} and Lemma \ref{lem:ImpFlat2}, and a compactness result provided by  Lemma \ref{lem:ImpFlat2.1}. Further, we will crucially use the 1D solutions studied in Lemma \ref{lem:1DSolution} and their truncations (cf. Remark \ref{rem:Choicettau}).
\normalcolor

\begin{defn}\label{def:Fla1Flat2}
Let  $u_\vep$ be a critical point of \eqref{eq:Energy} in $B_R\subset \RR^N$.

\begin{itemize}
\item We say that $u_\vep$ satisfies ${\rm Flat_1}( \nu, \delta, R)$ if
\begin{equation}\label{flat1}
\begin{array}{rlll}
{} &\hspace{-5pt} u_\vep(x) - \nu\cdot x \leq \delta R \quad &\mbox{in }& B_R \cap \{u_\vep \geq \vartheta_1\vep\}\\
-\delta R \leq&\hspace{-5pt}  u_\vep(x) - \nu\cdot x \quad   &\mbox{in }&B_R. 
\end{array}
\end{equation}

\item We say that $u_\vep$ satisfies ${\rm Flat_2}( \nu, \delta, R)$ if
\begin{equation}\label{flat2}
 w_\vep^\vep(\nu\cdot x - \delta R) < u_\vep(x) <  w_\vep^{-\vep}(\nu\cdot x + \delta R) \quad \mbox{in } B_R.
\end{equation}
\end{itemize}
\end{defn}
\begin{lem}\label{lem:ImpFlat2tris} There exist $\vep_0,\dd_0 \in (0,1)$ depending only on $\vth_1$, $\vth_2$ and $c_1 > 0$ as \eqref{eq:AssumptionsPhi},   such that for every $R > 0$, every $\nu \in \Sf^{N-1}$, every $\vep/R \in (0,\vep_0)$, $\dd \in [0,\dd_0)$ and every critical point $u_\varepsilon$ of \eqref{eq:Energy} in $B_R$, we have 
\begin{equation}\label{eq:Flat12}
u_\vep \text{ satisfies } {\rm Flat_1}(\nu, \delta, R) \Rightarrow  u_\vep \text{ satisfies } {\rm Flat_2}(\nu, \delta+ \sqrt{\vep/R}, (1 - \sqrt{\vep/R})R),
\end{equation}
\begin{equation}\label{eq:Flat21}
 u_\vep \text{ satisfies } {\rm Flat_2}(\nu, \delta, R) \Rightarrow  u_\vep \text{ satisfies } {\rm Flat_1}(\nu, \delta+ \sqrt{\vep/R} , (1 - \sqrt{\vep/R})R).
\end{equation}
\end{lem}
\begin{proof} Let $\vep_0 \in (0,1)$ as in Lemma \ref{lem:1DSolutionCompxwvep}, $\vep \in (0,\vep_0)$, and set $U_\vep := \{u_\vep \geq \vartheta_1\vep\}$. By scaling, we may assume $R =1$ while, up to a rotation of the coordinate system, we can set $\nu = e_N$.

\smallskip

\emph{Step 1.} Let us prove first \eqref{eq:Flat12}. Assume that $u_\vep$ satisfies ${\rm Flat_1}(\nu, \delta, 1)$, as defined in \eqref{flat1}. 
On the one hand we have $u_\vep(x) \geq x_N -\delta$ in $B_1$. Then, by the first inequality in \eqref{eq:1DSolCompxwvep} with $\sigma = 1/2$, we have
\begin{equation}\label{eq:EqBoundsPlaneSolSub}
u_\vep(x) \geq x_N -\delta \geq w_\vep^\vep(x_N - \delta - \sqrt{\vep}) \quad \text{ in } B_1 \cap \{w_\vep^\vep(x_N - \delta - \sqrt{\vep}) > 0 \}.
\end{equation}
Further, since $u_\vep \geq 0$, the same inequality holds true in $B_1 \cap \{w_\vep^\vep(x_N - \delta - \sqrt{\vep}) = 0 \}$ and the first inequality in \eqref{eq:Flat12} follows.

To show the second inequality, we use that, on the other hand,  $u_\vep(x) \leq x_N +\delta$ in $B_1 \cap U_\vep$. Then, by the second inequality in \eqref{eq:1DSolCompxwvep1}, we have
\begin{equation}\label{eq:ImpFlat2trisBoundAbove}
u_\vep(x) \leq x_N +\delta \leq w_\vep^{-\vep}(x_N + \delta + \sqrt{\vep}) - \tfrac{\sqrt{\vep}}{2} \quad \text{ in } B_1 \cap U_\vep.
\end{equation}
Now notice that by Lemma \ref{lem:ExpDecayu} (cf. \eqref{eq:ExpDecayu}), we have
\begin{equation}\label{eq:ImpFlat3ExpBound11}
u_\vep \leq 3\vartheta_1 \varepsilon \, e^{-\frac{\vep^{-1/4}}{4 c_1^{1/2}} } \quad \text{ in } B_{1-\sqrt{\vep}} \setminus V_\vep, \qquad V_\vep := B_1\cap \{x : d(x,U_\vep) \leq \vep^{3/4}\}.
\end{equation}
Thanks to \eqref{eq:1DEvenSolProp} (with $|\tau| = \varepsilon$), we also know that $w_\vep^{-\varepsilon}(x_N + \delta + \sqrt{\vep}) \geq \tfrac{1}{\sqrt{c_1}} \varepsilon^{3/2}$ and thus by \eqref{eq:ImpFlat3ExpBound11}
\[
u_\vep < w_\vep^{-\varepsilon}(x_N + \delta + \sqrt{\vep}) \quad \text{ in } B_{1-\sqrt{\vep}} \setminus V_\vep,
\]
for every $\vep \in (0,\vep_0)$, taking eventually $\vep_0$ smaller. 
 We are left to check that  $u_\vep(x) <  w_\vep^{-\vep}( x_N + \delta+ \sqrt \vep)$ in $B_{1-\sqrt{\vep}} \cap (V_\vep \setminus U_\vep)$. Let $x\in B_{1-\sqrt{\vep}} \cap (V_\vep \setminus U_\vep)$. 
 Let $\bar x\in B_1\cap U_\vep$ such that  $|x-\bar x| \le  \vep^{3/4}$. From  \eqref{eq:ImpFlat2trisBoundAbove}  (using $u_\vep\ge0$) we know that $w_\vep^{-\vep}(\bar x_N + \delta + \sqrt{\vep}) \geq \sqrt{\vep}/2$. 
 Hence (using that $w_\vep^{-\vep}$ is 1-Lipschitz), 
 \[
 u_\vep(x) \le \vartheta_1 \vep  \le\sqrt \vep/2 -\vep^{3/4}\le  w_\vep^{-\varepsilon}(\bar x_N + \delta + \sqrt{\vep})- \vep^{3/4} \le w_\vep^{-\varepsilon}(x_N + \delta + \sqrt{\vep}).
 \]
This completes the proof of \eqref{eq:Flat12}.

\smallskip

\emph{Step 2.} Now we show \eqref{eq:Flat21}. Assume $u_\vep(x) > w_\vep^\vep(x_N - \delta)$ in $B_1$. Then, by the second inequality in \eqref{eq:1DSolCompxwvep} (with $\sigma = 1/2$), we obtain
\[
u_\vep(x) > w_\vep^\vep(x_N - \delta) > x_N -\dd - \tfrac{\sqrt{\vep}}{2} > x_N -\dd - \sqrt{\vep} \quad \text{ in } B_1,
\]
and the first inequality in \eqref{eq:Flat21} follows. On the other hand, if $u_\vep(x) < w_\vep^{-\vep}(x_N + \delta)$ in $B_1$, the first inequality in \eqref{eq:1DSolCompxwvep1} yields
\[
u_\vep(x) < w_\vep^{-\vep}(x_N + \delta) < x_N + \dd + \tfrac{\sqrt{\vep}}{2} < x_N + \dd + \sqrt{\vep} \quad \text{ in } B_1\cap\{x_N \geq y_\vep^{-\vep} - \delta\},
\]
where $y_\vep^{-\vep}$ is as in Lemma \ref{lem:1DSolutionCompxwvep}. Finally, since $u_\vep(x) \geq \vth_1\vep$ and the assumption imply $w_\vep^{-\vep}(x_N + \delta) \geq \vth_1\vep$, we deduce, by monotonicity, that $x_N + \dd \geq 0 \geq y_\vep^{-\vep}$ in  $U_\vep=\{u_\ep\ge \vth_1\vep\}$. Thus $B_1 \cap U_\vep \subset B_1\cap\{x_N \geq y_\vep^{-\vep} - \delta\}$ and the second inequality in \eqref{eq:Flat21} follows too.
\end{proof}
\begin{lem}\label{lem:ImpFlat1}
There exist $\delta_0,c_0 \in (0,1)$ and $\theta_0 \in (\tfrac{1}{2},1)$ depending only on $N$, $\vartheta_1$, $\vartheta_2$ and $c_1$ as in \eqref{eq:AssumptionsPhi} such that for every $R > 0$, every $\delta \in (0,\delta_0)$, every $a \in \mathbb{R}$ and $b \leq 0$ such that $a + |b| = \delta R$, every $\varepsilon/R \in (0,c_0 \delta)$ and every critical point $u_\varepsilon$ of \eqref{eq:Energy} in $B_R$ satisfying
\begin{equation}\label{eq:FlatCondTruncSol1D}
\begin{aligned}
&w_\varepsilon^\varepsilon(x_N - a) \leq u_\varepsilon(x) \leq w_\varepsilon^{-\varepsilon}(x_N - b) \quad \text{ in } B_R, \\
&u_\vep(0) \in [\vartheta_1\vep, \vartheta_2\vep] , 
\end{aligned}
\end{equation}
where $w_\varepsilon^\varepsilon$ and $w_\varepsilon^{-\varepsilon}$ are as in Remark \ref{rem:Choicettau} with $w_\varepsilon^\varepsilon(0) = w_\varepsilon^{-\varepsilon}(0) = \vartheta_1\varepsilon$, then there exist $a' \in \mathbb{R}$, $b'\leq 0$ such that 
\begin{equation}\label{eq:ImpFlat1ImpClos}
\begin{aligned}
&w_\varepsilon^\varepsilon(x_N - a') \leq u_\varepsilon(x) \leq w_\varepsilon^{-\varepsilon}(x_N - b') \quad \text{ in } B_{R/4}, \\
&b \leq b' \leq a' \leq a,  \\
&a' + |b'| \leq \theta_0(a + |b|).
\end{aligned}
\end{equation}
\end{lem}
\begin{proof} By scaling, we may assume $R = 1$. Set $u = u_\varepsilon$, $w^\varepsilon = w_\varepsilon^\varepsilon$, $w^{-\varepsilon} = w_\varepsilon^{-\varepsilon}$, and define
\[
w^{\varepsilon,a}(x_N) := w^\varepsilon(x_N - a), \qquad w^{-\varepsilon,b}(x_N) := w^{-\varepsilon}(x_N - b).
\]
Notice that, up to replace $\delta$ with $\delta + 1/j$ and then taking the limit as $j \to +\infty$, we may assume
\[
w^{\varepsilon,a} < u < w^{-\varepsilon,b} \quad \text{ in } B_1,
\]
and, since $0 \in \{\vth_1\vep \leq u_\vep \leq \vth_2\vep\}$, we also have
\begin{equation}\label{eq:ImpFlat1Boundab}
a \geq -c\varepsilon, \qquad |b| \leq \delta + c\varepsilon,
\end{equation}
where $c > 1$ is as in Lemma \ref{lem:1DSolutionBoundFB}. This can be easily verified since $w^\varepsilon(0) = \vartheta_1\varepsilon$ and $\{\vartheta_1\varepsilon \leq w^\varepsilon \leq \vartheta_2\varepsilon \} \subset \{|x_N| \leq c\varepsilon\}$ by Lemma \ref{lem:1DSolutionBoundFB}.

We define
\begin{equation}\label{eq:ImpFlat1Conddelta0}
\delta_0 := \tfrac{1}{32}, \qquad c_0 := \tfrac{1}{16c}, \qquad \theta_0 := 1 - c_N,
\end{equation}
where $c > 1$ is as in Lemma \ref{lem:1DSolutionBoundFB} respectively (depending only on $\vartheta_1$, $\vartheta_2$ and $c_1 > 0$ in \eqref{eq:AssumptionsPhi}), and $c_N \in (0,1)$ is the dimensional constant appearing in \eqref{eq:ImpFlat1IneqBelow1} (notice that we may assume $\theta_0 > \tfrac{1}{2}$ taking eventually $c_N$ smaller). In particular, since $\delta \in (0,\delta_0)$, we have
\begin{equation}\label{eq:ImpFlat1BoundGamma}
\{\varepsilon \vartheta_1 \leq u \leq \varepsilon\vartheta_2 \} \subset \left\{|x_N| < \tfrac{1}{32} \right\}.
\end{equation}

\smallskip 

Fix $y = (y',y_N) = (0,\tfrac{1}{8})$. We consider the following alternative. Either: 
\begin{enumerate}
\item[(a)]  \quad $w^{-\varepsilon,b}(y) - u(y) \le u(y) - w^{\varepsilon,a}(y)$
\end{enumerate}
or 
\begin{enumerate}
\item[(b)]  \quad $w^{-\varepsilon,b}(y) - u(y) \ge u(y) - w^{\varepsilon,a}(y)$
\end{enumerate}

\smallskip

\emph{First case.} Assume (a) holds. We first prove that
\begin{equation}\label{eq:ImpFlat1DeltaLift}
u > w^{\varepsilon,a - c_N\delta}  \quad \text{ in } B_{15/16} \cap \{|x_N| \geq \tfrac{1}{16}\},
\end{equation}
for some $c_N \in (0,1)$. Let $v := u - w^{\varepsilon,a}$. In view of \eqref{eq:ImpFlat1BoundGamma}, $v$ is harmonic and positive in $B_1 \cap \{|x_N| > \tfrac{1}{32}\}$ and so, by the Harnack inequality, it follows
\[
\inf_{B_{15/16} \cap \{|x_N| \geq 1/16\}} v \geq 4c_N v(y) \geq 2c_N [w^{-\varepsilon,b}(y) - w^{\varepsilon,a}(y)] \geq c_N\delta,
\]
for some $c_N > 0$. To justify the last inequality we proceed as follows. If $\tilde{a} > a$ and $\tilde{b} > b$ are such that $w^{\varepsilon,\delta}(\tilde{a}) = w^{-\varepsilon,-\delta}(\tilde{b}) = \vartheta_2\varepsilon$, then $|\tilde{a} - a| \leq c\varepsilon$, $|\tilde{b} - b| \leq c\varepsilon$ where $c > 0$ is the constant appearing in the statement of Lemma \ref{lem:1DSolutionBoundFB}. Consequently, since $w^{\varepsilon,a}(y) = \vartheta_2\varepsilon + (1+\varepsilon)(y_N - \tilde{a})$, $w^{-\varepsilon,b}(y) = \vartheta_2\varepsilon + (1- \varepsilon)(y_N - \tilde{b})$ and $c > 1$, we find
\[
\begin{aligned}
w^{-\varepsilon,b}(y) - w^{\varepsilon,a}(y) &= (1-\varepsilon)(y_N - \tilde{b}) - (1+\varepsilon)(y_N - \tilde{a}) = \tilde{a} - \tilde{b} - \tfrac{\varepsilon}{4} + \varepsilon(\tilde{a} + \tilde{b}) \\
& \geq \delta - 2c\varepsilon - \tfrac{\varepsilon}{4} + \varepsilon(a+b -2c\varepsilon) \geq \delta - 6c\varepsilon - \varepsilon - \varepsilon\delta,
\end{aligned}
\]
thanks to \eqref{eq:ImpFlat1Boundab}. Further, recalling that $\varepsilon < c_0\delta$ by assumption, it follows
\[
w^{-\varepsilon,b}(y) - w^{\varepsilon,a}(y) \geq (1 - 8cc_0) \delta > \tfrac{1}{2}\delta,
\]
in view of the definition of $c_0$ in \eqref{eq:ImpFlat1Conddelta0}. As a consequence, $u \geq w^{\varepsilon,a} + c_N\delta$ in $B_{15/16} \cap \{|x_N| \geq \tfrac{1}{16}\}$ and thus, using that $w^{\varepsilon,a}$ is a line with slope $1+\varepsilon$ in $\{w^{\varepsilon,a} > \vartheta_2\varepsilon\}$ and $\varepsilon < 1$, we deduce \eqref{eq:ImpFlat1DeltaLift}.

The second step is to show
\begin{equation}\label{eq:ImpFlat1IneqBelow1}
u \geq w^{\varepsilon,a - c_N\delta} \quad \text{ in } B_{1/4},
\end{equation}
for some new $c_N \in (0,1)$. If \eqref{eq:ImpFlat1IneqBelow1} holds true, then \eqref{eq:ImpFlat1ImpClos} follows by setting $a' = a - c_N\delta$, $b' = b$, in view of the definition of $\theta_0$.

To prove \eqref{eq:ImpFlat1IneqBelow1} we use a sliding argument: given any smooth, nonnegative and bounded $h$, we define the family of functions
\[
v_\lambda(x) := w^{\varepsilon,a}(x_N + \lambda h(x)), \quad x \in B_1, \quad \lambda \in [0,c_N\delta].
\]
Notice that $v_0 = w^{\varepsilon,a}$. Using the equation of $w^\varepsilon$, it is not difficult to check that
\begin{equation}\label{eq:ImpFlat1Eqvt}
\Delta v_\lambda = \tfrac{1}{2} \Phi_\varepsilon'(v_\lambda) \left(1 + 2\lambda \partial_N h + \lambda^2|\nabla h|^2 \right) + \lambda \dot w^\varepsilon \Delta h,
\end{equation}
where $\partial_N := \partial_{x_N}$. We choose $h(x) := \tilde h(x-y)$, where $\tilde{h}$ is the unique radially decreasing harmonic function in $B_{1/2} \setminus B_{1/32}$ satisfying $\tilde{h} = 1$ in $\overline{B}_{1/32}$ and $\tilde{h} = 0$ in $\mathbb{R}^N\setminus B_{1/2}$. Consequently,
\begin{equation}\label{eq:ImpFlat1Ineqvt}
\Delta v_\lambda > \tfrac{1}{2} \Phi_\varepsilon'(v_\lambda) \quad \text{ in } D := B_{1/2}(y) \cap \{x_N < \tfrac{1}{16}\},
\end{equation}
for every $\lambda \in (0,c_N\delta]$. This follows neglecting the nonnegative terms in \eqref{eq:ImpFlat1Eqvt} and noticing that $\partial_N h > 0$ in $D$ by construction. On the other hand,
\begin{equation}\label{eq:ImpFlat1BoundDvt}
v_\lambda < u \quad \text{ in } \partial D,
\end{equation}
for every $\lambda \in [0,c_N\delta]$. Indeed, recalling that $h = 0$ in $\partial B_{1/2}(y)$ it follows $v_\lambda = w^{\varepsilon,a} < u$ in $\partial D \cap \{ x_N < \tfrac{1}{16} \}$ while, since $h \leq 1$ and $\lambda \leq c_N\delta$, we have $v_\lambda \leq w^{\varepsilon,a - c_N\delta}$ and so $v_\lambda < u$ in $\partial D \cap \{ x_N = \tfrac{1}{16} \}$ in view of \eqref{eq:ImpFlat1DeltaLift}. Now, we define
\[
\lambda_\ast := \max\{\lambda \in [0,c_N\delta]: v_\lambda \leq u \text{ in } D \},
\]
and show that $\lambda_\ast = c_N\delta$. If this is not true, there must be $\lambda \in [0,c_N\delta)$ and $x_\lambda \in \overline{D}$ such that $v_\lambda \leq u$ in $D$, with $v_\lambda(x_\lambda) = u(x_\lambda)$. Recalling that $v_0 = w^{\varepsilon,a}$ and that $w^{\varepsilon,a} < u$ by assumption, we immediately see that $\lambda > 0$ and, by \eqref{eq:ImpFlat1BoundDvt}, it must be $x_\lambda \in D$. Thus, using the equation of $u$ (or equivalently $u$) and \eqref{eq:ImpFlat1Ineqvt}, we obtain that the function $\tilde v_\lambda := u - v_\lambda$ satisfies
\[
\begin{cases}
\tilde v_\lambda \geq 0 \quad \text{ in } D \\
\tilde v_\lambda(x_\lambda) = 0, \; \Delta \tilde v_\lambda(x_\lambda) < 0,
\end{cases}
\]
which leads to a contradiction since $x_\lambda \in D$ is a minimum point for $\tilde v_\lambda$. Combining \eqref{eq:ImpFlat1DeltaLift} with $\lambda_\ast = c_N\delta$, and noticing that $B_{1/4} \subset B_{1/2}(y)$, we deduce
\[
u(x) \geq w^\varepsilon(x_N - a + c_N\delta h(x)) \quad \text{ in } B_{1/4},
\]
and thus, since $h \geq c_N$ in $B_{1/4}$ for some new constant $c_N > 0$ by construction, the monotonicity of $w^\varepsilon$ yields \eqref{eq:ImpFlat1IneqBelow1}.

\smallskip

\emph{Second case.} Assume now that (b) holds. In this case, following the proof of  \eqref{eq:ImpFlat1DeltaLift}, we find
\[
u < w^{-\varepsilon, b + c_N\delta}  \quad \text{ in } B_{15/16} \cap \{|x_N| \geq \tfrac{1}{16}\},
\]
where $c_N > 0$ can be taken as in \eqref{eq:ImpFlat1DeltaLift}. So, following the ideas of \emph{Step 1}, we must prove
\begin{equation}\label{eq:ImpFlat1IneqBelow12}
u \leq w^{-\varepsilon,b + c_N\delta} \quad \text{ in } B_{1/4},
\end{equation}
where $c_N \in (0,1)$ is as in \eqref{eq:ImpFlat1IneqBelow1}. As above, \eqref{eq:ImpFlat1IneqBelow12} implies \eqref{eq:ImpFlat1ImpClos} taking $a' = a$ and $b' = b + c_N\delta$.

To do so, we consider
\[
v_\lambda(x) := w^{-\varepsilon,a}(x_N + \lambda h(x)), \quad x \in B_1, \quad \lambda \in [-c_N\delta,0],
\]
where $h$ is as in Step 1 (note however that now $\lambda<0$).
Using \eqref{eq:ImpFlat1Eqvt}, we deduce $\Delta v_\lambda < \frac{1}{2} \Phi_\varepsilon'(v_\lambda)$ in $D$, for every $\lambda \in [-c_N\delta,0)$. To see this, it is enough to notice that
\[
2\partial_N h + \lambda |\nabla h|^2 \geq \overline{c}_N + \lambda |\nabla h|^2 > 0 \quad \text{ in } D,
\]
for some small $\overline{c}_N > 0$, if $|\lambda|$ is small enough and so, choosing eventually $\delta_0$ smaller (depending only on $N$), the above inequality is satisfied for $\lambda \in [-c_N\delta,0)$. Proceeding exactly as above, we find
\[
\lambda_\ast := \min\{\lambda \in [-c_N\delta,0]: v_\lambda \geq u \text{ in } D \} = -c_N\delta,
\]
and we are led to
\[
u(x) \leq w^{-\varepsilon}(x_N - b - c_N\delta) \quad \text{ in } B_{1/4},
\]
for some new $c_N > 0$, which is \eqref{eq:ImpFlat1IneqBelow12}.
\end{proof}
%


%
%
%
%
\begin{lem}\label{lem:ImpFlat2}
 There exist $\alpha,\tilde{\dd}_0 \in (0,1)$ and $C > 0$ depending only on $N$, $\vartheta_1$, $\vartheta_2$ and $c_1$ as in \eqref{eq:AssumptionsPhi} such that for every $\delta \in (0,\tilde{\delta}_0)$, every $\varepsilon \in (0,\delta^2)$ and every critical point $u_\varepsilon$ of \eqref{eq:Energy} in $B_1$ satisfying
\begin{equation}\label{eq:FlatCondB1}
\begin{array}{rlll}
{} &\hspace{-5pt} u_\vep(x) - x_N \leq \delta \quad &\mbox{in }& B_1 \cap \{u_\vep \geq \vartheta_1\vep\}\\
-\delta \leq&\hspace{-5pt}  u_\vep(x) - x_N \quad   &\mbox{in }&B_1, 
\end{array}
\end{equation}
with $u_\vep(0) \in [\vartheta_1\vep, \vartheta_2\vep]$, then the function
\[
v_{\varepsilon,\delta}(x) := \frac{u_\varepsilon(x) - x_N}{\delta}
\]
satisfies
\begin{equation}\label{eq:ImpFlat2CalphaBoundproof}
\begin{array}{rlll}
{} &\hspace{-5pt} v_{\vep,\dd}(x) - v_{\vep,\dd}(z) \leq  \omega_\dd(x-z) \quad &\mbox{in }& B_{1/2} \cap \{u_\vep \geq \vartheta_1\vep\}\\
- \omega_\dd(x-z) \leq&\hspace{-5pt}  v_{\vep,\dd}(x) - v_{\vep,\dd}(z) \quad   &\mbox{in }&B_{1/2}, 
\end{array}
\end{equation}
for every $z \in B_{1/2} \cap \{u_\vep \geq \vartheta_1\vep\}$, where 
\[
\omega_\dd(y) := C(\dd + |y|)^{\alpha}.
\]
\end{lem}
\begin{proof} Let $\delta_0,\theta_0 \in (0,1)$ and $c_0 \in (0,1)$ as in Lemma \ref{lem:ImpFlat1}, and $\vep_0 \in (0,1)$ as in Lemma \ref{lem:ImpFlat2tris}. We set 
\[
\tilde{\dd}_0 := \min\{\dd_0/4,\sqrt{\vep_0/4},c_0/4\},
\]
and take $\dd \in (0,\tilde{\dd}_0)$, $\varepsilon \in (0,\delta^2)$. Notice that the definition of $\tilde{\dd}_0$ guarantees $4\dd < \dd_0$ and $4\vep < \vep_0$. For simplicity we also set $u = u_\varepsilon$, $w^\varepsilon = w_\varepsilon^\varepsilon$ and $w^{-\varepsilon} = w_\varepsilon^{-\varepsilon}$, and define
\begin{equation}\label{eq:ImpFlat2DefD1Ep1}
\kappa := \frac{1}{\tilde{\dd}_0}, \qquad 0 < \alpha < |\log_4(\theta_0)|, \qquad C \geq 4^{1+2\alpha}.
\end{equation}

\smallskip

\emph{Step 1.} We first prove that \eqref{eq:ImpFlat2CalphaBoundproof} holds true for every $z \in \{\vth_1\vep \leq u_\vep \leq \vth_2\vep\} \cap B_{1/2}$. Let us set 
\[
\tilde{\dd} := 2(\dd + \sqrt{\vep/2}).
\]
Notice that $\vep < \dd^2$ implies $\tilde{\dd} \leq 4\dd$ and thus, since $\tilde{\dd} \geq 2\dd$ by definition, it is equivalent to work with $\tilde{\dd}$ instead of $\dd$, which is what we will do from now on.

So, we fix $j \in \NN$ such that
\begin{equation}\label{eq:BoundDeltaTil}
4^{-j-2} \leq \frac{\tilde{\delta}}{4\tilde{\dd}_0} = \frac{\kappa}{4} \tilde{\delta} < 4^{-j-1},
\end{equation}
and we use the definition of $\tilde{\dd}$ to combine \eqref{eq:FlatCondB1} and \eqref{eq:Flat12}, which yield
\[
 w^\vep(x_N - \tilde{\delta}) < u(x) <  w^{-\vep}(x_N + \tilde{\delta}) \quad \text{ in } B_{3/4}.
\]
Now, in view of \eqref{eq:BoundDeltaTil}, we have $\tilde{\dd} \leq 4^{-j}\tilde{\delta}_0$ and, since $\vep < \dd^2$, we also have $\vep \leq \tilde{\dd}^2 \leq c_0\tilde\dd$ and so we may apply Lemma \ref{lem:ImpFlat1}  (rescaled and translated from $B_1$ to $B_{1/4}(z)$, i.e. applied to the function $u(z+ 4\,\cdot\,)$) iteratively on $B_{4^{-k}}(z)$ for $1\le k \leq j$, deducing the existence of $a_k$ and $b_k$ (with $a_0 = -b_0 = \tilde{\dd}$) for which
\begin{equation}\label{eq:ImpFlat2IterFB}
\begin{aligned}
&w^\varepsilon(x_N - z_N - a_k) \leq u(x) \leq w^{-\varepsilon}(x_N - z_N - b_k) \quad \text{ in } B_{4^{-k}}(z), \\
&0 < a_k + |b_k| \leq 4\theta_0^k \tilde{\delta}.
\end{aligned}
\end{equation}
Then, applying \eqref{eq:Flat21} to \eqref{eq:ImpFlat2IterFB} (choosing $R = 4^{-k}$ and $\dd = (a_k + |b_k|)4^k$) and recalling that $\theta_0 \in (\tfrac{1}{2},1)$, it follows
\begin{equation}\label{eq:ImpFlat2BounduMinusx}
\begin{array}{rlll}
{} &\hspace{-5pt} u(x) - (x_N-z_N) \leq 2\theta_0^k\tilde{\dd} \quad &\mbox{in }& B_{4^{-k}/2}(z) \cap \{u_\vep \geq \vartheta_1\vep\}\\
-2\theta_0^k\tilde{\dd} \leq&\hspace{-5pt}  u(x) - (x_N-z_N) \quad   &\mbox{in }& B_{4^{-k}/2}(z), 
\end{array}
\end{equation}
for all  $1\le k \leq j$ (notice that since $\vep < \dd^2 < \tilde{\dd}^2$ and $\tilde{\dd}_0 < \sqrt{\vep_0}$ we automatically have $\vep < \vep_0 4^{-k}$, for every $k \leq j$).

Now, assume that $|x-z| \geq \kappa\tilde{\dd}$. Then $4^{-j+n-2} \leq |z-x| < 4^{-j+n-1}$, for some $0 \leq n \leq j$ ($n \in \NN$) by the definition of $j$. Applying \eqref{eq:ImpFlat2BounduMinusx} with $k=j-n$, we find
\[
u(x) - x_N - (u(z) - z_N) = u(x) - (x_N - z_N) - u(z) \;\;
\begin{cases}
\leq  4\theta_0^{j-n} \tilde{\dd} \quad\; \text{ if } x \in \{u_\vep \geq \vartheta_1\vep\} \\
\geq -4\theta_0^{j-n} \tilde{\dd},
\end{cases}
\]
and thus
\begin{equation}\label{eq:ImpFlat2Bound1vSS}
\begin{array}{rlll}
{} &\hspace{-5pt} v(x) - v(z) \leq  4\theta_0^{j-n} \quad \mbox{if } x \in  \{u_\vep \geq \vartheta_1\vep\}\\
- 4\theta_0^{j-n} \leq&\hspace{-5pt}  v(x) - v(z),
\end{array}
\end{equation}
where we have set $v := v_{\vep,\dd}$ for simplicity. Using the definitions of $\alpha$ and $C$ in \eqref{eq:ImpFlat2DefD1Ep1} and that $|x-z| \geq 4^{-j+n-2}$, we have $4\theta_0^{j-n} \leq C|x-z|^\alpha$ and \eqref{eq:ImpFlat2CalphaBoundproof} follows.

If $|x-z| \leq \kappa\tilde{\dd}$, then, proceeding as above, we find that \eqref{eq:ImpFlat2Bound1vSS} holds true with $n=0$ and so, since $4\theta_0^j \leq C (4^{-\alpha})^{j+2}$ for $\alpha$ and $C$ as in \eqref{eq:ImpFlat2DefD1Ep1}, and $\kappa\tilde{\dd} \geq 4^{-j-2}$, we deduce 
\[
\begin{array}{rlll}
{} &\hspace{-5pt} v(x) - v(z) \leq  C(\kappa\tilde{\dd})^\alpha \quad \mbox{if } x \in  \{u_\vep \geq \vartheta_1\vep\}\\
- C(\kappa\tilde{\dd})^\alpha \leq&\hspace{-5pt}  v(x) - v(z),
\end{array}
\]
and \eqref{eq:ImpFlat2CalphaBoundproof} follows.

\smallskip

\emph{Step 2.} Now we consider the case $x,z \in \{u \geq \vartheta_2\varepsilon \}\cap B_{1/4}$. We fix $x_0 \in \partial \{u > \vep\vth_2\}\cap B_{1/4}$ such that $|x-x_0| = \dist(x,\partial \{u > \vep\vth_2\}) := d(x)$.

Set $d := d(x)$ and assume first $\kappa\delta \vee |x-z| < d/4$. In this case, using \emph{Step 1}, we easily obtain
\[
|v(\xi) - v(x_0)| \leq C(\kappa \delta \vee |\xi - x_0|)^\alpha \leq 2^\alpha C \, d^\alpha, \quad \forall \xi \in B_{d/2}(x).
\]
So, since $v$ is harmonic in $B_d(x)$, we have
\[
\sup_{B_{d/4}(x)} |\nabla v| \leq c_N \frac{\osc_{B_{d/2}(x)} v }{d} \leq c_N C_\alpha d^{\alpha - 1},
\]
for some $C_\alpha > 0$ and thus $|v(x) - v(z)| \leq c_N C_\alpha d^{\alpha - 1} |x-z| \leq C_\alpha (\kappa \delta \vee |x - z|)^\alpha$ for some new $C_\alpha > 0$.

On the other hand, if $\kappa\delta \vee |x-z| \geq d/4$, we may apply the estimate of \emph{Step 1} twice to obtain
\[
\begin{aligned}
|v(x) - v(z)| &\leq |v(x) - v(x_0)| + |v(x_0) - v(z)| \\
&\leq C \left[ (\kappa \delta \vee |x - x_0|)^\alpha + (\kappa \delta \vee |z - x_0|)^\alpha \right]\\
&\leq C \left\{ [\kappa \delta \vee d(x)]^\alpha + \left[ \kappa \delta \vee (|x - z| + d(x))\right]^\alpha \right\} \leq C_\alpha (\kappa\delta \vee |x-z|)^\alpha,
\end{aligned}
\]
for some $C,C_\alpha > 0$ and our statement follows.

\smallskip

\emph{Step 3.} 
If $x \in \{u \leq \vth_1\vep \}$ and $z \in \{u > \vth_2\vep \}$ then there exists $\bar z \in  \{ \vth_1\vep \le u_\vep \le \vth_2\vep \}$ which belongs to the segment $xz$.
Hence, using the previous steps
\[
\begin{aligned}
v(x)-v(z) &\ge v(x) - v(\bar z) - |v(\bar z) - v(z)|  \\ 
&\ge -C(\kappa \delta \vee |x - \bar z|)^\alpha -C(\kappa \delta \vee |x - \bar z|)^\alpha  \ge -C(\kappa \delta \vee |x - z|)^\alpha, 
\end{aligned}
\]
and the proof of \eqref{eq:ImpFlat2CalphaBoundproof} is complete.
\end{proof}
\begin{lem}\label{lem:ImpFlat2.1}
There exists a H\"older continuous function $v : \{x_N \geq 0\}\cap \overline{B_{1/4}} \to \mathbb{R}$, harmonic in $\{x_N > 0\}\cap B_{1/4}$ and with $\|v\|_{L^\infty}=1$ such that for every sequence $\delta_j \to 0^+$, every $\varepsilon_j \in (0,\delta_j^2)$ and every critical point $u_{\varepsilon_j}$ of \eqref{eq:Energy} in $B_2$ satisfying
\begin{equation}\label{eq:FlatCondB12}
\begin{array}{rlll}
{} &\hspace{-5pt} u_{\vep_j}(x) - x_N \leq \delta_j \quad &\mbox{in }& B_1 \cap \{u_{\vep_j} \geq \vartheta_1\vep_j \}\\
-\delta_j \leq&\hspace{-5pt}  u_{\vep_j}(x) - x_N \quad   &\mbox{in }&B_1, 
\end{array}
\end{equation}
with $u_{\vep_j}(0) \in [\vartheta_1\vep_j, \vartheta_2\vep_j]$, then, setting
\[
v_j(x) := \frac{u_{\varepsilon_j}(x) - x_N}{\delta_j},
\]
the sequence of graphs
\begin{equation}\label{eq:ImpFlat2.1Gj}
G_j = \left\{ (x,v_j(x)): x \in \{u_{\varepsilon_j} \geq \vartheta_1\varepsilon_j\}\cap B_{1/4} \right\}
\end{equation}
converge in the Hausdorff distance in $\mathbb{R}^{N+1}$ to 
\begin{equation}\label{eq:ImpFlat2.1G}
G = \left\{ (x,v(x)): x \in \{x_N \geq 0\} \cap B_{1/4} \right\},
\end{equation}
as $j \to + \infty$, up to passing to a suitable subsequence.
\end{lem}
\begin{proof} Let $\alpha \in (0,1)$ and $\kappa,C > 0$ as in Lemma \ref{lem:ImpFlat2}. Let $\delta_j \to 0^+$, $\varepsilon_j \in (0,\delta_j^2)$ and set $U_j := \{u_{\varepsilon_j} > \vartheta_1\varepsilon_j\}\cap B_{1/4}$, $H := \{x_N > 0\} \cap B_{1/4}$.

\smallskip

\emph{Step 1: Compactness.} We show that there is $v$ (harmonic in $H$ and $\alpha$-H\"older in $\overline{H}$ and with $L^\infty$ norm bounded by 1) such that for every $\sigma \in (0,1/4)$,
\begin{equation}\label{eq:UnifConvvjv}
\|v_j - v\|_{L^\infty(H_\sigma)} \to 0,
\end{equation}
as $j \to + \infty$, up to passing to a suitable subsequence, where $H_\sigma := \{x_N > \sigma\} \cap B_{1/4}$.

By \eqref{eq:FlatCondB12}, there is $j_\sigma \in \NN$, such that $H_{\sigma/2} \subset U_j$  and $\|v_j\|_{L^\infty(H_\sigma)} \leq 1$ (this follows  \eqref{eq:FlatCondB12} by $\delta_j$) and every $j \geq j_\sigma$. In addition, $v_j$ is harmonic in $U_\sigma$ and thus, by standard elliptic estimates and a diagonal procedure, there exists a harmonic function $v$ in $H$ such that $v_j \to v$ locally uniformly in $H$, up to passing to a suitable subsequence. On the other hand, by \eqref{eq:ImpFlat2CalphaBoundproof}, we have
\[
|v_j(x) - v_j(y)| \leq C (\delta_j + |x-y|)^\alpha,
\]
for every $x,y \in \overline{U}_j$, and thus, passing to the limit as $j \to +\infty$, we obtain that $v$ can be continuously extended up to $\partial U$ and $v \in C^\alpha(\overline{H})$ with $\|v\|_{L^\infty(H)}\le 1$.

\smallskip

\emph{Step 2: Convergence of graphs.} Fix $\sigma \in (0,\tfrac{1}{4})$, $x \in \overline{H}$, $p := (x,v(x)) \in G$ and set $q := (y,v(y))$, where $y \in H_{\sigma/2}$ is taken such that $|x-y| \leq \sigma$. Then, by the $C^\alpha$ estimate proved above, we obtain
\[
|p-q|^2 = |x-y|^2 + |v(x) - v(y)|^2 \leq \sigma^2 + C^2\sigma^{2\alpha} \leq C^2\sigma^{2\alpha},
\]
for some new $C > 0$. Now, if $j$ is large enough, we have $H_{\sigma/2} \subset U_j$ and so
\[
\dist (q,G_j)^2 = \inf_{y' \in \overline{U}_j} |y-y'|^2 + |v(y) - v_j(y')|^2 \leq |v(y) - v_j(y)|^2 \leq \|v - v_j\|_{L^\infty(U_{\sigma/2})}^2,
\]
from which we deduce
\begin{equation}\label{eq:ImpFltHausdorff}
\dist (p,G_j) \leq |p-q| + \dist (q,G_j) \leq C\sigma^\alpha + \|v - v_j\|_{L^\infty(U_{\sigma/2})} \leq C\sigma^\alpha,
\end{equation}
for some new $C > 0$, for every $j$ large enough, in view of \eqref{eq:UnifConvvjv}.

On the other hand, given any sequence $p_j = (x_j,v_j(x_j)) \in G_j$,  if $j$ is large enough we may take $y_j \in H_{\sigma/2}$ such that $\tfrac{\sigma}{2} \leq |x_j - y_j| \leq \sigma$ with $j$ such that $ \delta_j \leq \tfrac{\sigma}{2}$. Consequently, setting $q_j = (y_j,v_j(y_j))$, we have by \eqref{eq:ImpFlat2CalphaBoundproof}
\[
|p_j - q_j|^2 = |x_j - y_j|^2 + |v_j(x_j) - v_j(y_j)|^2 \leq \sigma^2 + C^2\sigma^{2\alpha} \leq C^2 \sigma^{2\alpha}.
\]
Further, as above
\[
\dist(q_j,G) = \inf_{y' \in \overline{H}} |y_j-y'|^2 + |v_j(y_j) - v(y')|^2 \leq |v_j(y_j) - v(y_j)|^2 \leq \|v - v_j\|_{L^\infty(U_{\sigma/2})},
\]
and thus, by \eqref{eq:UnifConvvjv},
\begin{equation}\label{eq:ImpFltHausdorffbis}
\dist(p_j,G) \leq C\sigma^\alpha + \|v - v_j\|_{L^\infty(U_{\sigma/2})} \leq C\sigma^\alpha,
\end{equation}
for $j$ large enough. Since $p$, $p_j$ and $\sigma > 0$ are arbitrary, the thesis follows by \eqref{eq:ImpFltHausdorff} and \eqref{eq:ImpFltHausdorffbis}.
\end{proof}
%
%
%
%
%

%
%
%
%
\begin{proof}[Proof of Theorem \ref{lem:ImpFlat3}] By scaling, we may assume $R = 1$. Assume by contradiction that there are $\gamma \in (0,1)$ and a sequence $\delta_j \to 0^+$ such that for every $\varrho_0 \in (0,1)$, there is $\varepsilon_j \in (0,\delta_j^2)$, a solution $u_j := u_{\varepsilon_j}$ to \eqref{eq:EulerEq} in $B_1$ satisfying
\begin{equation}\label{eq:ImpFlat3AssSeqj}
\begin{array}{rlll}
{} &\hspace{-5pt} u_{\vep_j}(x) - x_N \leq \delta_j \quad &\mbox{in }& B_1 \cap \{u_{\vep_j} \geq \vartheta_1\vep_j \}\\
-\delta_j \leq&\hspace{-5pt}  u_{\vep_j}(x) - x_N \quad   &\mbox{in }&B_1, 
\end{array}
\end{equation}
with $u_{\vep_j}(0) \in [\vartheta_1\vep_j, \vartheta_2\vep_j]$, such that for every $\nu \in \Sf^{N-1}$, either
\begin{equation}\label{eq:ImpFlatFinalThesis1}
\begin{array}{rlll}
{} &\hspace{-5pt} u_{\vep_j}(x) - \nu\cdot x \leq \dd_j \varrho_0^{1+\gamma} \quad &\mbox{in }& B_{\varrho_0 } \cap \{u_{\vep_j} \geq \vartheta_1\vep_j\}\\
-\dd_j  \varrho_0^{1+\gamma} \leq&\hspace{-5pt}  u_{\vep_j}(x) -  \nu\cdot x \quad   &\mbox{in }&B_{\varrho_0} \cap \{u_{\vep_j} \geq 0\}
\end{array}
\end{equation}
or 
\begin{equation}\label{eq:ImpFlatFinalThesis2}
|\nu - e_N| \leq  \sqrt 2  n  \dd_j
\end{equation}
fails for $j \in \NN$ large enough.

\emph{Step 1: Compactness.} By Lemma \ref{lem:ImpFlat2.1}, we have that the sequence
\[
v_j := v_{\vep_j,\delta_j} = \frac{u_j - x_N}{\delta_j}
\]
converge uniformly on compact sets of $U := \{ x_N > 0 \} \cap B_{1/4}$ to some limit function $v \in C^\alpha(\overline{U})$ which is harmonic in $U$ and, further, the sequence of graphs $G_j$ defined in \eqref{eq:ImpFlat2.1Gj} converge in the Hausdorff distance in $\mathbb{R}^{N+1}$ to the graph $G$ defined in \eqref{eq:ImpFlat2.1G}. In addition, since $0 \in \{\vth_1\vep_j \leq u_j \leq \vth_2\vep_j\}$ and $\varepsilon_j \in (0,\delta_j^2)$, then
\[
0 < v_j(0) \leq \vartheta_2 \delta_j^2,
\]
for every $j$, and thus $v(0) = 0$. Before moving forward, we define the even reflection of $v$ w.r.t. the hyperplane $\{x_N = 0\}$
\[
\tilde{v}(x) =
\begin{cases}
v(x',x_N) \quad &\text{ in } x_N \geq 0 \\
v(x',-x_N) \quad &\text{ in } x_N < 0,
\end{cases}
\]
defined in the whole $B_{1/4}$ and satisfying $\tilde{v} \in C^\alpha(B_{1/4})$.

\smallskip

\emph{Step 2.} In this step we prove that $\partial_N \tilde{v} \leq 0$ in $\{x_N = 0\}$ in the viscosity sense, that is for every $\varphi \in C^\infty(B_1)$ such that $\varphi \leq \tilde{v}$ in $B_{1/4}$ with equality only at some $z \in \{x_N = 0\}\cap B_{1/4}$, then $\partial_N \varphi(z) \leq 0$.

By contradiction, we assume there is $\varphi \in C^\infty(B_1)$ and $z \in \{x_N = 0\}\cap B_{1/4}$ as above, with $\partial_N \varphi(z) > 0$. For simplicity, we assume $z=0$, $\varphi(z) = 0$ (the same proof work in the general case with minor modifications). In addition, we may take $\varphi$ to be a polynomial of degree $2$ (cf. \cite[Chapter 2]{CaffCabre2020:book}) with the form
\begin{equation}\label{eq:ImpFlat3ExprPhi}
\varphi(x) = m x_N + m'\cdot x' + x^T \cdot M\cdot x, \quad x \in B_r
\end{equation}
for some vector $(m',m)$ with $m > 0$, some matrix $M \in \RR^{N,N}$ with $\text{tr}(M) = 0$ and some $r > 0$. This can be easily obtained by modifying a generic polynomial of degree $2$, taking $r$ small enough and using the assumption $\partial_N\varphi(0) > 0$. Taking eventually $r$ smaller, we may also assume $\varphi \leq \tilde{v} - \epsilon$ in $\partial B_r$, for some $\epsilon > 0$ depending on $r$.

Now, since $G_j \to G$ in the Hausdorff distance and $\tilde{v} \in C^\alpha(B_{1/4})$, then for every sequence $\sigma_j \to 0^+$ there is a sequence $r_j \to 0^+$, such that
\begin{equation}\label{eq:ImpFlat3HausConvvjv}
|v_j(x) - \tilde{v}(y)| \leq \sigma_j, \quad \text{ for every } \; x,y \in \overline{U}_j \; \text{ satisfying } \; |x-y| \leq r_j,
\end{equation}
where $U_j := \{u_j > \vartheta_1\varepsilon_j\}\cap B_{1/4}$. Since $\tilde{v} \geq \varphi$ in $B_r$ with $\tilde{v} \geq \varphi + \epsilon$ in $\partial B_r$ and $v(0) = \varphi(0) = 0$, we have $v_j \geq \varphi - \sigma_j$ in $\overline{U}_j\cap B_r$, $v_j \geq \varphi + \epsilon - \sigma_j$ in $\overline{U}_j \cap \partial B_r$ and $v_j \leq \sigma_j$ in $\overline{U}_j \cap B_{r_j}$, for every $j$. Let
\[
t_j := \sup\{ t \in \mathbb{R}: v_j \geq \varphi + t\sigma_j \; \text{ in } \overline{U}_j \cap B_r \}.
\]
Since $v_j \leq \sigma_j$ in $\overline{U}_j \cap B_{r_j}$ and $\varphi > 0$ in $\{x_N > 0\} \cap\{x' = 0\} \cap B_{r_j}$, we have $t_j \in [-1,2]$.  So, setting $\tilde{\delta}_j := t_j\sigma_j\delta_j = o(\delta_j)$,
\[
\phi_j(x) := x_N + \delta_j\varphi(x) + \tilde{\delta}_j,
\]
and using the definition of $t_j$ and $v_j$, we deduce
\begin{equation}\label{eq:ImpFlat3BoundsujFirst}
\begin{cases}
u_j \geq \phi_j \quad &\text{ in } \overline{U}_j \cap B_r \\
u_j \geq \phi_j + \epsilon\delta_j \quad &\text{ in } \overline{U}_j \cap \partial B_r \\
u_j(x_j) = \phi_j(x_j) \quad &\text{ for some } x_j \in \overline{U}_j\cap B_r.
\end{cases}
\end{equation}
Further, by \eqref{eq:ImpFlat2CalphaBoundproof}, we have
\begin{equation}\label{wthiowhoithw}
v_j(x) - v_j(y) \geq -\omega_j(x-y) \quad \forall x \in B_r, \, y \in \overline{U}_j,
\end{equation}
where $\omega_j(x-y) := C(\dd_j + |x-y|)^{\alpha}$, and $C > 0$ and $\alpha \in (0,1)$ are as in Lemma \ref{lem:ImpFlat2}, for every $j$. 

Now, given $x \in \{ x_N > -\sqrt{\dd_j} \} \cap B_r$, since by assumption $\{x_N \geq -\dd_j\} \subset U_j$, we can take $y \in \overline{U}_j$ such that $|x-y| \leq 2\sqrt{\dd_j}$.
Hence, using \eqref{wthiowhoithw} we deduce
\begin{equation}\label{eq:ImpFlat3BoundBeloTotal}
\begin{aligned}
v_j(x) &\geq v_j(y) - \omega_j(x-y) \geq \varphi(y) - \sigma_j - C \dd_j^{\alpha/2} \geq \varphi(x) - C|x-y| - \sigma_j - C \dd_j^{\alpha/2} \\
&\geq \varphi(x) - 2C\dd_j^{1/2} - \sigma_j - C \dd_j^{\alpha/2},
\end{aligned}
\end{equation}
for $j$ large enough and a new constant $C > 0$. Consequently, noticing that
\begin{equation}\label{eq:ImpFlat3BoundBeloTotal1}
v_j(x) = \frac{u_j(x) - x_N}{\dd_j} \geq -\frac{x_N}{\dd_j} \geq \dd_j^{-1/2} \to +\infty \quad \text{ in } B_r\cap\{x_N \leq -\sqrt{\dd_j}\},
\end{equation}
for large $j$, it follows
\begin{equation}\label{eq:ImpFlat3BoundsujFirstBis}
u_j \geq \phi_j \quad \text{ in } B_r,
\end{equation}
for $j$ large enough, eventually taking $\tilde{\dd}_j =o(\dd_j)$ smaller.

Now, let us set $w^{\varepsilon_j} = w_{\vep_j}^{\varepsilon_j}$. Combining the first inequality of \eqref{eq:1DSolCompxwvep} (with $\dd = 0$ and $\sigma \in (1/2,3/4)$) with  \eqref{eq:ImpFlat3BoundsujFirstBis}, we obtain $u_j > w^{\varepsilon_j}(\phi_j - \vep_j^\sigma)$ in $B_r\cap \{ w^{\varepsilon_j}(\phi_j - \vep_j^\sigma) > 0\}$ and thus, since $u_j > 0$,
\begin{equation}\label{eq:ImpFlat3BoundaboveujwphijBr1}
u_j > w^{\varepsilon_j}(\phi_j - \vep_j^\sigma) \quad \text{ in } \overline{B}_r,
\end{equation}
for $j$ large enough. Using \eqref{eq:1DSolCompxwvep} again and the last two inequalities in \eqref{eq:ImpFlat3BoundsujFirst}, it follows
\begin{equation}\label{eq:ImpFlat3Boundsuj}
\begin{cases}
u_j > w^{\varepsilon_j} (\phi_j - \vep_j^\sigma + \tfrac{\epsilon}{2}\delta_j) + \tfrac{\epsilon}{2}\delta_j \;\quad &\text{ in } \overline{U}_j \cap \partial B_r \\
u_j(x_j) < w^{\varepsilon_j}(\phi_j(x_j) + \vep_j^\sigma),
\end{cases}
\end{equation}
for every $j$ large enough. Now, let us set $w_\lambda := w^{\varepsilon_j}(\phi_j + \lambda \vep_j^\sigma)$ and define
\[
\lambda_\ast :=  \sup \{ \lambda \in (-1,\infty): w_\lambda < u_j \; \text{ in } B_r \}.
\]
By definition of $\lm_\ast$, we have
\begin{equation}\label{eq:ImpFlat3TocuhingPointyj}
\begin{cases}
u_j \geq w_{\lm_\ast} \quad &\text{ in } B_r \\
u_j(y) = w_{\lm_\ast}(y) \quad &\text{ for some } y \in \{w_{\lm_\ast} > 0\} \cap \overline{B}_r,
\end{cases}
\end{equation}
while, following \eqref{eq:ImpFlat1Eqvt} and recalling that $\partial_N \varphi > 0$ in $B_r$ and $\Delta \varphi = \text{tr} (M) = 0$, we easily find
\begin{equation}\label{eq:ImpFlat3StrictSubsol}
\begin{aligned}
\Delta w_{\lm_\ast} &= \tfrac{1}{2} \Phi_{\varepsilon_j}'(w_{\lm_\ast}) \left[1 + 2\delta_j \partial_N \varphi + \delta_j^2|\nabla \varphi|^2 \right] 
 \\
&> \tfrac{1}{2} \Phi_\varepsilon'(w_{\lm_\ast}) \quad \text{ in } \{w_{\lm_\ast} > 0\} \cap B_r.
\end{aligned}
\end{equation}
If $y \in \{w_{\lm_\ast} > 0\} \cap B_r$, then $\Delta (u_j - w_{\lm_\ast})(y) < 0$, in contradiction with \eqref{eq:ImpFlat3TocuhingPointyj}. So, we are left to show that it cannot be $y \in \{w_{\lm_\ast} > 0\} \cap \partial B_r$, obtaining a contradiction with the definition of $\lm_\ast$.

To see this, we notice that $\lm_\ast \in (-1,1)$, thanks to \eqref{eq:ImpFlat3Boundsuj} and the monotonicity of $w^{\vep_j}$. Consequently, since for $j$ large enough we have $2\dd_j^{2\sigma-1} < \tfrac{\epsilon}{2}$, the first inequality in \eqref{eq:ImpFlat3Boundsuj} yields
\[
\begin{aligned}
w_{\lambda_\ast} &= w^{\varepsilon_j}(\phi_j + \lambda_\ast \vep_j^\sigma) \leq w^{\varepsilon_j}(\phi_j - \vep_j^\sigma + 2\vep_j^\sigma) \leq w^{\varepsilon_j}(\phi_j - \vep_j^\sigma + 2\dd_j^{2\sigma}) \\
&\leq w^{\varepsilon_j}(\phi_j - \vep_j^\sigma + \tfrac{\epsilon}{2}\delta_j) < u_j - \tfrac{\epsilon}{2}\delta_j \quad \text{ in } \overline{U}_j \cap \partial B_r.
\end{aligned}
\]
Notice that the above inequality also implies $w_{\lm_\ast} = 0$ in $\partial U_j \cap \partial B_r$, that is $\{w_{\lm_\ast} > 0\} \cap \partial B_r \subset U_j \cap \partial B_r$, and our contradiction follows.

\smallskip

\emph{Step 3.} Now we show that $\partial_N \tilde{v} \geq 0$ in $\{x_N = 0\}$ in the viscosity sense, that is for every $\varphi \in C^\infty(B_1)$ such that $\varphi \geq \tilde{v}$ in $B_{1/4}$ with equality only at some $z \in \{x_N = 0\}\cap B_{1/4}$, then $\partial_N \varphi(z) \geq 0$.

Proceeding as in \emph{Step 2}, we assume by contradiction $\partial_N\varphi(0) < 0$ for some $\varphi \in C^\infty(B_1)$ as in \eqref{eq:ImpFlat3ExprPhi} with $m < 0$ and $\text{tr}(M) = 0$.

By \eqref{eq:ImpFlat3HausConvvjv} and the assumptions on $\varphi$, we have $v_j \leq \varphi + \sigma_j$ in $\overline{U}_j \cap B_r$, $v_j \leq \varphi - \epsilon + \sigma_j$ in $\overline{U}_j \cap \partial B_r$ and $v_j \geq -\sigma_j$ in $\overline{U}_j \cap B_{r_j}$, for every $j$. So, similar to \emph{Step 2}, we deduce
\[
\begin{cases}
u_j \leq \phi_j \quad &\text{ in } \overline{U}_j \cap B_r \\
u_j \leq \phi_j - \epsilon\delta_j \quad &\text{ in } \overline{U}_j \cap \partial B_r \\
u_j(x_j) = \phi_j(x_j) \quad &\text{ for some } x_j \in \overline{U}_j\cap B_r,
\end{cases}
\]
where $\phi_j(x) := x_N + \delta_j\varphi(x) + \tilde{\delta}_j$, for some $\tilde{\delta}_j = o(\delta_j)$. As above, by the second inequality in \eqref{eq:1DSolCompxwvep1}, we obtain
\[
u_j < w^{-\varepsilon_j}(\phi_j + \vep_j^\sigma) \quad \text{ in } \overline{U}_j \cap\overline{B}_r,
\]
and
\begin{equation}\label{eq:ImpFlat3Boundsaboveuj}
\begin{cases}
u_j < w^{-\varepsilon_j}(\phi_j + \vep_j^\sigma - \tfrac{\epsilon}{2}\delta_j) - \tfrac{\epsilon}{2} \delta_j \quad &\text{ in } \overline{U}_j \cap \partial B_r \\
u_j(x_j) > w_\varepsilon^{-\varepsilon}(\phi_j(x_j) - \vep_j^\sigma).
\end{cases}
\end{equation}
Actually, we have
\begin{equation}\label{eq:ImpFlat3BoundaboveujwphijBr}
u_j < w^{-\varepsilon_j}(\phi_j + \varepsilon_j^\sigma) \quad \text{ in } \overline{B}_r,
\end{equation}
for $j$ large enough. Indeed, exactly as in \eqref{eq:ImpFlat3ExpBound11}, $u_j$ exponentially decays in $B_r \setminus V_j$, where
\[
V_j := B_r\cap \{x : d(x,U_j) \leq \vep_j^{3/4}\},
\]
and thus, by \eqref{eq:1DEvenSolProp}, we have $u_j < w^{-\varepsilon_j}(\phi_j + \vep_j^\sigma)$ in $B_r\setminus V_j$. Moreover, by monotonicity,
\[
w^{-\varepsilon_j}(\phi_j + \vep_j^\sigma) \geq w^{-\varepsilon_j}(\phi_j + \varepsilon_j^\sigma - \tfrac{\epsilon}{2}\delta_j) \geq u_j + \tfrac{\epsilon}{2}\delta_j \quad \text{ in } \overline{U}_j \cap \partial B_r,
\]
by the first inequality in \eqref{eq:ImpFlat3Boundsaboveuj}. So, thanks to the comparison principle, we are left to check that
\[
u_j < w^{-\varepsilon_j}(\phi_j + \vep_j^\sigma) \quad \text{ in } \partial B_r \cap (V_j \setminus U_j).
\]
This follows exactly as the end of the proof of Lemma \ref{lem:ImpFlat2tris}: by the inequality above and $\vep_j \leq \dd_j^2$, we have
\[
w^{-\varepsilon_j}(\phi_j + \vep_j^\sigma) \geq \tfrac{\epsilon}{2} \sqrt{\vep_j} \quad \text{ in } \overline{U}_j \cap \partial B_r,
\]
and so, if $y \in \overline{B}_r$ is any point such that $w^{-\varepsilon_j}(\phi_j(y) + \vep_j^\sigma) = \vth_2\vep_j$ and $x \in \overline{U}_j \cap \partial B_r$, then it must be $|x-y| \geq c\sqrt{\vep_j}$, for some $c > 0$ independent of $j$, which implies 
\[
w^{-\varepsilon_j}(\phi_j + \vep_j^\sigma) \geq \vth_2\vep_j \quad \text{ in } \{x: d(x,U_j\cap\partial B_r) \leq \tfrac{\sqrt{\vep_j}}{2} \}.
\]
Finally, since $\vep_j^{3/4} < \vep_j^\sigma$ for every $j$ large enough, we have 
\[
w^{-\varepsilon_j}(\phi_j + \vep_j^\sigma) \geq \vth_2\vep_j > \vth_1\vep_j \geq u_j \quad \text{ in } \partial B_r \cap (V_j \setminus U_j), 
\]
and \eqref{eq:ImpFlat3BoundaboveujwphijBr} follows.

Now, similar to \emph{Step 2}, we define $w_\lambda := w^{-\varepsilon_j}(\phi_j + \lambda \vep_j^\sigma)$
\[
\lambda_\ast :=  \inf\{ \lambda \in (-\infty,1): u_j < w_\lambda \; \text{ in } B_r \},
\]
which satisfies $\lm_\ast \in (-1,1)$ in view of the second inequality in \eqref{eq:ImpFlat3Boundsaboveuj}. Further,
\begin{equation}\label{eq:ImpFlat3TocuhingPointyj21}
\begin{cases}
u_j \leq w_{\lm_\ast} \quad &\text{ in } B_r \\
u_j(y) = w_{\lm_\ast}(y) \quad &\text{ for some } y \in \overline{B}_r,
\end{cases}
\end{equation}
and by \eqref{eq:ImpFlat1Eqvt}-\eqref{eq:ImpFlat3StrictSubsol}, and that $\partial_N \varphi < \tfrac{m}{2}$ in $B_r$ with $\text{Tr}(M) = 0$, there holds
\[
\begin{aligned}
\Delta w_{\lm_\ast} = \tfrac{1}{2} \Phi_{\varepsilon_j}'(w_{\lm_\ast}) \left[1 + 2\delta_j \partial_N \varphi + \delta_j^2|\nabla \varphi|^2 \right] + \delta_j \dot w_{\lm_\ast} \Delta \varphi < \tfrac{1}{2} \Phi_{\varepsilon_j}'(w_{\lm_\ast}) \quad \text{ in } B_r,
\end{aligned}
\]
if $j$ is large enough. Exactly as above, \eqref{eq:ImpFlat3TocuhingPointyj21}, the equation of $u_j$ and the above differential inequality imply $y \in \partial B_r$. However, repeating the arguments of the proof of \eqref{eq:ImpFlat3BoundaboveujwphijBr} above (replacing $w^{-\varepsilon_j}(\phi_j + \vep_j^\sigma)$ with $w_{\lm_\ast}$) and using that the fact that $\lm_\ast \in (-1,1)$ allow us show that this is impossible, i.e., $u_j < w_{\lm_\ast}$ in $\partial B_r$, obtaining a contradiction.

\smallskip

\emph{Step 4.} As a consequence of \emph{Step 2} and \emph{Step 3}, we obtain that $\tilde{v}$ is bounded and harmonic in $B_{1/4}$ and $\partial_N \tilde{v} |_{x_N = 0} = \partial_N v |_{x_N = 0} = 0$, $\tilde{v}(0) = v(0) = 0$. In particular, by standard elliptic estimates, $\tilde{v} \in C^\infty(B_\varrho)$ and
\[
\sup_{x \in B_\varrho} |\tilde{v}(x) - \nabla v(0)\cdot x| \leq c_N \varrho^2,
\]
every $\varrho \in (0,\tfrac{1}{4})$ and some $c_N > 0$.  Proceeding as in \eqref{eq:ImpFlat3BoundBeloTotal}, we have
\[
v_j(x) \geq v_j(y) - \omega_j(x-y) \geq \tilde{v}(y) - \sigma_j - C \dd_j^{\alpha/2} \geq \tilde{v}(x) - 2C\dd_j^{1/2} - \sigma_j - C \dd_j^{\alpha/2},
\]
for every $x \in \{ x_N > -\sqrt{\dd_j} \} \cap B_r$, we take $y \in \overline{U}_j$ such that $|x-y| \leq 2\sqrt{\dd_j}$, while, by \eqref{eq:ImpFlat3BoundBeloTotal1}, $v_j(x) \geq \dd_j^{-1/2}$ in $\{ x_N < -\sqrt{\dd_j} \} \cap B_r$. Consequently, by \eqref{eq:ImpFlat3HausConvvjv}, for every $\varrho \in (0,\tfrac{1}{4})$, there is $j_\varrho > 0$ such that
\begin{equation}\label{eq:ImpFlat3C1alphaBound}
\begin{array}{rlll}
{} &\hspace{-5pt} v_j(x) - \nabla v(0)\cdot x \leq c_N \varrho^2 \quad &\mbox{in }& B_{\varrho} \cap \overline{U}_j \\
-c_N \varrho^2  \leq&\hspace{-5pt}  v_j(x) - \nabla v(0)\cdot x \quad   &\mbox{in }&B_{\varrho}
\end{array}
\end{equation}
for some new $c_N > 0$ and all $j \geq j_\varrho$.  Now, let us define the unit vector
\[
\nu := \frac{e_N + \delta_j\nabla v(0)}{|e_N + \delta_j\nabla v(0)|}.
\]
Notice that, since $\partial_N v(0) = 0$, we have
\begin{equation}\label{eq:ImpFlat3RelNormApproVect}
|e_N + \delta_j\nabla \tilde v(0)|^2 = 1 + \delta_j^2 |\nabla \tilde v(0)|^2,
\end{equation}
and so 
\[
\begin{aligned}
|e_N - \nu|^2 &= \frac{\delta_j^2 |\nabla v(0)|^2 + \left( |e_N + \delta_j \nabla v(0)| - 1 \right)^2 }{1 + \delta_j^2 |\nabla v(0)|^2} = \frac{2\delta_j^2 |\nabla v(0)|^2 + 2\left( 1 - |e_N + \delta_j \nabla v(0)| \right) }{1 + \delta_j^2 |\nabla v(0)|^2} \\
&\leq 2\delta_j^2 |\nabla \tilde v(0)|^2.
\end{aligned}
\]
Hence,  recalling $\|\tilde v\|_{L^\infty(B_1)} = \|v\|_{L^\infty(B_1\cap \{x_N>0\})}  \le1$ and using the standard gradient estimate for harmonic functions
\[
|e_N - \nu| \leq  \sqrt 2 \delta_j |\nabla \tilde v(0)| \le \sqrt 2 \delta_j  N\|\tilde v\|_{L^\infty(B_1)}  \le  \sqrt 2 N\delta_j,
\]
for and $j$ large enough. On the other hand, since $u_j$ is uniformly bounded in $B_{1/4}$ by \eqref{eq:ImpFlat3AssSeqj}, \eqref{eq:ImpFlat3RelNormApproVect} yields
\[
\begin{aligned}
\frac{u_j(x) - \nu \cdot x}{\delta_j} &= \frac{u_j(x) \left(\sqrt{1 + \delta_j^2|\nabla v(0)|^2} - 1 \right)}{\delta_j} + \frac{u_j(x) - (e_N + \delta_j\nabla v(0))\cdot x}{\delta_j} \\
&= O(\delta_j) + v_j(x) - \nabla v(0)\cdot x,
\end{aligned}
\]
and thus, by \eqref{eq:ImpFlat3C1alphaBound},
\begin{equation}\label{eq:ImpFlat3ImpFlatUj}
\begin{array}{rlll}
{} &\hspace{-5pt} u_j(x) - \nu\cdot x \leq c_N \varrho^2\dd_j \quad &\mbox{in }& B_{\varrho} \cap \overline{U}_j \\
-c_N \varrho^2\dd_j  \leq&\hspace{-5pt}  u_j(x) - \nu\cdot x \quad   &\mbox{in }&B_{\varrho},
\end{array}
\end{equation}
for some new $c_N > 0$ and $j \geq j_\varrho$.  Finally, given any $\gamma \in (0,1)$ and taking $\varrho_0 \in (0,\tfrac{1}{4})$ such that $c_N \varrho_0^2 \leq \varrho_0^{1+\gamma}$, we obtain that both \eqref{eq:ImpFlatFinalThesis1} and \eqref{eq:ImpFlatFinalThesis2} are satisfied, a contradiction.
\end{proof}
%
%
%

%

%
%
%
%
\section{Proof of Theorem \ref{thm:main} and Corollary \ref{cor:main}}\label{Sec:AsymFlatImplies1D}
The goal of this section is to prove Theorem \ref{thm:main} and Corollary \ref{cor:main}. 
The former will be a consequence of Theorem \ref{thm:AsymFlatImplies1D} below, which is obtained combining Theorem \ref{lem:ImpFlat3} and a sliding argument in the spirit of \cite{BerCaffaNiren1997:art,DipierroEtAl2020:art}. 
The latter will be an immediate byproduct of Proposition \ref{prop:main}, Theorem \ref{thm:main} and the classification of 1-homogeneous entire local minimizers of \eqref{eq:Energy1Phase} established in \cite{CafJerKen2004:art,JerisonSavin2009:art}.

We begin with two consequences of Theorem \ref{lem:ImpFlat3} that we will use in the proof of Theorem \ref{thm:AsymFlatImplies1D}.
\begin{cor}[Preservation of flatness]\label{cor:PresFlat}
Fix $\gamma=1/2$,  and let $\delta_0>0$ and $\varrho_0\in (0,1/4)$ be the constants as in Theorem \ref{lem:ImpFlat3}. Let $R_0: = 1/\varrho_0$. 
Given $\delta>0$, we define
\begin{equation}\label{eq:PresFlatAssj}
j_\delta : =  \bigg\lceil  \frac{|\log \delta^2|}{\log R_0} \bigg\rceil.
\end{equation}
Let $u:\RR^N\to\RR_+$ be a critical point $\EE$  with $u(0) \in [\vartheta_1, \vartheta_2]$. If $u$
satisfies ${\rm Flat_1}(\nu_k, \delta, R_0^k)$   for some $\delta\in (0,\delta_0)$, $k \geq j_\delta$, and $\nu_k \in \Sf^{N-1}$, then for every  $i$ such that $j_\delta \leq i \leq k$,  $u$ satisfies ${\rm Flat}_1(\nu_i, \delta, R_0^i)$ for some $\nu_i \in \Sf^{N-1}$.
\end{cor}
\begin{proof} The proof is by iterating Theorem \ref{lem:ImpFlat3}. Indeed, thanks to  \eqref{eq:PresFlatAssj} we have 
\begin{equation}\label{eq:PresFlatCondEpsDel}
\frac{1}{\delta^2 R_0^i} \leq \frac{1}{\delta^2 R_0^{j_\delta}} < 1 \qquad \mbox{for all } i \ge j_\delta\,.
\end{equation}
Thanks to Theorem \ref{lem:ImpFlat3}   if $u$ satisfies ${\rm Flat_1}(\nu_i, \delta, R_0^i)$  for some $\nu_i \in \Sf^{N-1}$, and $i\ge j_\delta$ then  $u$ satisfies ${\rm  Flat_1}(\nu_{i-1}, R_0^\gamma \delta, R_0^{i-1})$ for some  $\nu_{i-1} \in \Sf^{N-1}$. In particular  $u$ satisfies ${\rm  Flat_1}(\nu_{i-1}, \delta, R_0^{i-1})$. Iterating this the corollary follows.
\end{proof}
\begin{cor} [Improvement of flatness]\label{cor:ImpFlat}
Fix $\gamma=1/2$,  and let $\delta_0>0$ and $\varrho_0\in (0,1/4)$ be the constants as in Theorem \ref{lem:ImpFlat3}. Let $R_0: = 1/\varrho_0$. 
Let $k,n \in \mathbb{N}$ and $\dd > 0$ such that
\begin{equation}\label{eq:ImpFlatAssnk}
(1+2\gamma) n \leq  k - \tfrac{ |\log \delta^2 |}{\log R_0}.
\end{equation}
Let $u:\RR^N\to\RR_+$ be a critical point of $\EE$  with $u(0) \in [\vartheta_1, \vartheta_2]$. If $u$
satisfies ${\rm Flat_1}(\nu_k, \delta, R_0^k)$   for some $\delta\in (0,\delta_0)$, $k \geq j_\delta$ and $\nu_k \in \Sf^{N-1}$,  for every  $i$ such that $k-n\leq i \leq k$,  $u$ satisfies ${\rm Flat}_1(\nu_i,  R_0^{-\gamma(k-i)} \delta, R_0^{i})$ for some $\nu_i \in \Sf^{N-1}$.
\end{cor}
\begin{proof} 
The proof is by iterating Theorem \ref{lem:ImpFlat3}. Indeed, thanks to \eqref{eq:ImpFlatAssnk} we have 
\begin{equation}\label{eq:ImpFlatCondEpsDel}
\frac{1}{ \big(R_0^{-\gamma(k-i)}  \delta\big)^2 R_0^i} \leq \frac{1}{\delta^2 R_0^{k-n -2\gamma n}} < 1 \qquad \mbox{for all } i \ge j_\delta\,.
\end{equation}
Thanks to Theorem \ref{lem:ImpFlat3}   if $u$ satisfies ${\rm Flat_1}(\nu_i,  R_0^{-\gamma(k-i)}\delta, R_0^i)$  for some $\nu_i \in \Sf^{N-1}$ (which is satisfied by assumption for $i=k$), and $i\ge k-n$ then  $u$ satisfies ${\rm  Flat_1}(\nu_{i-1},  R_0^{-\gamma(k-i+1)} \delta, R_0^{i-1})$ for some $\nu_{i-1} \in \Sf^{N-1}$. Iterating this, the corollary follows.
\end{proof}
\begin{thm}\label{thm:AsymFlatImplies1D} Let $\gamma= 1/2$,  and let $\varrho_0\in (0,1/4)$ be the constant in Theorem \ref{lem:ImpFlat3}, and  $R_0 :=1/\varrho_0\ge 2$. 

Suppose that  $u:\RR^N\to\RR_+$ is a critical point of $\EE$ with $0 \in \{ \vth_1 \leq u \leq \vth_2\}$ and let $\{u_{\varepsilon}\}_{\vep \in (0,1)}$ be a blow-down family, where 
$u_\vep : = \vep u(\, \cdot\, /\vep)$.

Set $\vep_j := R_0^{-j}$ and assume there exist $\nu \in \Sf^{N-1}$, and a sequence $j_l \to +\infty$ and $\dd_{l} \to 0$  (as $l \to +\infty$) for which
\begin{equation}\label{eq:1DSymAss1}
|u_{\varepsilon_{j_l}} - (\nu \cdot x )_+| \leq \dd_l \quad \text{ in } B_2,
\end{equation}
and
\begin{equation}\label{eq:1DSymAss2}
\{ x\,  :\,   \nu \cdot x \leq -\dd_{l} \} \subset \{ u_{\vep_{j_l}} \leq \vth_1\vep_{j_l} \} \subset \{ u_{\vep_{j_l}} \leq \vth_2\vep_{j_l} \} \subset \{x\, :\,  \nu \cdot x \leq \dd_l \} \quad \text{ in } B_2,
\end{equation}
for every $l \in \NN$. Then $u$ is $1D$.
\end{thm}
\begin{proof} 

Throughout the proof $\delta_0$ will denote the constant of Theorem \ref{lem:ImpFlat3}. Observe that, by possibly replacing $\delta_l$ by some sequence with slower convergence towards $0$, we may assume without loss of generality that $\vth_2 \vep_{j_l} \leq \dd_l/2$.

Up to a rotation of the coordinate system, we may assume $\nu = e_N$.
The proof is divided in several steps as follows.

\smallskip

\emph{Step 1.}  Fix $\dd\in (0, \dd_0)$ to be chosen later. We first show that 
\begin{equation}\label{eq:AsymptFlatAssScaled}
u \quad \mbox{satisfies}\quad {\rm Flat}_1(\nu_{j},  \delta, R_0^j) \quad \forall j\ge j_\delta : = \Big\lceil  \frac{|\log (\delta^2)|}{\log R_0} \Big\rceil ,
\end{equation}
for some $\nu_j\in \Sf^{n-1}$. 

By \eqref{eq:1DSymAss1}  ($\nu=e_N$), we have 
\begin{equation}\label{gwjiowhtiewh0}
(x_N)_+ -\dd_l \leq u_{\varepsilon_{j_l}} \leq (x_N)_+ + \dd_l \quad \text{ in }B_2.
\end{equation} 
 
 Let us show that this implies 
\begin{equation}\label{gwjiowhtiewh}
\begin{array}{rlll}
{} &\hspace{-5pt} u_{\vep_{j_l}}(x) -  x_N \leq \delta_l  \quad &\mbox{in }& B_1\cap \{u_{\vep_{j_l}} \geq \vartheta_1{\vep_{j_l}}\}\\
-\delta _l \leq&\hspace{-5pt}  u_{\vep_{j_l}}(x) - x_N \quad   &\mbox{in }&B_1,
\end{array} 
\end{equation}
for all $l$  sufficiently large.

Indeed on the one hand, \eqref{gwjiowhtiewh0} implies $u_{\varepsilon_{j_l}} \geq x_N -\dd_l$  in $B_1$ (for $l$ large),  which gives the inequality from below in \eqref{gwjiowhtiewh}.

To show the one from above, we set $v := u_{\vep_{j_l}} - x_N -2\dd_l$ and we show $v \leq 0$ in $B_1 \cap \{u_{\vep_{j_l}} \geq \vartheta_1\vep_{j_l} \}$ using a comparison argument. Thanks to \eqref{gwjiowhtiewh0}, using $(x_N)_+ -\dd_l\le x_N$  in $\{x_N\ge -\delta_l\}$ we find (using $\vth_2 \vep_{j_l} \leq \dd_l/2$)
\[
v  \leq  u_{\vep_{j_l}} - (x_N)_+ -\dd_l  \le \vth_2\vep_{j_l} - \dd_l  \leq  -\tfrac{\dd_l}{2} \quad \text{ in } B_2 \cap \{u_{\vep_{j_l}} \leq \vth_2\vep_{j_l}\} \cap \{ x_N \geq - \dd_l \},
\]
for every $l$ large enough.
Further, \eqref{gwjiowhtiewh0} automatically implies $v \leq -\dd_l \le 0$ in $B_2 \cap \{ x_N \geq 0\}$, since $(x_N)_+ = x_N$ there. 
Also,  by \eqref{gwjiowhtiewh0} again, $v \leq \dd_j$ in $\overline{B}_2\cap \{u_{\vep_{j_l}} \ge \vth_2\vep_{j_l}\} \cap\{|x_N| \leq \dd_j\}$.

On the other hand, $\Delta v = \Delta u_{\vep_{j_l}}= \tfrac12\Phi'_{\vep_{j_l}}(u_{\vep_{j_l}}) \geq 0$ in $B_2$ and thus the function
\[
\underline{v} := \frac{v}{\dd_l} + A  x_N ^2
\]
satisfies
\[
\begin{aligned}
\begin{cases}
\Delta \underline{v} \geq 2A  \quad &\text{ in } B_2\cap\{ -\dd_l\le x_N<0\} \\
\underline{v} \leq -\tfrac12  + A\dd_l^2  \quad &\text{ in } B_2\cap\{x_N = -\dd_l\} \\
\underline{v} \leq 0  \quad &\text{ in } B_2\cap\{x_N = 0\} \\
\underline{v} \leq 1+ A\dd_l^2  \quad &\text{ in } \partial B_2\cap\{-\dd_l\le x_N<0\},
\end{cases}
\end{aligned}
\]
for every $A > 0$. Now, consider the function $\overline{h}_{x_0}(x) = \tfrac{A}{N} |x-x_0|^2$. For every $x_0 \in B_1\cap\{-\dd_l < x_N  < 0\}$, we have
\[
\overline{h}_{x_0} \geq \tfrac{A}{N}  \quad \text{ in } \partial B_2 \cap\{ -\dd_l \le x_N \le 0\},
\]
and taking $A := 2N$ , we have $\overline{h}_{x_0} \geq 2$ in $\partial B_2 \cap\{ -\dd_l \le x_N \le 0\}$. Then, for $l$ large, we have $A\dd_l^2 \leq 1$ and 
\[
\begin{aligned}
\begin{cases}
\Delta \underline{v} \geq 2A = \Delta \overline{h}_{x_0}  \quad &\text{ in } B_2 \cap\{ -\dd_l\le x_N<0\} \\
\underline{v} \leq 0 \leq \overline{h}_{x_0}  \quad &\text{ in } B_2 \cap \partial\{ -\dd_l\le x_N<0\} \\
\underline{v} \leq 2 \leq \overline{h}_{x_0}  \quad &\text{ in } \partial B_2 \cap \{ -\dd_l\le x_N<0\}.
\end{cases}
\end{aligned}
\]
Then, by the maximum principle we obtain $\underline{v} \leq \overline{h}_{x_0}$. Consequently, since $\overline{h}_{x_0}(x_0) = 0$ and $x_0$ is arbitrary in $B_1 \cap \{-\dd_l < x_N  < 0\}$, we have $\overline{v} \leq 0$ in $B_1 \cap  \{-\dd_l < x_N  < 0\}$ and so, by the definition of $\underline{v}$, we obtain $v \leq 0$ in $B_1 \cap \{-\dd_l < x_N  < 0\}$. This proves \eqref{gwjiowhtiewh}. In other words, after scaling we have shown that  \eqref{eq:AsymptFlatAssScaled}  holds for $j= j_l$ and $\nu_{j_l}=\nu$, provided that $l$ is sufficiently large. Hence, as a consequence of  Corollary \ref{cor:PresFlat}  we obtain that  that   \eqref{eq:AsymptFlatAssScaled} holds for every integer $j$ such that $j_\dd \leq j \leq j_l$ for some $\nu_j \in \Sf^{N-1}$. Observing that $j_l$ can  be taken arbitrarily large  concludes the proof of \eqref{eq:AsymptFlatAssScaled}.

\smallskip

\emph{Step 2.} In this second step, we prove that there exists $C \geq 1$ such that for every $z \in \{\vth_1 \leq u \leq \vth_2\}$ and every $R \geq C$
\begin{equation}\label{eq:AsymptFlat_En}
u(z+\,\cdot\,) \quad \mbox{satisfies}\quad {\rm Flat}_1(e_N,  CR^{-1/2}, R) \qquad \forall R\ge C,
\end{equation}

Note that this is a really strong information since the constant $C$ and the direction $e_N$ of flatness are independent of  $z$, which varies in an unbounded set!

To obtain \eqref{eq:AsymptFlat_En}, we first show the existence of some $k_0$ (independent of $z$) such that for every $k \geq k_0$ and every $z \in \{\vartheta_1\le u\le \vartheta_2\}$, there are $\nu_{z,k} \in \Sf^{N-1}$ such that, for all $k\ge k_0$,
\begin{equation}\label{eq:AsymptFlatallz}
u(z+ \,\cdot\,) \quad \mbox{satisfies}\quad {\rm Flat}_1(\nu_{z,k},  \delta_0, R_0^k) 
\end{equation}
for some $\nu_{z,k}\in \Sf^{n-1}$. 
Indeed, given $z \in \{\vartheta_1\le u\le \vartheta_2\}$ choose  $i\in \mathbb N$ such that  $|z|\le \frac{\dd_0}{2} R_0^{i}$.  Take $j = i+1$ in \eqref{eq:AsymptFlatAssScaled}, and  choose  $\delta$  such that $2\dd R_0 \leq \dd_0$. We then have
\begin{equation}\label{whgtwhgohwoh}
\begin{array}{rlll}
{} &\hspace{-5pt} u(x) -  \nu_{i+1} \cdot x  \leq  \frac{\dd_0}{2} R_0^{i}  \quad &\mbox{in }& B_{R_0^{i+1}}(0)\cap \{u \geq \vartheta_1\}\\
-\frac{\dd_0}{2} R_0^{i} \leq&\hspace{-5pt}  u(x) - \nu_{i+1}\cdot x \quad   &\mbox{in }&B_{R_0^{i+1}}(0) .
\end{array} 
\end{equation}
%
Now since $|z|\le \frac{\dd_0}{2} R_0^{i}$ and $R_0\ge 2$ we have $B_{R_0^i}(z) \subset B_{R_0^{i+1}}$ and 
\[
|u(x) - \nu_{i+1} \cdot (x-z)| \leq |u(x) - \nu_{i+1} \cdot x|  + |z| \le \dd_0R_0^{i}   \quad \text{ in } B_{R_0^{i}}(z). 
\]
Thus, \eqref{whgtwhgohwoh} implies
\[
\begin{array}{rlll}
{} &\hspace{-5pt} u(x) -  \nu_{i+1} \cdot x  \leq \dd_0 R_0^{i}  \quad &\mbox{in }& B_{R_0^{i}}(z)\cap \{u \geq \vartheta_1\}\\
- \dd_0 R_0^{i} \leq&\hspace{-5pt}  u(x) - \nu_{i+1}\cdot x \quad   &\mbox{in }&B_{R_0^{i}}(z) .
\end{array}
\] 
In other words, setting $\nu_{k,z} := \nu_{i+1}$,  we see that \eqref{eq:AsymptFlatallz} is satisfied for  $k=i$ large enough (where $i$ depends on $z$).
But then thanks to Corollary \ref{cor:PresFlat} (applied with $\delta = \delta_0$ and to the ``translated function'' $u(z+\,\cdot\,)$)  we obtain that  \eqref{eq:AsymptFlatallz} holds  for all $k \ge k_0 := j_{\dd_0}$.

We will now use \eqref{eq:AsymptFlatallz} and Corollary \ref{cor:ImpFlat} (applied again to the translated function $u(z+\,\cdot\,)$)  to show \eqref{eq:AsymptFlat_En}. Indeed, for given $j \in \mathbb{N}$ large enough, set
\[
n : = \left\lfloor \frac{j- |\log \delta^2|/\log R_0}{2\gamma} \right\rfloor
\]
and 
\[
k :=  j +n
\]
 Then,
\begin{equation}\label{eq:Sym1Propn}
(1+2\gamma)n =  n + 2\gamma\bigg\lfloor \frac{j- |\log \delta^2|/\log R_0}{2\gamma}\bigg\rfloor    \le k  - \frac{|\log \delta_0^2|}{\log R_0}.
\end{equation}

The above inequality implies that \eqref{eq:ImpFlatAssnk} is satisfied.  By \eqref{eq:AsymptFlatallz}  (since we assume that  $j\ge C$ sufficiently large we have $k \ge j\ge  k_0$), we may apply Corollary \ref{cor:ImpFlat}  to $u(z +\,\cdot\,)$ to obtain that 
$u(z+\cdot)$ satisfies ${\rm Flat}_1(\nu_{z,i},  R_0^{-\gamma(k-i)} \delta_0, R_0^{i})$ for some $\nu_{z,i} \in\mathbb S^{N-1}$  for all $i = j,j+1, \dots, j+n$ (in particular for $i=j$).
Hence using the definition of ${\rm Flat}_1$ and that $k-i = n \ge \frac{j}{2\gamma} -C$, we obtain  
\[
\delta_0 R_0^{-\gamma(k-j)} R_0^{j}  \le \delta_0 R_0^{j-\gamma n} \le C R_0^{j(1-1/2)}
\]
and
\begin{equation}\label{eq:Sym1DFirstBoundImp}
\begin{array}{rlll}
 &\hspace{-5pt} u(x) -  \nu_{z,i}\cdot(x-z)\leq C R_0^{j/2} \quad &\mbox{in }& B_{R_0^j}(z)\cap \{u \geq \vartheta_1\}\\
-CR_0^{j/2} \leq&\hspace{-5pt}  u(x) -  \nu_{z,i}\cdot(x-z) \quad   &\mbox{in }&B_{R_0^j}(z).
\end{array} 
\end{equation}
Now, on the one hand, as a consequence of \eqref{eq:Sym1DFirstBoundImp}, we have 
\[
\max\big(0 \ ,\  \nu_{z,j}  \cdot x - CR_0^{j/2} \big) \le   u(z+x)  \le \max\big( \vartheta_1\  ,\  \nu_{z,j}  \cdot x + CR_0^{j/2}\big)  \quad \mbox{in }  B_{R_0^j},
\]
and thus, using this in two consecutive scales, we obtain 
\[
\max\big(0 \ ,\  \nu_{z,j}  \cdot x - CR_0^{j/2} \big)   \le \max\big( \vartheta_1\  ,\  \nu_{z,j+1}  \cdot x + CR_0^{(j+1)/2}\big)  \quad \mbox{in }  B_{R_0^j}. 
\]
This implies (for $j$ large) 
\begin{equation}\label{hgiohoiwhtiow}
\big| \nu_{z,j}  - \nu_{z,j+1} \big| \le  C(R_0) R_0^{-j/2},
\end{equation}
where $C(R_0)$ is independent of $z$ an $j$. This shows that $\nu_{z,j} \to \nu_z$ as $j\to \infty$ for all $z$. On the other hand, since for every two pair of points $z_1$,$z_2$ \eqref{eq:Sym1DFirstBoundImp}  applied at a scales $R_0^j>>|z_1-z_2|$ implies $(\nu_{z_1,j}-\nu_{z_2,j})\to 0$, we see that $\nu_z\equiv \nu_*$ for all $z$, where $\nu_*$ is independent of $z$.  On the other hand, assumption \eqref{eq:1DSymAss1}  (where $\nu=e_N$ as said in the beginning of the proof), forces $\nu_* =e_N$
and hence $\lim_{j\to \infty}  \nu_{z,j} = e_N$ for all $z$.
Finally, using again \eqref{hgiohoiwhtiow}, triangle inequality, and summing the geometric series we obtain
\[
\big|\nu_{z,j}  - e_N\big| \le\big|\nu_{z,j}  - \lim_{j\to \infty} \nu_{z,j}\big| \le   C(R_0)  \sum_{l=j}^{\infty } R_0^{-l/2} \le C R_0^{-j/2}.
\]
 for all $z \in \{\vth_1 \leq u \leq \vth_2\}$.
Combining this information with \eqref{eq:Sym1DFirstBoundImp} we conclude the proof of \eqref{eq:AsymptFlat_En}.

\smallskip

\emph{Step 3.} We now observe that  \eqref{eq:AsymptFlat_En} has two significant consequences. First, it implies the existence of a  function $G:\mathbb{R}^{N-1} \to \mathbb{R}$ with $G(0)=0$ satisfying
\begin{equation}\label{eq:Sym1DLipConstG}
|G(x') - G(y')| \leq C  \sqrt{|x' - y'|} , \quad \forall x', y' \in \RR^{N-1}
\end{equation}
and
\begin{equation}\label{eq:Sym1DFBBoundsGraph}
\{x_N \leq G(x') - C \} \subset \{u \leq \vartheta_1 \} \subset \{u \leq \vartheta_2 \} \subset \{ x_N \leq G(x') + C \} \quad \text{ in } \mathbb{R}^N.
\end{equation}
Second, since $u - x_N$ is harmonic in $\{u > \vartheta_2\}$, standard elliptic estimates yield
\[
\sup_{x \in B_{r/2}(y)} |\nabla (u(x) - x_N)| \leq \frac{c_N}{r} \sup_{x \in B_r(y)} |u(x) - x_N| \leq \frac{c_N C}{r} R^{j/2},
\]
for every $B_r(y) \subset \{u > \vartheta_2\} \cap B_{R^j}$. Consequently, for every $j \in \mathbb{N}$ and every $y$ such that $B_{R^j/2}(y) \subset \{u > \vartheta_2\} \cap B_{R^j}$, we have
\begin{equation}\label{eq:1DSymEstnbu-x}
\sup_{x \in B_{R^j/4}(y)} |\nabla (u(x) - x_N)| \leq c_N C R^{-j/2}.
\end{equation}

This easily implies that 
\begin{equation}\label{eq:whtihwiohw}
|\nabla u|  \le C \quad \mbox {in }  \RR^N 
\end{equation}

Indeed, if $x = (x',x_N)$ is a point in $\RR^N$ and  and let $R_\circ \ge C$ to be chosen. We consider two complementary cases: either
 $x_N - G(x')\le R_\circ$ or $x_N - G(x')> R_\circ$. In the first case, by  \eqref{eq:AsymptFlat_En} with $R = 2R_\circ$, we obtain  $|u| \le C$ in $B_{R_\circ}(x)$ (with a possibly larger $C$). Now, since $u$ is a bounded solution of the semilinear equation 
$\Delta u = \tfrac{1}{2}\Phi'(u)$  standard elliptic estimates yield $|\nabla u(x)|\le C$. 
In the second case, using  \eqref{eq:Sym1DFBBoundsGraph}-\eqref{eq:Sym1DLipConstG} we obtain that, if $R_\circ$ is chosen large enough and $R:={\rm dist}(x, \partial \{u>\vth_2\}) \ge R_\circ$, then it also follows $|\nabla u(x)|\le C$, thanks to \eqref{eq:1DSymEstnbu-x}.

\smallskip

\emph{Step 4.}  We now perform a sliding argument \`a la Caffarelli-Berestycki-Nirenberg. We fix $\sigma > 0$, $e' \in \Sf^{N-1}\cap\{x_N = 0\}$, and define, for any given $\lambda>0$,
\begin{equation}\label{eq:1DSymDefeulambda}
e := (e',\sigma), \qquad u^\lambda(x) := u(x - \lambda e).
\end{equation}
Choose  $\lambda_\sigma > 0$  such that $C \sqrt{\lambda_\sigma} + 2C  \le \sigma \lambda_\sigma$, where $C$ is the constant in \eqref{eq:Sym1DFBBoundsGraph}-\eqref{eq:1DSymEstnbu-x}. Let us show that
\begin{equation}\label{eq:1DSymSlidindThWeak}
u^\lambda \leq u \quad \text{ in } \mathbb{R}^N \quad \text{ for every } \lambda \geq \lambda_\sigma.
\end{equation}
To prove so, we first observe that,  for every $\lambda \geq \lambda_\sigma$
\begin{equation}\label{eq:1DSymFirstTranLlar}
\{ u \leq \vth_2 \} \subset \{ u^\lambda \leq \vth_1 \}.
\end{equation}
Indeed, let $x \in \{u \leq \vartheta_2\}$ and notice that \eqref{eq:Sym1DLipConstG} yields
\[
\begin{aligned}
(x - \lambda e)_N - G((x - \lambda e)') + C &= x_N - \sigma\lambda - G \big( x' - \lambda e' \big) + C \\
&\leq - \sigma\lambda + G(x') - G \big( x' - \lambda e' \big) + 2C \\
&\leq - \sigma\lambda +C \sqrt{\lambda} + 2C \leq 0,
\end{aligned}
\]
for every $\lambda \geq \lambda_\sigma$, provided  $\lambda_\sigma$ is chosen large enough. 

Now, we set $v := u - u^\lambda$ and we show that $v \geq 0$ in $\mathbb{R}^N$ for every $\lambda \geq \lambda_\sigma$, that is \eqref{eq:1DSymFirstTranLlar}. 

To do so, we first notice that $\RR^N = \Omega_1 \cup \Omega_2 : = \{u\ge \vth_2\} \cup \{u^\lambda <\vth_1\}$ for every $\lambda\ge \lambda_\sigma$, thanks to \eqref{eq:1DSymFirstTranLlar}. 
Further, in the domain $\Omega_1$ the function $v$ satisfies $\Delta v =0$ in $ \{u^\lambda>\vth_2\}$ and $u-u^\lambda \ge \vth_2 -\vth_2 =0$ in $u^\lambda \le\vth_2$. 
Hence the negative part of $v_-$ is subharmonic in $\Omega_1$. 

Also, thanks to \eqref{eq:1DSymFirstTranLlar}, the boundary $ \partial \{ u> \vartheta_2 \}$ of $\Omega_1$ is contained in  $\{ u^\lambda \leq \vartheta_1 \}$ and hence $u-u^\lambda\ge \vth_2-\vth_1>0$ on $\partial \{ u> \vartheta_2 \}$. 
In other words $v_-$ is subharmonic and vanishes on the boundary of $\Omega_1$.
Since ---thanks to \eqref{eq:Sym1DLipConstG} and \eqref{eq:Sym1DFBBoundsGraph}--- the complement of $\Omega_1$ contains a cone with nonempty interior, and ---thanks to \eqref{eq:whtihwiohw} $v$  (and in particular $v_-$) is bounded in all of $\RR^N$, we deduce $v_- =0$  in $\Omega_1$ from  the comparison principle in unbounded domains which contain a cone (see for instance \cite[Lemma 2.1]{BerCaffaNiren1997:art}).

Similarly inside $\Omega_2$, either  $u\ge\vth_1$ and $u^\lambda< \vth_1$ and so $v \ge 0$ or both $u$ and $u^\lambda$ are smaller than $\vth_1$. In that second case, recalling that $\Phi'$ is increasing in $(0, \vth_1)$, we have $\Delta v = \Phi'(u) -\Phi'( u^\lambda) \le 0$ at points where $v = u-u^\lambda \le 0$. Hence $v_-$ is again a subharmonic function in $\Omega_2$. 
Similarly as before we can show that $v_- = 0$ on $\partial\Omega_2$ and that it is bounded. And again the complement of $\Omega_2$ contains a cone, se we may conclude $v_- = 0$ everywhere (that is $v \ge 0$).

\smallskip

\emph{Step 5.}   Let $C$ and $G$ as in \eqref{eq:Sym1DFBBoundsGraph}. Define
\begin{equation}\label{eq:1DSymNewC}
\overline{C} := C + (1+\lambda_\sigma) \|\nabla G\|_{L^\infty(\mathbb{R}^{N-1})},
\end{equation}
and
\begin{equation}\label{eq:1DSymTube}
\mathcal{G} := \big\{ x = (x',x_N) \in \mathbb{R}^N: |x_N - G(x')| \leq \overline{C} \big\}.
\end{equation}
We prove that for every $\lambda > 0$
\begin{equation}\label{eq:1DSymIneqinGC}
u^\lambda \leq u \text{ in } \mathcal{G} \quad \Rightarrow \quad u^\lambda \leq u \text{ in } \mathbb{R}^N.
\end{equation}
In light of \eqref{eq:1DSymSlidindThWeak}, it is enough to treat the case $\lambda \in (0,\lambda_\sigma)$. Following the ideas of \emph{Step 4}, we observe that 
\[
\{x_n \le G(x') -\bar C\} \subset   \{u\le \vth_1\}
\]
Indeed, let  $x$ satisfy
\[
x_N \leq G(x') - \overline{C}.
\]
Consequently, by \eqref{eq:1DSymNewC}, the above inequality and the definition of $e$, we obtain
\[
\begin{aligned}
(x - \lm e)_N - G((x - \lm e)') &\leq x_N - G(x') + \lambda \|\nabla G\|_{L^\infty(\mathbb{R}^{N-1})} \\
&\leq x_N - G(x') + \lambda_\sigma \|\nabla G\|_{L^\infty(\mathbb{R}^{N-1})} \leq - \overline{C} + \overline{C} - C = - C,
\end{aligned}
\]
and thus, by \eqref{eq:Sym1DFBBoundsGraph}, we have $x - \lm e \in \{u \leq \vartheta_1\}$.
In a very similar way, we show
\begin{equation}\label{eq:1DSymInclforComp}
\{ u^\lambda \leq \vartheta_2 \} \subset \{ x_N \leq G(x') + \overline{C} \},
\end{equation}
for every $\lambda \in (0,\lambda_\sigma)$. To complete the proof of \eqref{eq:1DSymIneqinGC}, it is enough to consider $v = u - u^\lambda$, notice that $v \geq 0$ in $\partial \mathcal{G}$ (by assumption) and repeat the arguments of \emph{Step 5}. 
\smallskip

\emph{Step 6.} In this step we show that for every $\lambda > 0$
\begin{equation}\label{eq:1DSymSlidindThGC}
u^\lambda \leq u \quad \text{ in } \mathcal{G}.
\end{equation}
Notice that, as a consequence of \eqref{eq:1DSymIneqinGC}, \eqref{eq:1DSymSlidindThGC} implies that for every $\lambda > 0$
\begin{equation}\label{eq:1DSymSlidindTh}
u^\lambda \leq u \quad \text{ in } \mathbb{R}^N.
\end{equation}
To verify \eqref{eq:1DSymSlidindThGC}, we let
\[
\lambda_\ast := \inf\{ \lambda \geq 0: u^\lambda \leq u \text{ in } \mathcal{G}\} \le \lm_\sigma,
\]
and show that $\lambda_\ast = 0$. Assume by contradiction that $\lambda_\ast \in (0,\lm_\sigma)$.

By definition of $\lambda_\ast$, we have $u^{\lambda_\ast} \leq u$ in $\mathcal{G}$ and, in addtion, there exists $x_j \in \mathcal{G}$ such that $u(x_j) - u^{\lambda_\ast}(x_j) \leq 1/j$.  Set $u_j(x) := u(x + x_j)$, $u_j^{\lambda_\ast}(x) := u^{\lambda_\ast}(x + x_j)$, $v_j := u_j - u_j^{\lambda_\ast}$ and $\mathcal{G}_j:= \mathcal{G} - x_j$.

We have
\[
\mathcal{G}_j = \{(x',x_N) \in \mathbb{R}^N: |x_N - G_j(x')| \leq \overline{C} \},
\]
where $G_j(x') := G(x' + x_j') - x_{j,N}$ (here $x_{j,N} := (x_j)_N$) with
\begin{equation}\label{wthiohoiwhw}
\begin{aligned}
&|G_j(x')| \leq |G(x' + x_j') - G(x_j')| + |G(x_j') - x_{j,N}| \leq C \sqrt{|x'|} + 2\overline{C} \\
&\|\nabla G_j\|_{L^\infty(\mathbb{R}^{N-1})} \leq \|\nabla G\|_{L^\infty(\mathbb{R}^{N-1})},
\end{aligned}
\end{equation}
for every $j$ in view of \eqref{eq:Sym1DLipConstG} and that $x_j \in \mathcal{G}$. As a consequence, we deduce the existence of a locally bounded function $\overline{G}:\RR^N \to \RR$ such that $G_j \to \overline{G}$ locally uniformly in $\RR^N$ and $\mathcal{G}_j \to \overline{\mathcal{G}}$ locally Hausdorff in $\mathbb{R}^N$ (up to subsequence), where $\overline {\mathcal G} : =  \{(x',x_N) \in \mathbb{R}^N: |x_N - \overline G(x')| \leq \overline{C} \}$.

In particular, thanks to \eqref{eq:Sym1DFBBoundsGraph}, we have
\begin{equation}\label{eq:Sym1DFBBoundsGraphBis}
\{x_N \leq \overline{G}(x') - 2\overline{C} \} \subset \{u_j \leq \vartheta_1 \} \subset \{u_j \leq \vartheta_2 \} \subset \{ x_N \leq \overline{G}(x') + 2\overline{C} \} \quad \text{ in } \mathbb{R}^N,
\end{equation}
for every $j$ large enough. On the other hand,  using  \eqref{wthiohoiwhw} for $x'=0$ and recalling \eqref{eq:whtihwiohw} we have 
\[
|u_j(0)| \le 2\overline C\quad \mbox{and} \quad  |\nabla u_j| \le C\quad \mbox{in } \RR^N, 
\]
thus, the sequence $\{u_j\}_{j\in\NN}$ is locally uniformly bounded in $\RR^N$.

Further, since $\Delta u_j = \tfrac12\Phi'(u_j)$ in $\RR^N$ and $\Phi'$ is bounded, standard elliptic estimates and a diagonal argument yield $u_j \to \overline{u}$ in $C^2_{loc}$ as $j \to +\infty$, for some $\overline{u} \in C^2_{loc}(\RR^N)$, up to passing to a subsequence. Similar, $\overline{u}_j \to \overline{u}^{\lambda_\ast}$ in $C^2_{loc}$ as $j \to +\infty$ and, since
\[
\begin{cases}
\Delta v_j = \tfrac{1}{2}(\Phi'(u_j) - \Phi'(u_j^{\lambda_\ast})) \quad &\text{ in } \mathbb{R}^N \\
v_j \geq 0 \quad &\text{ in } \mathcal{G}_j \\
v_j(0) \leq 1/j,
\end{cases}
\]
for every $j \in \mathbb{N}$, $v_j \to \overline{v}$ in $C^2_{loc}$ as $j \to +\infty$. By uniform convergence we have $\overline{v}(0) = 0$, $\Delta \overline u = \tfrac12\Phi'(\overline u)$ in $\RR^N$, and $\overline{v} \geq 0$ in $\overline{\mathcal{G}}$. Therefore, using \eqref{eq:1DSymIneqinGC} applied  to the function $\bar u$ and with $\lambda = \lambda_\ast$, we deduce
\begin{equation}\label{eq:1DSymovwIneq}
\overline{v} \geq 0 \quad \text{ in } \mathbb{R}^N.
\end{equation}
On the other hand,
\[
\Delta \overline{v} = \tfrac{1}{2}(\Phi'(\overline{u}) - \Phi'(\overline{u}^{\lambda_\ast})) \leq \|\Phi\|_{C^{1,1}(\RR)} \overline{v} \quad \text{ in } \mathbb{R}^N,
\]
and so, $\overline{v}(0) = 0$, \eqref{eq:1DSymovwIneq} and the strong maximum principle yield $\overline{v} = 0$ in $\mathbb{R}^N$.  Consequently, for every fixed $x \in \mathbb{R}^N$
\[
\overline{u}(x) = \overline{u}^{\lambda_\ast}(x) = \lim_{j \to +\infty} u^{\lambda_\ast}(x + x_j) = \lim_{j \to +\infty} u(x + x_j - e\lambda_\ast) = \lim_{j \to +\infty} u_j(x - e\lambda_\ast) = \overline{u}(x - e\lambda_\ast),
\]
that is, $\overline{u}$ is $\lambda_\ast$-periodic along the direction $e$.

Now, fix $\vartheta \in (\vartheta_1,\vartheta_2)$ and take $\tilde{x}_j \in \mathcal{G}$ such that $\tilde{x}_{j,N} = x_{j,N}$ and $u(\tilde{x}_j) = \vartheta$, for every $j \in \mathbb{N}$ and set $\hat{x}_j := \tilde{x}_j - x_j$. By \eqref{eq:Sym1DFBBoundsGraph}, \eqref{eq:1DSymNewC} and \eqref{eq:1DSymTube}, we have $|\hat{x}_j| \leq 2 \overline{C}$ and thus, up to passing to a subsequence, $\hat{x}_j \to \hat{x} \in \mathcal{G}$ as $j \to +\infty$, and
\begin{equation}\label{eq:1DSymLimSeqFinal}
\vartheta = \lim_{j \to +\infty} u(\tilde{x}_j) = \lim_{j \to +\infty} u(\hat{x}_j + x_j) = \lim_{j \to +\infty} u_j(\hat{x}_j) = \overline{u}(\hat{x}),
\end{equation}
by uniform convergence. We also have
\begin{equation}\label{eq:1DSYmLastInclLevelth}
\{\overline{u} = \vartheta \} \subset \{x_N \geq -\overline{C} ( 1 + \sqrt{|x'|})\},
\end{equation}
up to taking $\overline{C} > 0$ larger. To see this, we take $y \in \{\overline{u} = \vartheta\}$ and we notice that $y \in \{\vartheta_1 \leq u_j \leq \vartheta_2\}$ for large $j$'s or, equivalently, $y + x_j \in \{\vartheta_1 \leq u \leq \vartheta_2\} \subset \{x_N \geq G(x') - C\}$, in light of \eqref{eq:Sym1DFBBoundsGraph}. This, combined with the fact that $x_{j,N} \leq G(x_j') + \overline{C}$ (since $x_j \in \mathcal{G}$) and \eqref{eq:Sym1DLipConstG} give
\[
y_N \geq G(y' + x_j') - x_{j,N} - C \geq G(y' + x_j') - G(x_j') - \overline{C} - C \geq -c\sqrt{|y'|} - \overline{C} - C,
\]
which is \eqref{eq:1DSYmLastInclLevelth}, up to taking $c > 0$ large enough (depending on $C$ and $\overline{C}$).

To complete the contradiction argument, we notice that by \eqref{eq:1DSymLimSeqFinal} and the $\lambda_\ast$-periodicity of $\overline{u}$ (along the direction $e$), it must be $\hat{x} - n e\lambda_\ast \in \{\overline{u} = \vartheta\}$ for every $n \in \mathbb{Z}$, and thus, using \eqref{eq:1DSYmLastInclLevelth}, it follows $\hat{x} - n e\lambda_\ast \in \{x_N \geq -\overline{C} ( 1 + \sqrt{|x'|})\}$ for every $n \in \mathbb{Z}$. Using the definition of $e$ and passing to the limit as $n \to +\infty$, we find
\[
\begin{aligned}
0 &\leq (\hat{x} - ne\lambda_\ast)_N + \overline{C} (1 + \sqrt{|(\hat{x} - ne\lambda_\ast)'|}) \\
& =  \hat{x}_N - n\sigma\lambda_\ast + \overline{C} \big(1 + \sqrt{|\hat{x}' - ne'\lambda_\ast|} \big) \to -\infty,
\end{aligned}
\]
as $n \to +\infty$, a contradiction, and \eqref{eq:1DSymSlidindThGC} follows.

\smallskip

\emph{Step 7.} By \eqref{eq:1DSymSlidindTh}, we have $\partial_e u \geq 0$ in $\mathbb{R}^N$, independently of $\sigma > 0$ (cf. \eqref{eq:1DSymDefeulambda}) and so $\partial_{(e',0)} u \geq 0$ in $\mathbb{R}^N$, for every $e' \in \Sf^{N-1}\cap\{x_N = 0\}$. Since $\partial_{(-e',0)} u = -\partial_{(e',0)} u \leq 0$ in $\mathbb{R}^N$ and $e'$ is arbitrary, it must be $\partial_{(e',0)} u = 0$ in $\mathbb{R}^N$ for every $e'\in \Sf^{N-1}\cap\{x_N = 0\}$, that is $u$ is $1$D.
\end{proof}
%
%
%
%
\begin{proof}[Proof of Theorem \ref{thm:main}] Let $u: \RR^N \to \RR_+$ be a critical point of $\mathcal E$ in  $\RR^N$ satisfying \eqref{eq:1DSymAss10} and \eqref{eq:1DSymAss20} for some $R_k \to \infty$ and $\dd_k \to 0$. Setting $\vep_k := R^{-k}$ and scaling, we immediately see that $u_{\vep_k} := \vep_k u(\cdot/\vep_k)$ satisfies \eqref{eq:1DSymAss1} and \eqref{eq:1DSymAss2} for $k$ large and so $u(x) = v(x_N)$ for some $v:\RR\to\RR$ (up to a rotation), by Theorem \ref{thm:AsymFlatImplies1D}. 

On the other hand, by Lemma \ref{lem:1DSolution}, we know there are exactly three families of 1D solutions (cf. (i), (ii), (iii) of Lemma \ref{lem:1DSolution} with $\vep=1$). However, by \eqref{eq:1DSymAss1}, we have $u_{\vep_k} \to (x_N)_+$ locally uniformly, up to a translation and a rotation, and thus $v$ cannot be of class (ii) and (iii). The only possibility is that $v$ is of class (i). Recalling that $v(0) = \vth_1$ by construction, a direct integration of \eqref{eq:ODEEnergyCons} (with $A=1$) yields (cf. \eqref{eq:1DImplicitv})
\[
v^{-1}(z) = \int_{\vth_1}^z \frac{\rd \zeta}{\sqrt{\Phi(\zeta)}},
\]
for every $z \in \RR$, which is \eqref{wtioewhot}, up to a shift.
\end{proof}
%
%
%
%
%
\begin{proof}[Proof of Corollary \ref{cor:main}] Let $u: \RR^N\to \mathbb \RR_+$ be an entire local minimizer of  $\mathcal E$ in $\RR^N$. Up to shift, we may assume $u(0) = \vth_1$. If $\{R_j\}_{j \in \NN}$ is an arbitrary sequence satisfying $R_j \to +\infty$ as $j \to +\infty$ then, by Proposition \ref{prop:main}, there exist sequences $R_{j_\ell} \to +\infty$, $\dd_\ell \to 0$ and a 1-homogeneous nontrivial entire local minimizer $u_0$ of \eqref{eq:Energy1Phase} with $0 \in \partial\{u_0 > 0\}$, such that \eqref{eq:1DSymAss1XX} and \eqref{eq:1DSymAss2XX} hold true (with $k=j$). Consequently, since we know that
\[
u_0(x) = (\nu\cdot x)_+,
\]
for some $\nu \in \Sf^{N-1}$ (see \cite{CafJerKen2004:art,JerisonSavin2009:art}), we deduce that \eqref{eq:1DSymAss10} and \eqref{eq:1DSymAss20} are satisfied too, and thus $u$ satisfies \eqref{wtioewhot} by Theorem \ref{thm:main}.
\end{proof}
%





\end{document}